\pdfoutput=1 
\documentclass{amsart} 

\usepackage{tikz}
\usetikzlibrary{calc}
\usetikzlibrary{decorations.markings}
\usetikzlibrary {arrows.meta}
\usetikzlibrary {bending}
\usetikzlibrary{intersections}

\usepackage{amssymb}
\usepackage{amsmath}
\usepackage{url}
\usepackage{enumerate}
\usepackage{graphicx}

\usepackage{hyperref}
\hypersetup{
 colorlinks = true,
 linkcolor = blue,
 filecolor = blue,
 urlcolor = blue,
 citecolor = blue,
}
\hypersetup{unicode=true}

\usepackage{algorithm}
\usepackage{algpseudocode}

\theoremstyle{plain}
\newtheorem{theorem}{Theorem}[section]
\newtheorem{lemma}[theorem]{Lemma}
\newtheorem{proposition}[theorem]{Proposition}
\newtheorem{corollary}[theorem]{Corollary}

\theoremstyle{definition}
\newtheorem{definition}[theorem]{Definition}
\newtheorem{example}[theorem]{Example}

\theoremstyle{remark}
\newtheorem{remark}[theorem]{Remark}

\numberwithin{equation}{section}
\numberwithin{table}{section}
\numberwithin{figure}{section}
\renewcommand{\qed}{\hfill {$\Box$}}

\renewcommand{\AA}{\mathord{\mathbb A}}
\newcommand{\CC}{\mathord{\mathbb C}}
\newcommand{\PP}{\mathord{\mathbb  P}}
\newcommand{\QQ}{\mathord{\mathbb  Q}}
\newcommand{\RR}{\mathord{\mathbb R}}
\newcommand{\ZZ}{\mathord{\mathbb Z}}

\newcommand{\AAA}{\mathord{\mathcal A}}
\newcommand{\CCC}{\mathord{\mathcal C}}
\newcommand{\DDD}{\mathord{\mathcal D}}
\newcommand{\FFF}{\mathord{\mathcal F}}
\newcommand{\GGG}{\mathord{\mathcal G}}
\newcommand{\KKK}{\mathord{\mathcal K}}
\newcommand{\LLL}{\mathord{\mathcal L}}
\newcommand{\PPP}{\mathord{\mathcal P}}
\newcommand{\RRR}{\mathord{\mathcal R}}
\newcommand{\WWW}{\mathord{\mathcal W}}

\newcommand{\SSSS}{\mathord{\mathfrak S}}

\newcommand{\inj}{\hookrightarrow}
\newcommand{\surj}{\mathbin{\to \hskip -7pt \to}}

\newcommand{\set}[2]{\{\,{#1}\mid {#2} \,\}}
\newcommand{\bigset}[2]{\left\{\; {#1} \; \left\vert \; {#2} \;  \right.\right \}}

\newcommand{\angs}[1]{\langle {#1}  \rangle}

\newcommand{\tensor}{\otimes}

\newcommand{\inv}{\sp{-1}}
\newcommand{\dual}{\sp{\vee}}

\newcommand{\sprime}{\sp{\prime}}
\newcommand{\sprimeinv}{\sp{\prime-1}}
\newcommand{\spar}[1]{\sp{(#1)}}

\newcommand{\spprime}{\sp{\prime\prime}}
\newcommand{\spcirc}{\sp{\mathord{\circ}}}
\newcommand{\sptimes}{\sp{\times}}
\newcommand{\sperp}{\sp{\perp}}

\newcommand{\Aut}{\mathord{\mathrm{Aut}}}

\newcommand{\Sing}{\operatorname{\mathrm{Sing}}}

\newcommand{\OG}{\mathord{\mathrm{O}}}

\newcommand{\SO}{\mathord{\mathrm{SO}}}

\newcommand{\id}{\mathord{\mathrm{id}}}

\newcommand{\quand}{\quad{\textrm{and}}\quad}

\newcommand{\mystruth}[1]{\phantom{\hbox{\vrule height #1}}}
\newcommand{\mystrutd}[1]{\phantom{\hbox{\vrule depth #1}}}

\newcommand{\intf}[1]{\langle #1 \rangle}

\newcommand{\ordplus}{\mathord{+}}
\newcommand{\ordminus}{\mathord{-}}

\newcommand{\barX}{\overline{X}}
\newcommand{\barL}{\overline{L}}

\newcommand{\NS}{\mathord{\mathrm{NS}}}
\newcommand{\SX}{S_X}
\newcommand{\NX}{N_X}
\newcommand{\VX}{V_X}
\newcommand{\RX}{{\RRR}_X}
\newcommand{\PX}{{\PPP}_X}

\newcommand{\LMrats}{\LLL_{32}}
\newcommand{\Sep}{\mathord{\mathrm{Sep}}}
\newcommand{\RatsX}{\mathord{\mathrm{Rats}}(X)}
\newcommand{\theenr}{\varepsilon}
\newcommand{\Lts}{L_{26}} 
\newcommand{\Pts}{{\PPP}_{26}}
\newcommand{\Rts}{{\RRR}_{26}}
\newcommand{\RRRL}{\RRR_{L}}

\newcommand{\Concham}{\mathord{\mathbf{C}}}
\newcommand{\weyl}{\mathord{\mathbf{w}}}
\newcommand{\ampleL}{\mathord{\mathbf{a}_L}}

\newcommand{\vect}[1]{\mathord{\mathbf{#1}}}

\newcommand{\spbr}[1]{\sp{[#1]}}

\newcommand{\MW}{\mathord{\mathrm{MW}}}

\newcommand{\PPPL}{\PPP_{L}}
\newcommand{\PPPM}{\PPP_{M}}
\newcommand{\PPPf}{\PPP_{f}}
\newcommand{\PPPC}{\PPP_{\CCC}}
\newcommand{\Faces}[1]{\FFF^{#1}}

\newcommand{\CCCC}{\mathfrak{C}}
\newcommand{\tilV}{\widetilde{V}}
\newcommand{\tilVC}{\tilV_{\CCC}}
\newcommand{\VC}{V_{\CCC}}

\newcommand{\Adj}{\mathord{\mathrm{Adj}}}
\newcommand{\Rel}{\RRR}
\newcommand{\Reltriv}{\Rel_{\mathord{\mathrm{triv}}}}
\newcommand{\Relface}{\Rel_{\mathord{\mathrm{face}}}}

\newcommand{\Gens}{\Gamma}
\newcommand{\Genszero}{\Gens_0}
\newcommand{\GensA}{\Gens_{\AAA}}
\newcommand{\Gensst}{\Gens^{*}}

\newcommand{\empseq}{\varepsilon}

\newcommand{\word}[1]{\boldsymbol{#1}}

\newcommand{\mult}{m}

\newcommand{\typeI}{\mathrm{I}}
\newcommand{\typeII}{\mathrm{II}}

\newcommand{\boldgreek}[1]{\boldsymbol{#1}}
\newcommand{\blambda}{\boldsymbol{\lambda}}
\newcommand{\brho}{\boldgreek{\rho}}
\newcommand{\gh}{\mathrm{gh}}
\newcommand{\WK}{\WWW_{\KKK}}

\newcommand{\ADE}{\mathord{\mathrm{ADE}}}
\newcommand{\tilA}{\widetilde{A}}
\newcommand{\tilD}{\widetilde{D}}
\newcommand{\tilE}{\widetilde{E}}

\begin{document}

\title[Ap\'ery--Fermi $K3$ surface]
{The automorphism group of an Ap\'ery--Fermi $K3$ surface}

\author{Ichiro Shimada}
\address{Department of Mathematics,
Graduate School of Science,
Hiroshima University,
1-3-1 Kagamiyama,
Higashi-Hiroshima,
739-8526 JAPAN}
\email{ichiro-shimada@hiroshima-u.ac.jp}
\thanks{Supported by JSPS KAKENHI Grant Number~20H01798, 20K20879, 20H00112, and 23H00081}
\thanks{The author declares that there is no conflict of interest. 
The data supporting the findings of this study are publicly available at~\url{https://doi.org/10.5281/zenodo.14842583}.
}
\begin{abstract}
An Ap\'ery--Fermi $K3$ surface is a complex $K3$ surface of Picard number $19$
that is birational to 
a general member of
a certain one-dimensional family of affine surfaces 
related to the Fermi surface in solid-state physics.
This $K3$ surface is also linked to 
a recurrence relation that appears in the famous proof of the irrationality of $\zeta(3)$ by Ap\'ery.
\par
We compute the automorphism group $\Aut(X)$ 
of the Ap\'ery--Fermi $K3$ surface $X$ using Borcherds' method.
We describe $\Aut(X)$ in terms of  generators and  relations.
Moreover, we determine the action of $\Aut(X)$ 
on the set of $\ADE$-configurations of smooth rational curves on $X$ for some $\ADE$-types.
In particular, we show that $\Aut(X)$ acts transitively on the set of smooth rational curves,
and that it partitions the set of
 pairs of disjoint smooth rational curves into two orbits.
\end{abstract}
\keywords{K3 surface, automorphism group, lattice}
\makeatletter
\@namedef{subjclassname@2020}{\textup{2020} Mathematics Subject Classification}
\makeatother
\subjclass[2020]{14J28}
\maketitle
\section{Introduction}
\subsection{Main results}\label{subsec:mainresults}
We consider a pencil of complex affine surfaces $X_s\spcirc\subset\AA^3$ 
defined by the equation 
\begin{equation}\label{eq:Fermi}
\xi_1+\frac{1}{\xi_1}+\xi_2+\frac{1}{\xi_2}+\xi_3+\frac{1}{\xi_3}=s,
\end{equation}
where $\xi_1, \xi_2, \xi_3$ are coordinates of $\AA^3$, and $s\in \CC$ is a parameter.
When $s$ is very general, the surface $X_s\spcirc$ is birational to a projective $K3$ surface $X_s$
whose N\'eron--Severi lattice is isomorphic to
\begin{equation}\label{eq:M6}
M_6:=U\oplus E_8(-1)\oplus E_8(-1)\oplus \angs{-12},
\end{equation}
where $U$ is the hyperbolic plane,
$E_8(-1)$ is the negative-definite root lattice of type $E_8$,
and $\angs{-12}$ is a rank-one lattice 
generated by a vector with square-norm $-12$.
We call the $K3$ surface $X_s$ with $s$ sufficiently general an \emph{Ap\'ery--Fermi $K3$ surface}.
For simplicity, we assume that the parameter $s$ is very general throughout this work.
\par
In this paper, we study
the automorphism group $\Aut(X_s)$ of the Ap\'ery--Fermi $K3$ surface $X_s$
by using Borcherds' method.
We provide a finite set of generators of $\Aut(X_s)$,
and describe the action of $\Aut(X_s)$ on the nef-and-big cone of $X_s$ explicitly.
We prove that the nef-and-big cone of $X_s$ is tessellated by copies of a polyhedral cone with $80$ walls,
that the action of $\Aut(X_s)$ preserves this tessellation,
and that $\Aut(X_s)$ acts transitively on the set of tiles of this tessellation 
 with the stabilizer subgroup $\Aut(X_s, D_0)$ of a tile $D_0$
being isomorphic to a dihedral group of order $16$.
Using this tessellation, we obtain the following result in Section~\ref{sec:NSlattice}:
\begin{theorem}\label{thm:aut}
The automorphism group $\Aut(X_s)$ is generated by a finite subgroup $\Aut(X_s, D_0)$
of order $16$, and eight extra automorphisms.
\end{theorem}
In Section~\ref{sec:geometric},
we provide an explicit geometric description of these generators
in terms of Mordell--Weil groups of Jacobian fibrations, using 
the algorithm for computing the Mordell--Weil action on the N\'eron--Severi lattice 
described in our previous paper~\cite{Shimada2024}.
We also analyze the faces of $D_0$, and 
using the list of codimension-$2$ faces,
we describe $\Aut(X_s)$ in terms of generators and relations in~Section~\ref{sec:relations}.
\par
Next we study the action of $\Aut(X_s)$ on the set of $\ADE$-configurations of smooth rational curves on $X_s$.
Let $\tau$ be an ordinary $\ADE$-type,
and let $\mu$ be the number of nodes in the corresponding Dynkin diagram.
We denote by $\CCCC(\tau)$ the set of all non-ordered sets $\CCC=\{C_1, \dots, C_{\mu}\}$
of smooth rational curves on $X_s$ such that the dual graph of $\CCC$ is the Dynkin diagram of type $\tau$.
For example, $\CCCC(A_1)$ is the set of smooth rational curves on $X_s$,
 $\CCCC(2A_1)$ is the set of non-ordered pairs of disjoint smooth rational curves,
 whereas $\CCCC(A_2)$ is the set of non-ordered pairs of smooth rational curves intersecting at one point transversely.
\begin{theorem}\label{thm:rats}
For $\mu\leq 4$, the numbers of the orbits of the action of $\Aut(X_s)$ on the set $\CCCC(\tau)$
are given in Table~\ref{table:sizesCCC}.
\end{theorem}
%
%
\begin{table}
\[
\begin{array}{clc}
\mu & \tau & |\CCCC(\tau)/\Aut(X_s)|\mystrutd{5pt}\\
\hline 
1 & A_{1} & 1 \mystruth{11pt} \\ 
2 & 2A_{1} & 2 \\ 
2 & A_{2} & 1 \\ 
3 & 3A_{1} & 2 \\ 
3 & A_{1}+A_{2} & 1 \\ 
3 & A_{3} & 3 \\ 
\end{array}
\qquad \qquad 
\begin{array}{clc}
\mu & \tau & |\CCCC(\tau)/\Aut(X_s)| \mystrutd{5pt}\\
\hline 
4 & 4A_{1} & 2 \mystruth{11pt} \\ 
4 & 2A_{1}+A_{2} & 2 \\ 
4 & A_{1}+A_{3} & 9 \\ 
4 & 2A_{2} & 2 \\ 
4 & A_{4} & 1 \\ 
4 & D_{4} & 2 
\end{array}
%
\]
\caption{Sizes of $\CCCC(\tau)/\Aut(X_s)$}\label{table:sizesCCC}
\end{table}
\begin{corollary}\label{cor:ratstransitive}
The group $\Aut(X_s)$ acts on the set of smooth rational curves on $X_s$ transitively.
\qed
\end{corollary}
In fact,
it is theoretically possible to obtain 
the same result for $\ADE$-types $\tau$ with higher Milnor numbers $\mu$.
However, we stopped the computation at $\mu=4$ because 
the computation becomes too expensive for $\mu\ge 5$.
See~Section~\ref{subsec:orbitCCCC}.
\par
Our result is obtained by using Borcherds' method. 
This method was introduced by Borcherds~\cite{Bor1}, ~\cite{Bor2}, and 
its first geometric application was given by Kondo~\cite{Kondo1998}.
In~\cite{Shimada2015} and~\cite{ Shimada2024}, we presented 
tools and techniques for implementing Borcherds' method in a computer.
\par
Borcherds' method has been applied to many $K3$ and Enriques surfaces. 
For the Ap\'ery--Fermi $K3$ surface, the tasks of computing  
a finite  generating set of $\Aut(X_s)$ and  obtaining geometric realizations of these generators
were carried out  smoothly by 
the  tools that had been established previously. 
A new tool introduced in this paper is 
an algorithm that enumerates the faces of 
higher-codimensions of
the nef-and-big cone modulo $\Aut(X_s)$, which is described in Section~\ref{sec:faces}.
By this tool,  we obtain Theorem~\ref{thm:rats}  above, 
and defining relations  of $\Aut(X_s)$ as presented in Section~\ref{sec:relations}.
\par
With the advances in machine computing power, 
the geometric information that can be obtained by  Borcherds' method is rapidly expanding. 
A project is also underway to implement this method in a new  computer algebra system~\cite{OSCAR_K3Surfaces}.
A secondary aim of this paper is to highlight the power and utility of Borcherds'  method
through its application 
 to a significant $K3$ surface. 
\subsection{Previous studies of the Ap\'ery--Fermi \texorpdfstring{$K3$}{K3} surface}\label{subsec:history}
The Ap\'ery--Fermi $K3$ surface is an important $K3$ surface that has been extensively studied
by many authors.
Here, we provide a brief review of previous works related to the Ap\'ery--Fermi $K3$ surface.
\par
In 1984, Beukers and Peters~\cite{BeukersPeters1984} 
constructed a one-dimensional family of $K3$ surfaces
whose Picard-Fuchs equation is the differential equation 
arising in Ap\'ery's famous proof~\cite{Apery1979} of irrationality of $\zeta(3)$.
In 1986, Peters~\cite{Peters1986} determined the N\'eron--Severi lattice and
the transcendental lattice of the general member of this family, and 
in 1989, Peters and Stienstra~\cite{PetersStienstra1989}
showed that 
the general member is an Ap\'ery--Fermi $K3$ surface
defined above. 
The equation~\eqref{eq:Fermi}
has its origin in the solid-state physics,
where it is related to the Fermi surface of electrons moving in a crystal.
See~Introduction of~\cite{PetersStienstra1989}  and the reference therein for the background in physics.
\par
In 1996, Dolgachev~\cite{DolgachevLatticePolarizedK3} introduced the notion of \emph{lattice polarized $K3$ surfaces}. 
The Ap\'ery--Fermi $K3$ surface is an $M_6$-lattice polarized $K3$ surface, 
where $M_6$ is defined in~\eqref{eq:M6}, 
and 
Dolgachev~\cite{DolgachevLatticePolarizedK3} determined, among other things, 
the coarse moduli space of Ap\'ery--Fermi $K3$ surfaces.
In 2004, Hosono et al.~\cite{HLOY2004} used Ap\'ery--Fermi $K3$ surfaces 
in the study of 
the autoequivalences of derived category of its Fourier--Mukai partner,
a $K3$ surface with Picard number $1$ and of degree $12$.
In the paper~\cite{DardanelliGeemen2007} by Dardanelli and van Geemen,
the Ap\'ery--Fermi $K3$ surfaces appear as the Hessians of certain cubic surfaces
(see Proposition~5.7 of~\cite{DardanelliGeemen2007}).
\par
On the other hand, there exists a  rigid Calabi--Yau $3$-fold birational to 
a smooth affine $3$-fold defined by 
\begin{equation}\label{eq:rigidCalabiYau3fold}
\xi_1+\frac{1}{\xi_1}+\xi_2+\frac{1}{\xi_2}+\xi_3+\frac{1}{\xi_3}+\xi_4+\frac{1}{\xi_4}=0.
\end{equation}
Its modularity was studied by
van Geemen and Nygaard~\cite{vanGeemenNygaard1995},
Verrill~\cite{Verrill2000}, and 
Ahlgren and Ono~\cite{AhlgrenOno2000}.
\par
In 2015, Mukai and Ohashi~\cite{MukaiOhashi2015} found another birational model 
of the Ap\'ery--Fermi $K3$ surface:
the symmetric quartic surface $Y_t\subset \PP^3$ defined by
\begin{equation}\label{eq:SymQ}
(x_1 x_2+x_1 x_3+x_1 x_4+ x_2 x_3 + x_2 x_4+ x_3 x_4)^2= t\, x_1 x_2 x_3 x_4,
\end{equation}
where $(x_1:x_2:x_3:x_4)$ are homogeneous coordinates of $\PP^3$ and $t\in \CC$ is a parameter.
Mukai and Ohashi~\cite{MukaiOhashi2015} exhibited an Enriques involution $\theenr$ of $Y_t$,
and described the automorphism group of the Enriques surface birational to $Y_t/\angs{\theenr}$.
\begin{remark}\label{rem:informed}
To the best of our knowledge, 
the fact that the Ap\'ery--Fermi $K3$ surface $X_s$ is
birational to the quartic surface $Y_t$ for some $t=t(s)$
has not yet appeared in the literature.
We were informed of this fact through personal communication with the authors of~\cite{MukaiOhashi2015}.
See Proposition~\ref{prop:isom} for the proof of this fact.
\end{remark}
In 2020, Bertin and Lecacheux~\cite{BertinLecacheux2020} determined all Jacobian fibrations 
of the Ap\'ery--Fermi $K3$ surface by
using Kneser--Nishiyama method.
Some special members of the pencil~\eqref{eq:Fermi} 
with Picard number $20$ 
have also been studied,
for example, 
in~\cite{Shimada2015},~\cite{Shimada2016}, 
and in Bertin and Lecacheux~\cite{BertinLecacheux2013},~\cite{BertinLecacheux2022}.
In~\cite{DinoDuco2019}, Festi and van Straten provided an account on the relation 
between the Ap\'ery--Fermi $K3$ surfaces 
and quantum electrodynamics,
highlighting the importance of studying this $K3$ surface.
\subsection{Plan of this paper}\label{subsec:plan}
In Section~\ref{sec:twomodels}, we review the result of
Peters and Stienstra~~\cite{PetersStienstra1989},
and present $32$ smooth rational curves on the $K3$ surface $X_s$
whose classes generate the N\'eron--Severi lattice of $X_s$.
We also compare $X_s$ with the quartic surface $Y_t$ of Mukai and Ohashi~\cite{MukaiOhashi2015},
and prove that $X_s$ is birational to $Y_t$ for a suitable choice of $t$ (see Remark~\ref{rem:informed}).
In Section~\ref{sec:NSlattice},
we execute Borcherds' method,
and obtain a set of generators of $\Aut(X_s)$ \emph{lattice-theoretically}, 
thereby proving Theorem~\ref{thm:aut}.
We also describe the finite polytope $D_0$ with $80$ walls.
Although Corollary~\ref{cor:ratstransitive} is a part of Theorem~\ref{thm:rats},
it can already be proved at this stage, and hence  we provide its proof in this section.
In Section~\ref{sec:geometric},
we give geometric realization to each of the generators of $\Aut(X_s)$ given in Theorem~\ref{thm:aut}.
In Section~\ref{sec:faces},
we calculate the set of faces of the polytope $D_0$,
and prove Theorem~\ref{thm:rats} in Section~\ref{subsec:orbitCCCC}.
In Section~\ref{sec:relations},
we explain how to describe $\Aut(X_s)$ in terms of generators and relations 
using the codimension-$2$ faces of $D_0$.
\par
Detailed computational data are available from~\cite{AFcompdata}.
For our computation, we used {\tt GAP}~\cite{GAP}.
\par
\medskip
{\bf Acknowledgements.}
We are grateful to  Professor Shigeyuki Kondo,  Professor Shigeru Mukai, and Professor Hisanori Ohashi for 
providing information
about the quartic surface $Y_t$.
We also thank 
Professor Takuya Yamauchi for enlightening us 
about the rigid Calabi--Yau $3$-fold~\eqref{eq:rigidCalabiYau3fold}.
Finally, we thank the referees for their many helpful comments and suggestions. 
\section{Two projective models of an Ap\'ery--Fermi \texorpdfstring{$K3$}{K3}  surface}\label{sec:twomodels}
In Sections~\ref{subsec:32PS} and~\ref{subsec:32MO}, 
we review results by Peters and Stienstra~\cite{PetersStienstra1989}, 
and by Mukai and Ohashi~\cite{MukaiOhashi2015}, respectively.
The main purpose of this section is 
to label certain $32$ smooth rational curves on an Ap\'ery--Fermi $K3$ surface. 
Since we employ the labeling of~\cite{PetersStienstra1989} and use it throughout this paper, 
the results in Section~\ref{subsec:32MO} are not  used for the computation of 
the automorphism group.
\subsection{The Fermi surface model}\label{subsec:32PS}
We review the result of Peters and Stienstra~\cite{PetersStienstra1989}.
Let $X_s\spcirc$ be the affine surface in $\AA^3$ defined by the equation~\eqref{eq:Fermi},
and let $X_s$ be the $K3$ surface containing $X_s\spcirc$ as a Zariski open subset.
Recall that we have assumed that the parameter $s\in \CC$ is very general.
We present $32$ smooth rational curves on $X_s$ 
whose classes generate the N\'eron--Severi lattice $\NS(X_s)$ of $X_s$.
\par
The $K3$ surface $X_s$ is isomorphic to 
a smooth surface in $\PP^6\times \PP^1\times \PP^1\times \PP^1$
defined by the equation (4) in~\cite{PetersStienstra1989}.
Considering 
 the projection onto the first factor $\PP^6$, we see that $X_s$ is birational to the surface $\barX_s$ in $\PP^6$ 
defined by
\begin{equation}\label{eq:wuuuvvv}
\begin{aligned}
&u_1+u_2+u_3+v_1+v_2+v_3=s w, & \\
&u_1 v_1-w^2=u_2 v_2-w^2=u_3 v_3-w^2=0,&
\end{aligned}
\end{equation}
where $(w:u_1:u_2:u_3:v_1:v_2:v_3)$ is a homogeneous coordinate system of $\PP^6$ such that 
 we have $\xi_i=u_i/w=w/v_i$
on $X_s\spcirc$.
We denote by $H_{\infty}$ the hyperplane of $\PP^6$ defined by $w=0$.
For $i=1, 2, 3$, let $\gamma_i\in \{0, +, -\}$ denote the condition
\[ 
\begin{cases}
u_i=0\;\textrm{and}\; v_i=0, &\textrm{if $\gamma_i=0$,}\\
u_i\ne 0\;\textrm{and}\;  v_i=0, &\textrm{if $\gamma_i=+$,}\\
u_i= 0\;\textrm{and}\;  v_i\ne 0, &\textrm{if $\gamma_i=-$.}
\end{cases}
\]
If one of $\gamma_1, \gamma_2, \gamma_3$ is $0$ and the other two are not, 
then the conditions $\gamma_1$, $\gamma_2$, $\gamma_3$ with $w=0$ determine 
a single point $p_{\gamma_1\gamma_2\gamma_3}$ on $\bar{X}_s\cap H_{\infty}$. 
For example, we have 
\[
p_{+-0}=(0: 1:0:0:0:-1:0).
\]
The points $p_{\gamma_1\gamma_2\gamma_3}$ are ordinary nodes of $\barX_s$,
and 
the singular locus $\Sing \barX_s$ 
of $\barX_s$ consists of these $12$ points.
Let $L_{\gamma_1\gamma_2\gamma_3}$ denote the exceptional $(-2)$-curve
of the minimal desingularization $X_s\to \barX_s$ over $p_{\gamma_1\gamma_2\gamma_3}$.
If none of $\gamma_1, \gamma_2, \gamma_3$ is $0$,
then the conditions $\gamma_1$ and $\gamma_2$ and $\gamma_3$ with $w=0$ define
a line $\barL_{\gamma_1\gamma_2\gamma_3}$ on $\barX_s\cap H_{\infty}$.
 For example, we have 
\[
\barL_{+-+}=\set{(0: \lambda_1:0: \lambda_3: 0: \lambda_2:0)}{ \lambda_1+ \lambda_2+ \lambda_3=0}.
\]
Let $L_{\gamma_1\gamma_2\gamma_3}\subset X_s$ denote the strict transform of $\barL_{\gamma_1\gamma_2\gamma_3}$
in $X_s$.
Thus, we obtain $12+8$ smooth rational curves $L_{\gamma_1 \gamma_2 \gamma_3}$ on $X_s$.
\par
Let $\sigma, \sigma\inv\in \CC$ be the roots of the equation $\xi+1/\xi=s$.
For $k\in \{1,2,3\}$ and $\alpha, \beta \in \{+, -\}$,
we define the curve $M_{k\alpha\beta}$ on $X_s$ as follows.
Let $i, j$ be the indexes such that $\{i,j,k\}=\{1,2 ,3\}$.
The curve defined by 
\[
\xi_i+1/\xi_i+\xi_j+1/\xi_j=0
\]
 in $\AA^2$ with coordinates $(\xi_i, \xi_j)$ is a union of two rational curves $\xi_i+\xi_j=0$ and $\xi_i\xi_j+1=0$.
Let $M_{k\alpha\beta}\spcirc$
be the curve on 
$X_s\spcirc \subset \AA^3$ defined by 
\[
\begin{cases}
\xi_k= \sigma &\textrm{if $\alpha=+$}, \\
\xi_k= \sigma\inv &\textrm{if $\alpha=-$}
\end{cases},
\qquad\textrm{and}\qquad 
\begin{cases}
\xi_i+\xi_j= 0 &\textrm{if $\beta=+$}, \\
\xi_i\xi_j+1=0 &\textrm{if $\beta=-$},
\end{cases}
\]
and let $M_{k\alpha\beta}\subset X_s$ be the strict transform of the closure of $M_{k\alpha\beta}\spcirc$. 
Thus, we obtain $12$ smooth rational curves $M_{k\alpha\beta}$ on $X_s$.
\par
We now confirm 
the following results proved in Section 7 of~\cite{Peters1986} and~\cite{PetersStienstra1989} by direct computation.
\begin{lemma}\label{lem:intnumbsPS}
{\rm (1)}
The intersection numbers of these $20+12$ smooth rational curves $L_{\gamma_1 \gamma_2 \gamma_3}$ and $M_{k\alpha\beta}$
are as follows.
\begin{enumerate}[{\rm (i)}]
\item The dual graph of the curves $L_{\gamma_1 \gamma_2 \gamma_3}$ is 
shown in Figure~\ref{fig:Lcube}, which we refer to as the \emph{$L$-cube}.
\item The curves $M_{k\alpha\beta}$ intersect as follows:
\[
\intf{M_{k\alpha\beta}, M_{k\sprime\alpha\sprime\beta\sprime}}
=\begin{cases}
-2 & \textrm{if $k=k\sprime$, $\alpha=\alpha\sprime$, $\beta=\beta\sprime$}, \\
2 & \textrm{if $k=k\sprime$, $\alpha=\alpha\sprime$, $\beta\ne\beta\sprime$}, \\
0 & \textrm{if $k=k\sprime$, $\alpha\ne\alpha\sprime$, $\beta=\beta\sprime$}, \\
0 & \textrm{if $k=k\sprime$, $\alpha\ne\alpha\sprime$, $\beta\ne\beta\sprime$}, \\
1 & \textrm{if $k\ne k\sprime$, $\alpha=\alpha\sprime$, $\beta=\beta\sprime$}, \\
0 & \textrm{if $k\ne k\sprime$, $\alpha=\alpha\sprime$, $\beta\ne\beta\sprime$}, \\
0 & \textrm{if $k\ne k\sprime$, $\alpha\ne\alpha\sprime$, $\beta=\beta\sprime$}, \\
1 & \textrm{if $k\ne k\sprime$, $\alpha\ne\alpha\sprime$, $\beta\ne\beta\sprime$}. 
\end{cases}
\]
\item 
We have 
\[
\intf{L_{\gamma_1 \gamma_2 \gamma_3}, M_{k\alpha\beta}}
=\begin{cases}
1 & \textrm{if $\gamma_k=0$ and $\beta=\gamma_i\gamma_j$, where $\{i,j,k\}=\{1,2,3\}$,} \\
0 &\textrm{otherwise.}
\end{cases}
\]
\end{enumerate}
\par
{\rm (2)}
The classes of these $32$ smooth rational curves 
span the N\'eron--Severi lattice $\NS(X_s)$ of $X_s$,
which is of rank $19$ and with discriminant $-12$.
\par
{\rm (3)}
The lattice $\NS(X_s)$ is isomorphic to the lattice $M_6$ defined by~\eqref{eq:M6}.
\qed

\end{lemma}
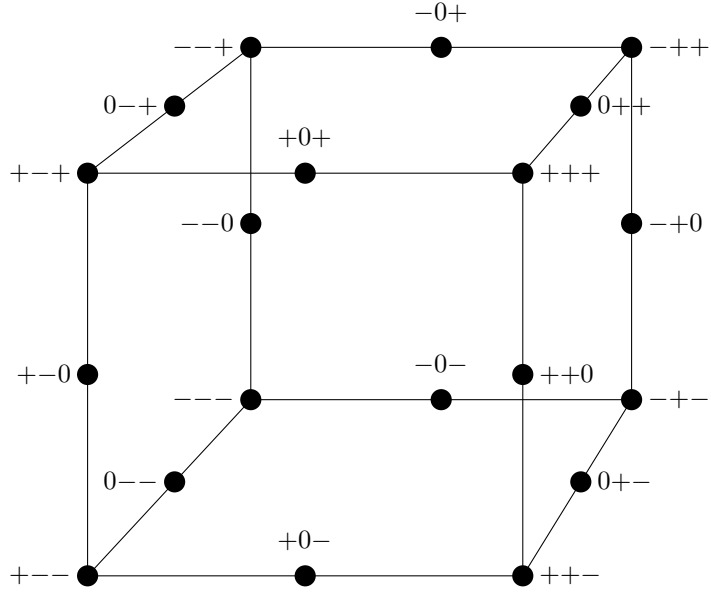
\begin{figure}
\begin{tikzpicture}[x=2.7cm, y=2.5cm]
%
\coordinate (mmm) at (-0.625, -0.5625);
\coordinate (mmz) at (-0.625, 0.375);
\coordinate (mmp) at (-0.625, 1.3125);
\coordinate (mzm) at (0.3125, -0.5625);
\coordinate (mzp) at (0.3125, 1.3125);
\coordinate (mpm) at (1.25, -0.5625);
\coordinate (mpz) at (1.25, 0.375);
\coordinate (mpp) at (1.25, 1.3125);
\coordinate (zmm) at (-1., -1.);
\coordinate (zmp) at (-1., 1.);
\coordinate (zpm) at (1., -1.);
\coordinate (zpp) at (1., 1.);
\coordinate (pmm) at (-1.4285714285714286, -1.5);
\coordinate (pmz) at (-1.4285714285714286, -0.42857142857142855);
\coordinate (pmp) at (-1.4285714285714286, 0.6428571428571429);
\coordinate (pzm) at (-0.35714285714285715, -1.5);
\coordinate (pzp) at (-0.35714285714285715, 0.6428571428571429);
\coordinate (ppm) at (0.7142857142857143, -1.5);
\coordinate (ppz) at (0.7142857142857143, -0.42857142857142855);
\coordinate (ppp) at (0.7142857142857143, 0.6428571428571429);
%
\fill (mmm) circle (4pt);
\fill (mmz) circle (4pt);
\fill (mmp) circle (4pt);
\fill (mzm) circle (4pt);
\fill (mzp) circle (4pt);
\fill (mpm) circle (4pt);
\fill (mpz) circle (4pt);
\fill (mpp) circle (4pt);
\fill (zmm) circle (4pt);
\fill (zmp) circle (4pt);
\fill (zpm) circle (4pt);
\fill (zpp) circle (4pt);
\fill (pmm) circle (4pt);
\fill (pmz) circle (4pt);
\fill (pmp) circle (4pt);
\fill (pzm) circle (4pt);
\fill (pzp) circle (4pt);
\fill (ppm) circle (4pt);
\fill (ppz) circle (4pt);
\fill (ppp) circle (4pt);
\draw (mmz) -- (mmm);
\draw (mmp) -- (mmz);
\draw (mzm) -- (mmm);
\draw (mzp) -- (mmp);
\draw (mpm) -- (mzm);
\draw (mpz) -- (mpm);
\draw (mpp) -- (mzp);
\draw (mpp) -- (mpz);
\draw (zmm) -- (mmm);
\draw (zmp) -- (mmp);
\draw (zpm) -- (mpm);
\draw (zpp) -- (mpp);
\draw (pmm) -- (zmm);
\draw (pmz) -- (pmm);
\draw (pmp) -- (zmp);
\draw (pmp) -- (pmz);
\draw (pzm) -- (pmm);
\draw (pzp) -- (pmp);
\draw (ppm) -- (zpm);
\draw (ppm) -- (pzm);
\draw (ppz) -- (ppm);
\draw (ppp) -- (zpp);
\draw (ppp) -- (pzp);
\draw (ppp) -- (ppz);
%
\node [left] at (mmm) {$\ordminus \ordminus \ordminus $\;};
\node [left] at (mmz) {$\ordminus \ordminus 0 $\;};
\node [left] at (mmp) {$\ordminus \ordminus \ordplus $\;};
\node [above] at (mzm) {$\ordminus 0 \ordminus $\mystrutd{7pt}};
\node [above] at (mzp) {$\ordminus 0 \ordplus $\mystrutd{7pt}};
\node [right] at (mpm) {\;$\ordminus \ordplus \ordminus $};
\node [right] at (mpz) {\;$\ordminus \ordplus 0 $};
\node [right] at (mpp) {\;$\ordminus \ordplus \ordplus $};
\node [left] at (zmm) {$0 \ordminus \ordminus $\;};
\node [left] at (zmp) {$0 \ordminus \ordplus $\;};
\node [right] at (zpm) {\;$0 \ordplus \ordminus $};
\node [right] at (zpp) {\;$0 \ordplus \ordplus $};
\node [left] at (pmm) {$\ordplus \ordminus \ordminus $\;};
\node [left] at (pmz) {$\ordplus \ordminus 0 $\;};
\node [left] at (pmp) {$\ordplus \ordminus \ordplus $\;};
\node [above] at (pzm) {$\ordplus 0 \ordminus $\mystrutd{7pt}};
\node [above] at (pzp) {$\ordplus 0 \ordplus $\mystrutd{7pt}};
\node [right] at (ppm) {\;$\ordplus \ordplus \ordminus $};
\node [right] at (ppz) {\;$\ordplus \ordplus 0 $};
\node [right] at (ppp) {\;$\ordplus \ordplus \ordplus $};
\end{tikzpicture}
\caption{Dual graph of the curves $L_{\gamma_1 \gamma_2 \gamma_3}$ ($L$-cube)}\label{fig:Lcube}
\end{figure}
To prove the assertion (3), we use the following Jacobian fibration of $X_s$.
The configuration of the $32$ smooth rational curves described in Lemma~\ref{lem:intnumbsPS} contains a sub-configuration 
shown in Figure~\ref{fig:2E8}.
Hence $X_s$ has an elliptic fibration with a section $M_{2\ordplus\ordplus}$ and two singular fibers of type $\mathrm{II}^*$.
Consequently, $\NS(X_s)$ contains a rank-$18$ sublattice isomorphic to $U\oplus E_8(-1)\oplus E_8(-1)$.
Since this sublattice is unimodular, it must be a direct summand of $\NS(X_s)$.
Comparing the discriminant, we see that $\NS(X_s)$ is isomorphic to $M_6$.
\begin{figure}
\begin{tikzpicture}[x=1.4cm, y=1cm]
{\small
%
\coordinate (n1) at (2 ,2);
\coordinate (n2) at (0 ,1);
\coordinate (n3) at (1 ,1);
\coordinate (n4) at (2 ,1);
\coordinate (n5) at (3 ,1);
\coordinate (n6) at (4 ,1);
\coordinate (n7) at (5 ,1);
\coordinate (n8) at (6 ,1);
\coordinate (n9) at (7 ,1);
\coordinate (n10) at (2 ,-2);
\coordinate (n11) at (0 ,-1);
\coordinate (n12) at (1 ,-1);
\coordinate (n13) at (2 ,-1);
\coordinate (n14) at (3 ,-1);
\coordinate (n15) at (4 ,-1);
\coordinate (n16) at (5 ,-1);
\coordinate (n17) at (6 ,-1);
\coordinate (n18) at (7 ,-1);
\coordinate (n19) at (8 ,0);
\draw (n1) -- (n4);
\draw (n2) -- (n3);
\draw (n3) -- (n4);
\draw (n4) -- (n5);
\draw (n5) -- (n6);
\draw (n6) -- (n7);
\draw (n7) -- (n8);
\draw (n8) -- (n9);
\draw (n9) -- (n19);
\draw (n10) -- (n13);
\draw (n11) -- (n12);
\draw (n12) -- (n13);
\draw (n13) -- (n14);
\draw (n14) -- (n15);
\draw (n15) -- (n16);
\draw (n16) -- (n17);
\draw (n17) -- (n18);
\draw (n18) -- (n19);
\filldraw [fill=white] (n1) circle [radius=3pt];
\filldraw [fill=white] (n2) circle [radius=3pt];
\filldraw [fill=white] (n3) circle [radius=3pt];
\filldraw [fill=white] (n4) circle [radius=3pt];
\filldraw [fill=white] (n5) circle [radius=3pt];
\filldraw [fill=white] (n6) circle [radius=3pt];
\filldraw [fill=white] (n7) circle [radius=3pt];
\filldraw [fill=white] (n8) circle [radius=3pt];
\filldraw [fill=white] (n9) circle [radius=3pt];
\filldraw [fill=white] (n10) circle [radius=3pt];
\filldraw [fill=white] (n11) circle [radius=3pt];
\filldraw [fill=white] (n12) circle [radius=3pt];
\filldraw [fill=white] (n13) circle [radius=3pt];
\filldraw [fill=white] (n14) circle [radius=3pt];
\filldraw [fill=white] (n15) circle [radius=3pt];
\filldraw [fill=white] (n16) circle [radius=3pt];
\filldraw [fill=white] (n17) circle [radius=3pt];
\filldraw [fill=white] (n18) circle [radius=3pt];
\filldraw [fill=white] (n19) circle [radius=3pt];
\node at (n1) [left] {$L\sb{\ordminus \ordminus \ordminus }$\;};
\node at (n2) [below] {$L\sb{\ordminus 0 \ordplus }$\mystruth{10pt}};
\node at (n3) [below] {$L\sb{\ordminus \ordminus \ordplus }$\mystruth{10pt}};
\node at (n4) [below] {$L\sb{\ordminus \ordminus 0 }$\mystruth{10pt}};
\node at (n5) [above] {$M\sb{3\ordminus \ordplus }$\mystrutd{6pt}};
\node at (n6) [above] {$L\sb{\ordplus \ordplus 0 }$\mystrutd{6pt}};
\node at (n7) [above] {$L\sb{\ordplus \ordplus \ordplus }$\mystrutd{6pt}};
\node at (n8) [above] {$L\sb{0 \ordplus \ordplus }$\mystrutd{6pt}};
\node at (n9) [above] {$M\sb{1\ordplus \ordplus }$\mystrutd{6pt}};
\node at (n10) [left] {$L\sb{\ordplus \ordminus \ordplus }$\;};
\node at (n11) [above] {$L\sb{\ordplus 0 \ordminus }$\mystrutd{6pt}};
\node at (n12) [above] {$L\sb{\ordplus \ordminus \ordminus }$\mystrutd{6pt}};
\node at (n13) [above] {$L\sb{\ordplus \ordminus 0 }$\mystrutd{6pt}};
\node at (n14) [below] {$M\sb{3\ordplus \ordminus }$\mystruth{10pt}};
\node at (n15) [below] {$L\sb{\ordminus \ordplus 0 }$\mystruth{10pt}};
\node at (n16) [below] {$L\sb{\ordminus \ordplus \ordminus }$\mystruth{10pt}};
\node at (n17) [below] {$L\sb{0 \ordplus \ordminus }$\mystruth{10pt}};
\node at (n18) [below] {$M\sb{1\ordminus \ordminus }$\mystruth{10pt}};
\node at (n19) [right] {\;$M\sb{2\ordplus \ordplus }$};
}
\end{tikzpicture}
\caption{Sub-configuration containing $2\,\mathrm{II}^*$}\label{fig:2E8}
\end{figure}
\par
Thus, $X_s$ can be regarded as an $M_6$-lattice polarized $K3$ surface in the sense of 
Dolgachev~\cite{DolgachevLatticePolarizedK3}.
According to~\cite{DolgachevLatticePolarizedK3},
the isomorphism classes of 
$M_6$-lattice polarized $K3$ surfaces
are parameterized by an irreducible curve,
and our surface $X_s$ corresponds to a geometric generic point of this curve.
\subsection{The Mukai--Ohashi quartic}\label{subsec:32MO}
We review the paper~\cite{MukaiOhashi2015} by Mukai and Ohashi.
The results in this section are not directly related to the main line of argument of the paper.
\par
Let $Y_t$ be the quartic surface in $\PP^3$ defined by the quartic polynomial~\eqref{eq:SymQ},
where the parameter $t$ is assumed to be very general.
For $i\in \{1,\dots, 4\}$,
let $H_i\subset \PP^3$ denote the plane defined by $x_i=0$,
and let $p_i$ denote the point such that $\{p_i\}=H_j\cap H_k\cap H_{l}$,
where $\{i,j,k,l\}=\{1, \dots, 4\}$.
Then the singular locus $\Sing Y_t$ of $Y_t$ consists of four points $p_1, \dots, p_4$,
each of which is a rational double point of type $D_4$.
Let $(\PP^3)\sprime\to \PP^3$ be the blowing up at the points $p_1, \dots, p_4$.
We denote by 
$Y_t\sprime\subset (\PP^3)\sprime$ the strict transform of $Y_t$, 
and by $E_i\subset (\PP^3)\sprime$ 
the exceptional divisor over $p_i$.
We have homogeneous coordinates $(u_{ij}: u_{ik}:u_{il})$ of $E_i\cong \PP^2$,
where $\{i, j, k, l\}=\{1, \dots, 4\}$, 
such that the strict transform of the plane in $\PP^3$ 
defined by $a_j x_j+a_k x_k+a_l x_l=0$ intersects $E_i$ along the line
$a_ju_{ij}+a_k u_{ik}+a_l u_{il}=0$.
We consider the line 
\[
\Lambda_i\;\;\colon\;\; u_{ij}+u_{ik}+u_{il}=0
\]
on $E_i$.
Then the scheme-theoretic intersection of $Y_t\sprime$ and $E_i$ is the double line $2\Lambda_i$.
For $\nu \in \{j, k, l\}$, 
let $q_{i\nu}$ be the intersection point in $E_i\cong \PP^2$ of the line $\Lambda_i$ and the line defined by $u_{i\nu}=0$.
Then 
the singular points of $Y_t\sprime$ located on $E_i$ are precisely the three points 
$q_{i\nu}$,
forming a total of $3\times 4$ ordinary nodes of $Y_t\sprime$.
Let $(\PP^3)\spprime\to (\PP^3)\sprime$ be the blowing up at these nodes $q_{i\nu}$,
and let $Y_t\spprime\subset (\PP^3)\spprime$ be the strict transform of $Y_t\sprime$.
Then $Y_t\spprime$ is smooth.
Let $P_i\subset Y_t\spprime$ be the strict transform of $\Lambda_i$,
and let $Q_{i\nu}\subset Y_t\spprime$
be the exceptional curve over the ordinary node $q_{i\nu}\in Y_t\sprime$.
Then, for each $i$,  
the smooth rational curves $P_i$ and $Q_{i\nu}$ ($\nu\in \{j,k,l\}$) form a dual graph
isomorphic to the Dynkin diagram of type $D_4$ with $P_i$ being the central node.
\par
The scheme-theoretic intersection of $Y_t$ and $H_{\lambda}=\{x_{\lambda}=0\}$ is a double conic $2T_{\lambda}$,
where $T_{\lambda}$ is a smooth conic on $H_{\lambda}$.
Let $T_{\lambda}\sprime\subset Y_t\sprime$ and $T_{\lambda}\spprime\subset Y_t\spprime$ be the strict transforms of 
$T_{\lambda}$ in $Y_t\sprime$ and in $Y_t\spprime$, respectively.
Suppose that $i\ne \lambda$.
Then $T_{\lambda}\sprime$ intersects $E_i$ at the point $q_{i\lambda}$, and 
the curve $T_{\lambda}\spprime$ intersects $Q_{i\lambda}$,
but is disjoint from the other three component $P_{i}$ and $Q_{i\mu}$, $Q_{i\nu}$
of the $D_4$-configuration over $p_i$,
where $\{i, \lambda, \mu, \nu\}=\{1, \dots, 4\}$.
\par
Let $\tau$ and $1/\tau$ be 
 the two roots of the equation
$(u-1)^2-tu=0$
in variable $u$.
Let $\mu, \nu\in \{1, \dots,4\}$ be distinct indexes,
and let $H_{\mu\nu}\subset\PP^3$ be the plane in $\PP^3$ defined by $x_{\mu}+x_{\nu}=0$.
We put $\{i, j\}=\{1, \dots, 4\}\setminus \{\mu, \nu\}$.
Then $H_{\mu\nu}\cap Y_t$ is a union of two conics
\[
C_{\mu\nu, \rho} \quad \colon\quad x_{\mu} x_{\nu}+\rho\, x_i x_j=x_{\mu}+x_{\nu}=0, 
\]
where $\rho\in \{\tau, 1/\tau\}$.
Let $C_{\mu\nu, \rho}\sprime\subset Y_t\sprime$ and
$C_{\mu\nu, \rho}\spprime\subset Y_t\spprime$ be
 the strict transforms of $C_{\mu\nu, \rho}$ in $Y_t\sprime$ and in $Y_t\spprime$, respectively.
 Note that, since the strict transform $H_{\mu\nu}\sprime\subset (\PP^3)\sprime$ of $H_{\mu\nu}$
 intersects the exceptional surface $E_i$ along the line $u_{i\mu}+u_{i\nu}=0$,
 the curves $C_{\mu\nu, \tau}\sprime$ and $C_{\mu\nu,1/\tau} \sprime$ pass through $q_{ij}$,
because $\Lambda_i$ is defined by $u_{ij}+u_{i\mu}+u_{i\nu}=0$.
Thus, the curves $C_{\mu\nu, \tau}\spprime$ and $ C_{\mu\nu, 1/\tau} \spprime$ intersect $Q_{ij}$.
 \par
We can now establish the following result by direct computation.
\begin{lemma}\label{lem:intnumbsMO}
The intersection numbers of the $32$ smooth rational curves 
$P_i$, $Q_{ij}$, $T_{\nu}\spprime$, and $C_{\mu\nu,\rho}\spprime$ on $Y_t\spprime$
are as follows.
\begin{enumerate}[{\rm (i)}]
\item The dual graph of the curves $P_i$, $Q_{ij}$, $T\spprime_{\nu}$ is 
shown in Figure~\ref{fig:PQTcube},
where the thick edges indicate the four $D_4$-configurations over the singular points of $Y_t$.
\item The intersection numbers of the curves $C_{\mu\nu, \rho}\spprime$,
where $\mu, \nu\in \{1, \dots, 4\}$ with $\mu\ne \nu$ and $\rho\in \{\tau, 1/\tau\}$,
are as follows.
\par
If $\{\mu, \nu\}=\{\mu\sprime, \nu\sprime\}$, then 
\[
\intf{C_{\mu\nu, \rho}\spprime, C_{\mu\sprime\nu\sprime, \rho\sprime}\spprime}
=\begin{cases}
-2 & \textrm{if $\rho=\rho\sprime$}, \\
0 & \textrm{if $\rho\ne\rho\sprime$}.
\end{cases}
\]
If $\{\mu, \nu\}\cap \{\mu\sprime, \nu\sprime\}$ consists of a single element, 
then 
\[
\intf{C_{\mu\nu, \rho}\spprime, C_{\mu\sprime\nu\sprime , \rho\sprime}\spprime}
=\begin{cases}
1& \textrm{if $\rho=\rho\sprime$}, \\
0 & \textrm{if $\rho\ne\rho\sprime$}.
\end{cases}
\]
If $\{\mu, \nu\}\cap \{\mu\sprime, \nu\sprime\}=\emptyset$, 
then 
 \[
\intf{C_{\mu\nu, \rho}\spprime, C_{\mu\sprime\nu\sprime , \rho\sprime}\spprime}
=\begin{cases}
0& \textrm{if $\rho= \rho\sprime$}, \\
2 & \textrm{if $\rho\ne\rho\sprime$}.
\end{cases}
\]
\item 
The curve $C_{\mu\nu, \rho}\spprime$ is disjoint from $P_i$, $T\spprime_{\nu}$, and
we have 
\[
\intf{C_{\mu\nu, \rho}\spprime, Q_{ij}}
=\begin{cases}
1 & \textrm{if $\{\mu, \nu, i,j\}=\{1,2,3, 4\}$,} \\
0 &\textrm{otherwise.}
\end{cases}
\]
\end{enumerate}
\qed
\end{lemma}
\begin{figure}
\begin{tikzpicture}[x=2.7cm, y=2.5cm]
%
\coordinate (mmm) at (-0.625, -0.5625);
\coordinate (mmz) at (-0.625, 0.375);
\coordinate (mmp) at (-0.625, 1.3125);
\coordinate (mzm) at (0.3125, -0.5625);
\coordinate (mzp) at (0.3125, 1.3125);
\coordinate (mpm) at (1.25, -0.5625);
\coordinate (mpz) at (1.25, 0.375);
\coordinate (mpp) at (1.25, 1.3125);
\coordinate (zmm) at (-1., -1.);
\coordinate (zmp) at (-1., 1.);
\coordinate (zpm) at (1., -1.);
\coordinate (zpp) at (1., 1.);
\coordinate (pmm) at (-1.4285714285714286, -1.5);
\coordinate (pmz) at (-1.4285714285714286, -0.42857142857142855);
\coordinate (pmp) at (-1.4285714285714286, 0.6428571428571429);
\coordinate (pzm) at (-0.35714285714285715, -1.5);
\coordinate (pzp) at (-0.35714285714285715, 0.6428571428571429);
\coordinate (ppm) at (0.7142857142857143, -1.5);
\coordinate (ppz) at (0.7142857142857143, -0.42857142857142855);
\coordinate (ppp) at (0.7142857142857143, 0.6428571428571429);
%
\draw [very thick] (mmz) -- (mmm);
\draw (mmp) -- (mmz);
\draw [very thick] (mzm) -- (mmm);
\draw (mzp) -- (mmp);
\draw (mpm) -- (mzm);
\draw (mpz) -- (mpm);
\draw [very thick] (mpp) -- (mzp);
\draw [very thick] (mpp) -- (mpz);
\draw [very thick] (zmm) -- (mmm);
\draw (zmp) -- (mmp);
\draw (zpm) -- (mpm);
\draw [very thick] (zpp) -- (mpp);
\draw (pmm) -- (zmm);
\draw (pmz) -- (pmm);
\draw [very thick] (pmp) -- (zmp);
\draw [very thick] (pmp) -- (pmz);
\draw (pzm) -- (pmm);
\draw [very thick] (pzp) -- (pmp);
\draw [very thick] (ppm) -- (zpm);
\draw [very thick] (ppm) -- (pzm);
\draw [very thick] (ppz) -- (ppm);
\draw (ppp) -- (zpp);
\draw (ppp) -- (pzp);
\draw (ppp) -- (ppz);
\fill (mmm) circle (4pt);
\fill (mmz) circle (4pt);
\node at (mmp) [fill=white, draw, circle] {};
\fill (mzm) circle (4pt);
\fill (mzp) circle (4pt);
\node at (mpm) [fill=white, draw, circle] {};
\fill (mpz) circle (4pt);
\fill (mpp) circle (4pt);
\fill (zmm) circle (4pt);
\fill (zmp) circle (4pt);
\fill (zpm) circle (4pt);
\fill (zpp) circle (4pt);
\node at (pmm) [fill=white, draw, circle] {};
\fill (pmz) circle (4pt);
\fill (pmp) circle (4pt);
\fill (pzm) circle (4pt);
\fill (pzp) circle (4pt);
\fill (ppm) circle (4pt);
\fill (ppz) circle (4pt);
\node at (ppp) [fill=white, draw, circle] {};
%
\node [left] at (mmm) {$P_{3}$\;};
\node [left] at (mmz) {$Q_{34}$\;};
\node [left] at (mmp) {$T\spprime_{4}$\;};
\node [above] at (mzm) {$Q_{31}$\mystrutd{7pt}};
\node [above] at (mzp) {$Q_{24}$\mystrutd{7pt}};
\node [right] at (mpm) {\;$T\spprime_{1}$};
\node [right] at (mpz) {\;$Q_{21}$};
\node [right] at (mpp) {\;$P_{2}$};
\node [left] at (zmm) {$Q_{32}$\;};
\node [left] at (zmp) {$Q_{14}$\;\;};
\node [right] at (zpm) {\;$Q_{41}$};
\node [left] at (zpp) {$Q_{23}$\;\;};
\node [left] at (pmm) {$T\spprime_{2}$\;};
\node [left] at (pmz) {$Q_{12}$\;};
\node [left] at (pmp) {$P_{1}$\;};
\node [above] at (pzm) {$Q_{42}$\mystrutd{7pt}};
\node [above] at (pzp) {$Q_{13}$\mystrutd{7pt}};
\node [right] at (ppm) {\;$P_{4}$};
\node [right] at (ppz) {\;$Q_{43}$};
\node [right] at (ppp) {\;$T\spprime_{3}$};
\end{tikzpicture}
\caption{Dual graph of the curves $P_i$, $Q_{ij}$, $T\spprime_{\nu}$}\label{fig:PQTcube}
\end{figure}
As noted in Remark~\ref{rem:informed},
the following result was known to the authors of~\cite{MukaiOhashi2015}.
\begin{proposition}\label{prop:isom}
There exists a parameter $t(s)\in \CC$ 
such that the $K3$ surfaces $X_s$ and $Y_{t(s)}\spprime$ are isomorphic.
\end{proposition}
\begin{proof}
The $32$ smooth rational curves 
in Lemma~\ref{lem:intnumbsPS} and those 
in Lemma~\ref{lem:intnumbsMO}
have the same configuration.
Indeed, 
a bijection between these two sets of $32$ curves preserving their intersection numbers
can be established
 by comparing the cubes in Figures~\ref{fig:Lcube} and~\ref{fig:PQTcube}
for the curves $L_{\gamma_1 \gamma_2 \gamma_3}$, and 
using Table~\ref{table:bij12curvess}
for $M_{k\alpha\beta}$.
\par
Since the isomorphism class of the $K3$ surface $Y_t\spprime$ varies as $t$ changes,
and $t$ is assumed to be very general,
we conclude that $\NS(Y_t\spprime)$ is of rank $19$.
By Lemma~\ref{lem:intnumbsPS}~(2) and the bijection above, 
$\NS(Y_t\spprime)$ contains a sublattice isomorphic to $\NS(X_s)$ with finite index.
Since $\NS(X_s)$
admits no non-trivial even overlattice,
as was noted in the proof of Proposition 7.1.1 of~\cite{Peters1986},
we see that $\NS(X_s)\cong \NS(Y_t)$.
By Corollary~7.1.3 of~\cite{Peters1986}, the transcendental lattice of $X_s$ is isomorphic to that of $Y_t\spprime$.
Applying the Torelli theorem for $K3$ surfaces,
we conclude that there exists a suitable choice of $t(s)$ for which 
 $X_s\cong Y_{t(s)}\spprime$.
\begin{table}
\begin{eqnarray*}
&&\begin{array}{c|cccccccc}
k\alpha\beta & 1\ordminus \ordminus & 1\ordminus \ordplus & 1\ordplus \ordminus & 1\ordplus \ordplus & 2\ordminus \ordminus & 2\ordminus \ordplus & 2\ordplus \ordminus & 2\ordplus \ordplus \\ 
\hline 
\mu\nu, \rho & 23, \tau & 14, 1/\tau & 23, 1/\tau & 14, \tau & 13, \tau & 24, 1/\tau & 13, 1/\tau & 24, \tau 
\end{array}
\\
&&\begin{array}{c|cccc}
k\alpha\beta & 3\ordminus \ordminus & 3\ordminus \ordplus & 3\ordplus \ordminus & 3\ordplus \ordplus \\ 
\hline 
\mu\nu, \rho & 34, \tau & 12, 1/\tau & 34, 1/\tau & 12, \tau 
\end{array}
\end{eqnarray*}
\vskip 10pt
\caption{Bijection between $M_{k\alpha\beta}$ and $C\spprime_{\mu\nu, \rho}$}
\label{table:bij12curvess}
\end{table}
\end{proof} 
\section{N\'eron--Severi lattice and automorphism group}\label{sec:NSlattice}
From now on, we omit the parameter $s$ in $X_s$,
and simply denote the Ap\'ery--Fermi $K3$ surface by $X$.
We also write $S_X$ for the N\'eron--Severi lattice $\NS(X)$ of $X$.
We make the orthogonal group 
$\OG(\SX)$ act on $\SX$ from the right.
\par
In this section, we execute Borcherds' method.
In Section~\ref{subsec:chambers}, we fix terminology and notation about \emph{chambers}.
In Section~\ref{subsec:SX}, we describe  $\SX$ explicitly,  and
in Section~\ref{subsec:anample}, we present an ample class $a_{32}$.
In Section~\ref{subsec:EmbeddingAutInO}, we embed $\Aut(X)$ into $\OG(\SX)$,
so that  out  geometric problem is transformed into a lattice-theoretic problem.
In Section~\ref{subsec:AutXLMrats},
we describe a finite subgroup  $\Aut(X, \LMrats)\subset \Aut(X)$ of order $48$.
With these preparations, 
in  Section~\ref{subsec:Borcherds}, we perform Borcherds' method,
and obtain a set of generators of $\Aut(X)$ in Proposition~\ref{prop:fromfact3}.
In Section~\ref{subsec:ProofOfCor}, we prove Corollary~\ref{cor:ratstransitive}.

\subsection{Chambers and their faces}\label{subsec:chambers}
We fix terminology and notation about lattices and hyperbolic spaces.
Let $L$ be an even lattice of signature $(1, l-1)$ with $l\ge 2$.
A \emph{positive cone} of $L$ is one of the two connected components of
the space 
\[
\set{x\in L\tensor\RR}{\intf{x,x}>0}.
\]
We fix a positive cone $\PPPL$, and 
define the autochronous subgroup of $\OG(L)$ as 
\[
\OG(L, \PPPL):=\set{g\in \OG(L)}{\PPPL^g=\PPPL}.
\]
We also define
\[
\RRRL:=\set{r\in L}{\intf{r, r}=-2}.
\]
For $v\in L\tensor\RR$ with $\intf{v,v}<0$,
let $(v)\sperp$ denote the real hyperplane in $\PPPL$ defined by $\intf{x, v}=0$.
The \emph{Weyl group} $W(L)$ is the subgroup of $\OG(L, \PPPL)$
generated by reflections $x\mapsto x+\intf{x, r}r$ into the mirrors $(r)\sperp$ defined by vectors $r\in \RRRL$.
A \emph{standard fundamental domain of the action of the Weyl group $W(L)$ on $\PPPL$}
is the closure in $\PPPL$ of a connected component of the space 
\[
\PPPL\;\;\setminus\;\; \bigcup_{r\in \RRRL} (r)\sperp.
\]
Now, let $M$ be a primitive sublattice of $L$ with signature $(1, m-1)$ with $m\ge 2$,
and let $\PPPM$ be the positive cone $(M\tensor\RR)\cap \PPPL$ of $M$.
\begin{definition}\label{def:LMchamber}
An \emph{$L/M$-chamber} is a closed subset $D$ of $\PPPM$ 
such that 
\begin{enumerate}[(i)]
\item $D$ has the form $\PPPM\cap D_L$,
where $D_L$ is a standard fundamental domain of the action of $W(L)$ on $\PPPL$,
and 
\item
$D$ contains a nonempty open subset of $\PPPM$.
\end{enumerate}
\end{definition}
Each $L/M$-chamber is defined in $\PPPM$ by locally finite linear inequalities 
\begin{equation}\label{eq:ineqs}
\intf{x, v_i}\ge 0, \quad\textrm{where}\quad v_i\in M\tensor \QQ.
\end{equation}
\begin{remark}\label{rem:LL}
According to this terminology,
the lengthy phrase
``standard fundamental domain of the action of $W(L)$ on $\PPPL$"
can be shortened to ``$L/L$-chamber".
Note that $W(L)$ acts on the set of $L/L$-chambers simply transitively.
\end{remark}
\begin{remark}\label{rem:notisom}
In general, $L/M$-chambers are \emph{not} congruent to each other.
\end{remark}
\begin{remark}\label{rem:tess}
Each $M/M$-chamber is a union of $L/M$-chambers,
meaning that each $M/M$-chamber is \emph{tessellated} by $L/M$-chambers.
(We use the term ``tessellation" even when the constituent tiles are not congruent to each other.)
\par
More generally,
if $M\sprime$ is a primitive sublattice of $M$,
then every $M/M\sprime$-chamber is tessellated by $L/M\sprime$-chambers.
\end{remark}
For $v\in L\cap \PPPL$, we put
\[
[v]\sperp:=\set{r\in L}{\intf{v, r}=0}.
\]
Then a point $v\in L\cap \PPPL$ is an interior point of an $L/L$-chamber if and only if $[v]\sperp\cap \RRRL=\emptyset$.
Suppose that $v$ is an interior point of an $L/L$-chamber $N$, and 
let $v\sprime$ be another vector of $ L\cap \PPPL$.
Then $v\sprime$ belongs to the same $L/L$-chamber $N$ 
as $v$ if and only if the set 
\[
\Sep(v, v\sprime):=\set{r\in \RRRL}{\;\intf{v, r}>0, \;\;\intf{v\sprime, r}<0\;}
\]
of \emph{separating $(-2)$-vectors} is empty.
The set $\Sep(v, v\sprime)$ can be computed using an algorithm given in Section 3.3 of~\cite{Shimada2014}.
\par
Let $D$ be an $L/M$-chamber.
A closed subset $f$ of $D$ is called a \emph{face of codimension $\mu$ of $D$}
if there exists a linear subspace $\PPPf$ of $\PPPM$ of codimension $\mu$ such that
\begin{enumerate}[(i)]
\item $f=\PPPf\cap D$,
\item $\PPPf$ is disjoint from the interior of $D$, and
\item $f$ contains a nonempty open subset of $\PPPf$.
\end{enumerate}
The linear subspace $\PPPf$ is called the \emph{supporting linear subspace} of the face $f$.
A face of codimension $1$ is called a \emph{wall}.
%
\par
Let $w$ be a wall of $D$.
We say that a vector $v$ of the dual lattice $M\dual$ is a \emph{primitive defining vector}
of the wall $w=\PPP_w\cap D$ of $D$ if 
\begin{enumerate}[(i)]
\item $\PPP_w=(v)\sperp$, 
\item $v$ is primitive in $M\dual$,
and 
\item $\intf{v, x}>0$ for an (and hence every) interior point $x$ of $D$.
\end{enumerate}
Each wall of $D$ has a unique primitive defining vector.
\par
For a face $f$ of $D$,
let $\DDD(f)$ be the set of $L/M$-chambers that contain $f$.
If $w$ is a wall of $D$, 
there exists a unique $L/M$-chamber $D\sprime\ne D$ such that $\DDD(w)=\{D, D\sprime\}$.
We call $D\sprime$ the $L/M$-chamber \emph{adjacent to $D$ across the wall $w$}.
\subsection{The lattice \texorpdfstring{$\SX$}{SX}}\label{subsec:SX}
We study the N\'eron--Severi lattice $\SX$,
which is an even lattice of signature $(1, 18)$.
Let $\PX\subset \SX\tensor\RR$ be 
the positive cone of $\SX$ containing an ample class of $X$.
We define the \emph{nef-and-big cone} of $X$ by 
\[
\NX:=\set{x\in \PX}{\textrm{$\intf{x, C}\ge 0$ for all curves $C$ on $X$}}.
\]
It is well known that $\NX$ is an $\SX/\SX$-chamber.
We then define 
\[
\RX:=\RRR_{\SX}=\set{r\in \SX}{\intf{r,r}=-2}, 
\]
and denote by 
$\RatsX\subset \RX$
the set of classes of smooth rational curves on $X$.
Then $\NX$ is determined by 
\[
\NX=\set{x\in \PX}{\textrm{$\intf{x, C}\ge 0$ for any $C\in \RatsX$}}.
\]
\begin{remark}
To simplify notation,
we do not distinguish a smooth rational curve on $X$ and its class in $\SX$.
For example, we often write $C\in \SX$ for $C\in\RatsX$.
\end{remark}
We introduce a basis of the N\'eron--Severi lattice $\SX$.
First, 
we fix a basis for the lattice $M_6$ defined in~\eqref{eq:M6}.
Let $u_1, u_2$ be the basis of the hyperbolic plane $U$ with the Gram matrix 
{\small
\[
 \left(\begin{array}{cc} 
0 & 1 \\ 
1 & 0 
\end{array}\right).
\]}%
For $\nu=1, 2$,
let $e_1\spar{\nu}, \dots, e_8\spar{\nu}$ be the $(-2)$-vectors 
in the two copies of $E_8 (-1)$ 
that form the dual graph illustrated in Figure~\ref{fig:E8}.
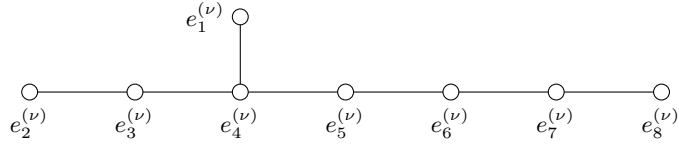
\begin{figure}
\begin{tikzpicture}[x=1.4cm, y=1cm]
{\small
%
\coordinate (n1) at (2 ,2);
\coordinate (n2) at (0 ,1);
\coordinate (n3) at (1 ,1);
\coordinate (n4) at (2 ,1);
\coordinate (n5) at (3 ,1);
\coordinate (n6) at (4 ,1);
\coordinate (n7) at (5 ,1);
\coordinate (n8) at (6 ,1);
\draw (n1) -- (n4);
\draw (n2) -- (n3);
\draw (n3) -- (n4);
\draw (n4) -- (n5);
\draw (n5) -- (n6);
\draw (n6) -- (n7);
\draw (n7) -- (n8);
\filldraw [fill=white] (n1) circle [radius=3pt];
\filldraw [fill=white] (n2) circle [radius=3pt];
\filldraw [fill=white] (n3) circle [radius=3pt];
\filldraw [fill=white] (n4) circle [radius=3pt];
\filldraw [fill=white] (n5) circle [radius=3pt];
\filldraw [fill=white] (n6) circle [radius=3pt];
\filldraw [fill=white] (n7) circle [radius=3pt];
\filldraw [fill=white] (n8) circle [radius=3pt];
\node at (n1) [left] {$e\spar{\nu}_1$\;};
\node at (n2) [below] {$e\spar{\nu}_2$\mystruth{12pt}};
\node at (n3) [below] {$e\spar{\nu}_3$\mystruth{12pt}};
\node at (n4) [below] {$e\spar{\nu}_4$\mystruth{12pt}};
\node at (n5) [below] {$e\spar{\nu}_5$\mystruth{12pt}};
\node at (n6) [below] {$e\spar{\nu}_6$\mystruth{12pt}};
\node at (n7) [below] {$e\spar{\nu}_7$\mystruth{12pt}};
\node at (n8) [below] {$e\spar{\nu}_8$\mystruth{12pt}};
}
\end{tikzpicture}
\caption{Dual graph of $e_1\spar{\nu}, \dots, e_8\spar{\nu}$}\label{fig:E8}
\end{figure}
Let $\angs{-12}=\ZZ\, v_{12}$ be the rank-one lattice
generated by a vector $v_{12}$ satisfying $\intf{v_{12}, v_{12}}=-12$.
Then the $19$ vectors 
\begin{equation}\label{eq:basisSX}
u_1, \;\; u_2, \;\; e_1\spar{1}, \;\; \dots, \;\; e_8\spar{1}, \;\; e_1\spar{2}, \;\; \dots, \;\; e_8\spar{2}, \;\; v_{12}
\end{equation}
form a basis of $M_6$.
We write vectors of $M_6$ as row vectors of length $19$
with respect to this basis.
Next, we choose an isometry $M_6\cong \SX$ as given in Table~\ref{table:isomM6SX},
and express vectors of $\SX$ using the same row vector representation.
\begin{table}
{\scriptsize
\[
\begin{array}{lcl}
%
%
L_{---} &\colon & [ 0, 0, 1, 0, 0, 0, 0, 0, 0, 0, 0, 0, 0, 0, 0, 0, 0, 0, 0 ]\\ 
L_{--0} &\colon & [ 0, 0, 0, 0, 0, 1, 0, 0, 0, 0, 0, 0, 0, 0, 0, 0, 0, 0, 0 ]\\ 
L_{--+} &\colon & [ 0, 0, 0, 0, 1, 0, 0, 0, 0, 0, 0, 0, 0, 0, 0, 0, 0, 0, 0 ]\\ 
L_{-0-} &\colon & [ 4, 3, -8, -5, -10, -15, -12, -9, -6, -3, -6, -4, -8, -12, -10, -8, -6, -3, -1 ]\\ 
L_{-0+} &\colon & [ 0, 0, 0, 1, 0, 0, 0, 0, 0, 0, 0, 0, 0, 0, 0, 0, 0, 0, 0 ]\\ 
L_{-+-} &\colon & [ 0, 0, 0, 0, 0, 0, 0, 0, 0, 0, 0, 0, 0, 0, 0, 0, 1, 0, 0 ]\\ 
L_{-+0} &\colon & [ 0, 0, 0, 0, 0, 0, 0, 0, 0, 0, 0, 0, 0, 0, 0, 1, 0, 0, 0 ]\\ 
L_{-++} &\colon & [ 4, 4, -8, -6, -11, -16, -13, -10, -7, -4, -9, -6, -12, -18, -15, -12, -8, -4, -1 ]\\ 
L_{0--} &\colon & [ 4, 4, -8, -5, -10, -15, -12, -9, -6, -3, -10, -7, -14, -20, -16, -12, -8, -4, -1 ]\\ 
L_{0-+} &\colon & [ 4, 4, -10, -7, -14, -20, -16, -12, -8, -4, -8, -5, -10, -15, -12, -9, -6, -3, -1 ]\\ 
L_{0+-} &\colon & [ 0, 0, 0, 0, 0, 0, 0, 0, 0, 0, 0, 0, 0, 0, 0, 0, 0, 1, 0 ]\\ 
L_{0++} &\colon & [ 0, 0, 0, 0, 0, 0, 0, 0, 0, 1, 0, 0, 0, 0, 0, 0, 0, 0, 0 ]\\ 
L_{+--} &\colon & [ 0, 0, 0, 0, 0, 0, 0, 0, 0, 0, 0, 0, 1, 0, 0, 0, 0, 0, 0 ]\\ 
L_{+-0} &\colon & [ 0, 0, 0, 0, 0, 0, 0, 0, 0, 0, 0, 0, 0, 1, 0, 0, 0, 0, 0 ]\\ 
L_{+-+} &\colon & [ 0, 0, 0, 0, 0, 0, 0, 0, 0, 0, 1, 0, 0, 0, 0, 0, 0, 0, 0 ]\\ 
L_{+0-} &\colon & [ 0, 0, 0, 0, 0, 0, 0, 0, 0, 0, 0, 1, 0, 0, 0, 0, 0, 0, 0 ]\\ 
L_{+0+} &\colon & [ 4, 3, -6, -4, -8, -12, -10, -8, -6, -3, -8, -5, -10, -15, -12, -9, -6, -3, -1 ]\\ 
L_{++-} &\colon & [ 4, 4, -9, -6, -12, -18, -15, -12, -8, -4, -8, -6, -11, -16, -13, -10, -7, -4, -1 ]\\ 
L_{++0} &\colon & [ 0, 0, 0, 0, 0, 0, 0, 1, 0, 0, 0, 0, 0, 0, 0, 0, 0, 0, 0 ]\\ 
L_{+++} &\colon & [ 0, 0, 0, 0, 0, 0, 0, 0, 1, 0, 0, 0, 0, 0, 0, 0, 0, 0, 0 ]\\ 
M_{1--} &\colon & [ 1, 0, 0, 0, 0, 0, 0, 0, 0, 0, -3, -2, -4, -6, -5, -4, -3, -2, 0 ]\\ 
M_{1-+} &\colon & [ 7, 7, -15, -10, -20, -30, -25, -19, -13, -7, -12, -8, -16, -24, -20, -15, -10, -5, -2 ]\\ 
M_{1+-} &\colon & [ 7, 7, -12, -8, -16, -24, -20, -15, -10, -5, -15, -10, -20, -30, -25, -19, -13, -7, -2 ]\\ 
M_{1++} &\colon & [ 1, 0, -3, -2, -4, -6, -5, -4, -3, -2, 0, 0, 0, 0, 0, 0, 0, 0, 0 ]\\ 
M_{2--} &\colon & [ 3, 3, -5, -4, -7, -10, -8, -6, -4, -2, -5, -4, -7, -10, -8, -6, -4, -2, -1 ]\\ 
M_{2-+} &\colon & [ 5, 5, -12, -8, -16, -24, -20, -15, -10, -5, -12, -8, -16, -24, -20, -15, -10, -5, -1 ]\\ 
M_{2+-} &\colon & [ 9, 7, -17, -12, -23, -34, -28, -21, -14, -7, -17, -12, -23, -34, -28, -21, -14, -7, -2 ]\\ 
M_{2++} &\colon & [ -1, 1, 0, 0, 0, 0, 0, 0, 0, 0, 0, 0, 0, 0, 0, 0, 0, 0, 0 ]\\ 
M_{3--} &\colon & [ 12, 11, -24, -16, -32, -48, -40, -30, -20, -10, -24, -16, -32, -48, -39, -30, -20, -10, -3 ]\\ 
M_{3-+} &\colon & [ 0, 0, 0, 0, 0, 0, 1, 0, 0, 0, 0, 0, 0, 0, 0, 0, 0, 0, 0 ]\\ 
M_{3+-} &\colon & [ 0, 0, 0, 0, 0, 0, 0, 0, 0, 0, 0, 0, 0, 0, 1, 0, 0, 0, 0 ]\\ 
M_{3++} &\colon & [ 12, 11, -24, -16, -32, -48, -39, -30, -20, -10, -24, -16, -32, -48, -40, -30, -20, -10, -3 ]
\end{array}
\]
}
\vskip .5cm
\caption{Isometry between $M_6$ and $ \SX$}\label{table:isomM6SX}
\end{table}
\begin{remark}
Under this isomorphism $M_6\cong \SX$,
the vector $u_1\in M_6$ corresponds to the class of a fiber
of the elliptic fibration $\phi\colon X\to \PP^1$ defined by the configuration in Figure~\ref{fig:2E8},
the vector $u_2\in M_6$ corresponds to the class $z+u_1$, where 
$z$ is  
the zero section $M_{2\ordplus\ordplus}$ of $\phi$,
and the vectors $e_i\spar{\nu}$ correspond to the reduced parts $C$ of the irreducible components 
of the two reducible fibers of $\phi$ satisfying $\intf{z, C}=0$.
The sign of $v_{12}$ is chosen so that $\intf{v_{12}, C}\ge 0$ holds for 
all $32$ smooth rational curves $C$ in Lemma~\ref{lem:intnumbsPS}.
\end{remark}
%
%
%
\subsection{An ample class \texorpdfstring{$a_{32}$}{a32}}\label{subsec:anample}
Let $\LMrats$ be the set of $32$ smooth rational curves in Lemma~\ref{lem:intnumbsPS}. 
%
%
Let $h_8\in \SX$ be the class of a hyperplane section of the projective model
$\barX_s\subset \PP^6$ of $X$ defined by~\eqref{eq:wuuuvvv}. 
Since $\barX_s$ is a $(2,2,2)$-complete intersection in the hyperplane of $\PP^6$
defined by the first equation of~\eqref{eq:wuuuvvv}, 
it follows that $h_8$ is a nef vector of degree $8$.
Examining the intersection numbers with the $32$ smooth rational curves in $\LMrats$, 
we find 
%
\begin{eqnarray*}
h_{8}&=&[24, 22, -48, -32, -64, -95, -78, -59, -40,\\
&&-20, -48, -32, -64, -95, -78, -59, -40, -20, -6 ].
\end{eqnarray*}
%
%
Recall that $\Sing\barX_s$ consists of $12$ ordinary nodes
$p_{\gamma_1 \gamma_2 \gamma_3}$,
where one of $\gamma_1, \gamma_2, \gamma_3$ is $0$ and the other two are in 
$\{\ordplus, \ordminus\}$.
Thus, we obtain 
\[
\set{r\in \RatsX}{\intf{r, h_8}=0}=\set{L_{\gamma_1 \gamma_2 \gamma_3}}{\textrm{one of $\gamma_1, \gamma_2, \gamma_3$ is $0$}}
\;\;\subset\;\; \LMrats.
\]
It follows that  $\NX$ is the $\SX/\SX$-chamber containing $h_8$ and contained in the region of $\PX$ 
defined by 
\begin{equation}\label{eq:regiondef}
\textrm{$\intf{x, C}\ge 0$ 
for all $C=L_{\gamma_1 \gamma_2 \gamma_3}$
with one of $\gamma_1, \gamma_2, \gamma_3$ being $0$.}
\end{equation}
Now we define
\begin{eqnarray*}
a_{32}&:=& [ 70, 63, -140, -94, -187, -279, -230, -174, -117, \\
&& -59, -140, -94, -187, -279, -230, -174, -117, -59, -17 ].
\end{eqnarray*}
(See Remark~\ref{rem:randomsearcha32} for a method by which we found this vector.)
We verify that
\[
\intf{a_{32}, a_{32}}=32.
\]
%
%
%
%
%
The intersection numbers of $a_{32}$ with elements of $\LMrats$ are 
\begin{equation}\label{eq:Ca32}
\intf{C, a_{32}}=
\begin{cases}
1&\textrm{ if $C=L_{\gamma_1\gamma_2\gamma_3}$ or 
$C\in\{M_{2\ordminus\ordplus}, M_{2\ordplus\ordminus}, M_{3\ordminus\ordminus}, M_{3\ordplus\ordplus} \}$,}\\
4&\textrm{ if $C=M_{1\alpha\beta}$ }, \\
7&\textrm{ if $C\in\{M_{2\ordminus\ordminus}, M_{2\ordplus\ordplus}, M_{3\ordminus\ordplus}, M_{3\ordplus\ordminus} \}$}. \\
\end{cases}
\end{equation}
Hence  $a_{32}$ lies in the region defined by~\eqref{eq:regiondef}.
We confirm by direct computation that $x=a_{32}$ satisfies the following:
\begin{equation}\label{eq:ampleconds}
[x]\sperp\cap \RRR_X=\emptyset, \quad
\Sep(h_8, x)=\emptyset.
\end{equation}
Therefore $a_{32}$ is ample.
\par
Thanks to 
the ample class $a_{32}$,
we can now utilize various tools and methods explained in~\cite{ Shimada2024}.
For example, 
we can determine whether a given $(-2)$-vector $r\in \RX$ belongs to $\RatsX$ or not
by the criterion in Section 3.4 of~\cite{ Shimada2024}.
The numbers of smooth rational curves $C$ on $X$ of low degree $\intf{C, a_{32}}$ 
are given in Table~\ref{table:numrats}.
\begin{remark}
The smooth rational curves $C$ with $\intf{C, a_{32}}<7$ belong to  $\LMrats$.
Only four smooth rational curves among the $8$ curves $C$ with $\intf{C, a_{32}}=7$ are  in $\LMrats$.
The dual graph of the smooth rational curves $C$ with $\intf{C, a_{32}}=1$ (that is,
the lines of the projective model of $X$ by $a_{32}$) is obtained from~\eqref{eq:Ca32} and Lemma~\ref{lem:intnumbsPS}.
\end{remark}
%
%
%
%
 %
 %
 \begin{table}
\begin{eqnarray*}
&&\begin{array}{c|ccccccccccccccccc}
\intf{C, a_{32}}& 1& 2& 3& 4& 5& 6& 7& 8& 9& 10& 11& 12 &13 & 14 & 15 & 16 & 17 \\ 
\hline 
\textrm{number} & 24& 0& 0& 4& 0& 0& 8& 0& 0& 16& 32& 0 & 32& 16& 0&80& 192
\end{array}\\
&&
\begin{array}{c|cccccccccccccc}
\intf{C, a_{32}}& 18 &19& 20& 21& 22& 23& 24 &25 & 26 & 27 & 28 & 29 &&\\ 
\hline 
\textrm{number} &0 & 136& 96& 0& 248& 384& 0 & 416& 320 & 304 &560 &816 &&
\end{array}.
\end{eqnarray*}
\caption{Numbers of smooth rational curves of low degrees}\label{table:numrats}
\end{table}
\begin{remark}\label{rem:randomsearcha32}
We explain a method to find many ample classes  
 by random search.
First we find an ample class $a$ by the following method.
We choose  a random vector $v\in \SX$ 
such that 
$\intf{v, C}> 0$ 
for all $C=L_{\gamma_1 \gamma_2 \gamma_3}$
with one of $\gamma_1, \gamma_2, \gamma_3$ being $0$.
Then we choose a positive integer $n$ and put $a:=nh_8+v$.
If $n$ is sufficiently large, then $a\in \PX$ and $x=a$ satisfies~\eqref{eq:regiondef} and~\eqref{eq:ampleconds},
and hence $a$ is ample.
Once an ample class $a$ is found, we repeat the following process. 
 We choose  a random vector $v\in \SX\cap \PX$ of degree $d:=\intf{v, v}$
such that $[v]\sperp\cap \RRR_X=\emptyset$.
 Then we calculate the set $\Sep(v, a)$ of $(-2)$-vectors in $\SX$ separating  $v$ and $a$.
 Applying to $v$ the reflections with respect to the elements $r$ of $\Sep(v, a)$
 in an appropriate order, we obtain  
 a vector $a\sprime$ of degree  $\intf{a\sprime, a\sprime}=d$ such that $[a\sprime]\sperp\cap \RRR_X=\emptyset$
 and $\Sep(a, a\sprime)=\emptyset$.
 Thus we obtain a new ample class $a\sprime$.
\end{remark}
\begin{remark}
The class $a_{32}$  is the image of the orthogonal projection 
of Weyl vector $\weyl_0\in \Lts$, and hence it 
plays an important role in specifying  the $\Lts/\SX$-chamber $D_0=\PX\cap \Concham(\weyl_0)$ in Borcherds' method. 
See Section~\ref{subsec:Borcherds}.
\end{remark}
\subsection{Embedding \texorpdfstring{$\Aut(X)$}{Aut(X)}  into \texorpdfstring{$\OG(\SX, \PX)$}{OSXPX}}
\label{subsec:EmbeddingAutInO}
Let
\[
q_{\SX}\;\;\colon\;\; \SX\dual/\SX \to \QQ/2\ZZ
\]
denote the discriminant form of the even lattice $\SX$ (see~\cite{theNikulin}), 
where $\SX\dual$ is the dual lattice of $\SX$.
The discriminant group $\SX\dual/\SX$ is a cyclic group of order $12$ generated by $v_{12}/12 \bmod \SX$, and 
satisfies
$q_{\SX}(v_{12}/12)=-1/12 \bmod 2\ZZ$.
Let
\[
\OG(q_{\SX})\cong (\ZZ/12\ZZ)\sptimes=\{\pm 1, \pm 5\}
\]
denote the automorphism group of the finite quadratic form $q_{\SX}$.
We have a natural homomorphism
\[
\eta\colon \OG(\SX)\to \OG(q_{\SX}).
\]
By Theorem 5.4 and Example 5.5 of~\cite{ Shimada2024}, we obtain the following result:
\begin{proposition}\label{prop:Torelli}
The natural homomorphism $\Aut(X) \to \OG(\SX, \PX)$ is injective, and its image
consists precisely of isometries $g\in \OG(\SX, \PX)$ satisfying $\NX^g=\NX$ and $ \eta(g)\in \{\pm 1 \}$.
\qed
\end{proposition}
From this point onward, we will regard $\Aut(X)$ as a subgroup of $\OG(\SX, \PX)$.
An isometry $g\in \OG(\SX, \PX)$ satisfies the condition $\NX^g=\NX$
if and only if the set $\Sep(a_{32}, a_{32}^g)$ of $(-2)$-vectors 
separating
$a_{32}$ and $a_{32}^g$
is empty. Thus, for $g\in \OG(\SX, \PX)$, we have 
\[
g\in \Aut(X) \;\; \Longleftrightarrow\;\; (\;\;\Sep(a_{32}, a_{32}^g)=\emptyset\;\;\textrm{and}\;\;\eta(g)\in \{\pm 1\}\;\;).
\]
\subsection{The finite subgroup \texorpdfstring{$\Aut(X, \LMrats)$}{AutXL32}}\label{subsec:AutXLMrats}
Let $\OG(\SX, \LMrats)$ denote the group of permutations of the set $\LMrats$ of $32$ smooth rational curves in Lemma~\ref{lem:intnumbsPS}
that preserve intersection numbers.
Since the classes of curves in $\LMrats$ generate $\SX$,
we can naturally embed $\OG(\SX, \LMrats)$ into $\OG(\SX)$.
%
%
Since the sum $s$ of elements of $\LMrats$ satisfies
$\intf{s,s}>0$ and $\intf{s, a_{32}}>0$,
 it follows that $\OG(\SX, \LMrats)$ is contained in $\OG(\SX, \PX)$.
We put
\[
\Aut(X, \LMrats):=\OG(\SX, \LMrats)\cap\Aut(X),
\]
where the intersection is taken in $\OG(\SX, \PX)$.
In this section, 
we present various facts about this finite automorphism group $\Aut(X, \LMrats)$.
\newcounter{smallfact}
\setcounter{smallfact}{1}
\newcommand{\newsmallfact}{(\hbox to .25cm{\hfill{\bf \alph{smallfact}}\hfill}) \addtocounter{smallfact}{1}}
\par
\medskip
\newsmallfact The size of the group $\OG(\SX, \LMrats)$ is $96$.
Every element $g$ of $\OG(\SX, \LMrats)$ satisfies $\Sep(a_{32}, a_{32}^g)=\emptyset$,
which implies 
\[
\Aut(X, \LMrats)=\set{g\in \OG(\SX, \LMrats)}{\eta(g)\in \{\pm 1\}}.
\]
Let $\mu \in \OG(\SX, \LMrats)$ be the involution given by 
\[
L_{\gamma_1\gamma_2\gamma_3}^{\mu}=L_{\gamma_1\gamma_2\gamma_3},
\quad
M_{k\alpha\beta}^{\mu}=M_{k ({-\alpha})\beta}.
\]
Then we have $\eta(\mu)=5\in \OG(q_{\SX})$, and 
\[
\OG(\SX, \LMrats)=\angs{\mu}\times \Aut(X, \LMrats).
\]
In particular, the size of the group $\Aut(X, \LMrats)$ is $48$.
This group $\Aut(X, \LMrats)$ acts on the $L$-cube (Figure~\ref{fig:Lcube}) faithfully.
We put the $L$-cube in $\RR^3$ by
\[
L_{\gamma_1 \gamma_2 \gamma_3} \;\; \mapsto \;\; \gamma_1 \vect{e}_1 + \gamma_2 \vect{e}_2 +\gamma_3 \vect{e}_3,
\]
where $\vect{e}_1, \vect{e}_2, \vect{e}_3$ are the standard ortho-normal basis of $\RR^3$.
This gives a representation 
\begin{equation}\label{eq:rhoL}
\rho_{L}\colon \Aut(X, \LMrats)\inj \OG(3).
\end{equation}
Then the morphism $\eta\colon \Aut(X, \LMrats)\to\{\pm 1\}\subset\OG(q_{\SX})$ 
is given by 
\begin{equation}\label{eq:detrhoL}
\eta(g)=1 \;\; \Leftrightarrow\;\; \rho_{L}(g)\in \SO(3).
\end{equation}
\par
\medskip
\newsmallfact
The action of $\Aut(X, \LMrats)$ decomposes
$ \LMrats$ 
 into three orbits
\[
\set{L_{\gamma_1\gamma_2\gamma_3}}{\textrm{none of $\gamma_1,\gamma_2,\gamma_3$ is zero}},\quad
\set{L_{\gamma_1\gamma_2\gamma_3}}{\textrm{one of $\gamma_1,\gamma_2,\gamma_3$ is zero}},\quad 
\{M_{k\alpha\beta}\}.
\]
These orbits have sizes $8, 12, 12$, respectively.
\par
\medskip
\newsmallfact
By the natural embedding $\LMrats\inj \SX$,
we have
\[
\LMrats=\set{r\in \RatsX}{\intf{r, h_8}\le 2}.
\]
Hence the group
\[
\Aut(X, h_8):=\set{g\in \Aut(X)}{h_8^g=h_8}
\]
of projective automorphisms of the $(2,2,2)$-complete intersection $\barX_s\subset \PP^5$ 
given by~\eqref{eq:wuuuvvv}
is contained in $\Aut(X, \LMrats)$. 
In fact, by computing the order of $\Aut(X, h_8)$, 
we can show that $\Aut(X, h_8)=\Aut(X, \LMrats)$.
\par
\medskip
\newsmallfact
There exists an involution $\theenr \in \Aut(X, \LMrats)$ defined by 
\[
L_{\gamma_1\gamma_2\gamma_3}^{\theenr}=L_{(-\gamma_1) (-\gamma_2) (-\gamma_3)},
\quad
M_{k\alpha\beta}^{\theenr}=M_{k(-\alpha)\beta}.
\]
The center of $\Aut(X, \LMrats)$ is equal to $\angs{\theenr}$.
Let $\Sigma\subset \Aut(X, \LMrats)$ be the subgroup
consisting of all $g\in \Aut(X, \LMrats)$ such that
\[
\{P_1^{g}, P_2^{g}, P_3^{g}, P_4^{g}\}=\{P_1, P_2, P_3, P_4\},
\]
where $P_1=L_{+-+}$, $P_2=L_{-++}$, $P_3=L_{---}$, $P_4=L_{++-}$
are the vertices of a regular tetrahedron 
in the cube in Figure~\ref{fig:PQTcube}.
Then we have
\[
\Aut(X, \LMrats)=\angs{\theenr}\times \Sigma,
\]
and $\Sigma$ is isomorphic to the symmetric group $\SSSS_4$.
The involution $\theenr$ is induced by the Enriques involution 
\begin{equation}\label{eq:theEnriques}
x_1\leftrightarrow 1/x_1, \quad 
x_2\leftrightarrow 1/x_2, \quad 
x_3\leftrightarrow 1/x_3, \quad 
x_4\leftrightarrow 1/x_4 \quad 
\end{equation}
of the quartic surface $Y_t$.
This Enriques involution and the associated Enriques surface were studied by Mukai and Ohashi~\cite{MukaiOhashi2015}. 
The action of $\Sigma\cong \SSSS_4$ on $Y_t$
is induced by the permutations 
of the coordinates $(x_1: x_2: x_3: x_4)$ of $\PP^3$.
\par
\medskip
\newsmallfact
We have an isomorphism 
\[
\Aut(X, \LMrats)\,=\, \angs{\theenr}\times \Sigma
\;\; \cong\;\; (\ZZ/2\ZZ)^3\rtimes \SSSS_3.
\]
This isomorphism arises from 
the action of $\Aut(X, \LMrats)$
on the affine Fermi surface $X_s\spcirc$ in $\AA^3$
via the three involutions $\xi_i\leftrightarrow 1/\xi_i$ and 
the permutations 
of the coordinates $(\xi_1,\xi_2, \xi_3)$ of $\AA^3$.
\subsection{Borcherds' method}\label{subsec:Borcherds}
Let $\Lts$ be an even unimodular lattice of rank $26$ and signature $(1, 25)$.
Note that such a lattice is unique up to isomorphism.
We embed $\SX$ into $\Lts$ primitively 
using the technique of discriminant forms~\cite{theNikulin} as follows.
Recall that the discriminant group $\SX\dual/\SX$ is a cyclic group of 
order $12$ generated by $\gamma_S$, 
where $\gamma_S:=v_{12}/12 \bmod\SX$, 
and the discriminant form $q_{\SX}$ is given by
$q_{\SX}(\gamma_S)=-1/12 \bmod 2\ZZ$.
Let $R$ be the \emph{negative-definite} root lattice of type $D_5+A_2$.
We fix a basis $f_1, \dots, f_7$ of $R$ 
as is shown in the Dynkin diagram in Figure~\ref{fig:D5A2}.
Then $R\dual/R$ is a cyclic group of order $12$ generated by $\gamma_R:=\tilde{\gamma}_R \bmod R$, where 
\[
\tilde{\gamma}_R:= \frac{\;1\;}{4}(3 f_1+f_2+2f_3+2f_5)+\frac{\;1\;}{3}(f_6+2 f_7)\;\;\in\;\; R\dual, 
\] 
and we have $q_{R}(\gamma_R)=1/12 \bmod 2\ZZ$, where $q_R\colon R\dual/R \to \QQ/2\ZZ$ 
is the discriminant form of $R$.
Hence 
$\gamma_S \mapsto -\gamma_R$
 gives an anti-isomorphism $q_{\SX}\cong -q_{R}$.
The graph of this anti-isomorphism in $(\SX\dual/\SX)\times (R\dual/R)$ 
yields an even unimodular overlattice $\Lts$ of the orthogonal direct sum $\SX\oplus R$.
Indeed, 
$\Lts$ is generated in $ \SX\dual\oplus R\dual$ over $\SX\oplus R$ by 
the vector $v_{12}/12+\tilde{\gamma}_R$.
From this point forward, 
we regard $\SX$ and $R$ as primitive sublattices of $\Lts$ 
via this embedding $(\SX\oplus R)\inj \Lts$.
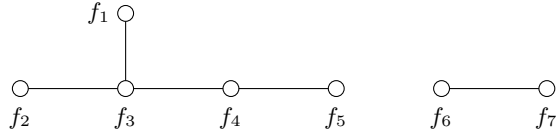
\begin{figure}
\begin{tikzpicture}[x=1.4cm, y=1cm]
{\small
%
%
\coordinate (n1) at (1 ,2);
\coordinate (n2) at (0 ,1);
\coordinate (n3) at (1 ,1);
\coordinate (n4) at (2 ,1);
\coordinate (n5) at (3 ,1);
\coordinate (n6) at (4 ,1);
\coordinate (n7) at (5 ,1);
\draw (n1) -- (n3);
\draw (n2) -- (n3);
\draw (n3) -- (n4);
\draw (n4) -- (n5);
\draw (n6) -- (n7);
\filldraw [fill=white] (n1) circle [radius=3pt];
\filldraw [fill=white] (n2) circle [radius=3pt];
\filldraw [fill=white] (n3) circle [radius=3pt];
\filldraw [fill=white] (n4) circle [radius=3pt];
\filldraw [fill=white] (n5) circle [radius=3pt];
\filldraw [fill=white] (n6) circle [radius=3pt];
\filldraw [fill=white] (n7) circle [radius=3pt];
\node at (n1) [left] {$f_1$\;};
\node at (n2) [below] {$f_2$\mystruth{10pt}};
\node at (n3) [below] {$f_3$\mystruth{10pt}};
\node at (n4) [below] {$f_4$\mystruth{10pt}};
\node at (n5) [below] {$f_5$\mystruth{10pt}};
\node at (n6) [below] {$f_6$\mystruth{10pt}};
\node at (n7) [below] {$f_7$\mystruth{10pt}};
}
\end{tikzpicture}
\caption{Basis of $R$}\label{fig:D5A2}
\end{figure}
\par
Let $\Pts\subset \Lts\tensor\RR$ denote the positive cone of $\Lts$ containing the positive cone $\PX$ of $\SX$.
We refer to an $\Lts/\Lts$-chamber as a \emph{Conway chamber},
as its structure was determined by Conway~\cite{Conway1983}.
The tessellation of $\Pts$ by Conway chambers
induces a tessellation of $\PX$ by $\Lts/\SX$-chambers.
Each $\SX/\SX$-chamber, 
including the nef-and-big cone $\NX$, 
is also tessellated by $\Lts/\SX$-chambers.
For every $g\in \Aut(X)$, its action 
$\eta(g)\in \OG(q_{\SX})$ on the discriminant form $q_{\SX}$ is in $ \{\pm 1\}$, 
and hence the action of $g$ on $\SX$ extends to an action on $\Lts$.
Consequently, 
the action of $\Aut(X)$ on $\NX$ preserves the tessellation of $\NX$ by $\Lts/\SX$-chambers.
We put 
\begin{equation}\label{eq:VX}
\VX:=\textrm{the set of $\Lts/\SX$-chambers contained in $\NX$}.
\end{equation}
Our goal is to analyze the action of $\Aut(X)$ on $\NX$ 
via the the action of $\Aut(X)$ on $\VX$.
\begin{definition}\label{def:innerouter}
Let $D$ be an element of $\VX$,
and $f$ a face of $D$.
We say that $f$ is \emph{inner} if the set $\DDD(f)$ of all $\Lts/\SX$-chambers 
containing $f$ is a subset of $\VX$.
Otherwise, we say that $f$ is \emph{outer}.
\end{definition}
Suppose that $w=D\cap (v)\sperp$ is a wall of $D\in \VX$,
where $v\in \SX\dual$ is 
the primitive defining vector~(see Section~\ref{subsec:chambers}).
Then $w$ is inner if and only if 
the $\Lts/\SX$-chamber adjacent to $D$ across the wall $w$
belongs to $\VX$.
It is also obvious that 
$w$ is outer if and only if $v$
is equal to $\alpha C$ for some $\alpha\in \QQ_{>0}$ and $C\in \RatsX$.
\par
\medskip
We put
\[
\Rts:=\set{r\in \Lts}{\intf{r,r}=-2}.
\]
Recall that $a_{32} \in \SX$ is an ample class with $\intf{a_{32}, a_{32}}=32$,
which we regard as a vector of $\Lts$ by the embedding $\SX\inj \Lts$.
\begin{proposition}\label{claim:a32interior}
The ample class $a_{32}$ is an interior point of an $\Lts/\SX$-chamber.
\end{proposition}
\begin{proof}
By direct computation,
we verify that the set $\set{r\in \Rts}{\intf{r, a_{32}}=0}$
is equal to the set
\[
\set{r\in \Rts}{\intf{r, v}=0\;\;\textrm{for all}\;\; v\in \SX}=
\set{r\in \Rts}{r\in R}\cong\set{r\in R}{\intf{r,r}=-2}
\]
of roots of $R$.
This implies that, 
if $r\in \Rts$ satisfies $\intf{r, a_{32}}=0$,
then we have $\PX\subset (r)\sperp$ in $\Pts$.
\end{proof}
\begin{definition}\label{def:weyl}
A  vector 
$\weyl$ of $\Lts$  is called a \emph{Weyl vector} if 
\begin{enumerate}[(i)]
\item $\weyl$ is non-zero, primitive in $\Lts$, and  of square-norm $0$, 
\item 
$\weyl$ is contained in the closure of $\Pts$ in $\Lts\tensor\RR$,
and 
\item 
the negative-definite even unimodular lattice $[\ZZ\weyl]\sperp/\ZZ\weyl$ of rank $24$ 
contains no vectors of square-norm $-2$.
\end{enumerate}
For a Weyl vector $\weyl$, 
we call a $(-2)$-vector $r\in \Rts$ of $\Lts$   a \emph{Leech root of $\weyl$}  if $\intf{r, \weyl}=1$ holds.
\end{definition}
Conway~\cite{Conway1983} proved that the mapping $\weyl\mapsto \Concham(\weyl)$,
where
\[
\Concham(\weyl) :=\set{x\in \Pts}{\intf{x, r}\ge 0 \;\;\textrm{for all Leech roots $r$ of $\weyl$}\;\; },
\]
is a bijection from the set of Weyl vectors to the set of Conway chambers.
Moreover, he showed that
$\Concham(\weyl)\cap (r)\sperp$ is a wall of the Conway chamber $\Concham(\weyl)$
for each Leech roots $r$ of $\weyl$;
that is, $\Concham(\weyl)\cap (r)\sperp$ contains a nonempty open subset of $(r)\sperp$
for every Leech root $r$ of $\weyl$.
\par
We put
\[
a_R:=[ -5, -5, -9, -7, -4, -1, -1 ] \in R,
\]
which is a vector of $R$ satisfying $\intf{a_R, f_j}=1$ for $j=1, \dots, 7$.
Since $\intf{a_R, a_R}=-32$,
the vector 
\[
\weyl_0:=a_{32}+a_R
\]
of $\Lts$ is of square-norm $0$.
We verify that $\weyl_0$ is a primitive vector in $\Lts$, and that  
the negative-definite even unimodular lattice $[\ZZ\weyl_0]\sperp/\ZZ\weyl_0$
has no $(-2)$-vectors.
Thus, we confirm that $\weyl_0$ is a Weyl vector.
\begin{proposition}\label{claim:a32interior2}
The closed subset 
\[
D_0:=\PX\cap \Concham(\weyl_0)
\]
of $\PX$ is the $\Lts/\SX$-chamber containing $a_{32}$ in its interior.
\end{proposition}
\begin{proof}
We have already proved
 that $a_{32}$ is an interior point of a certain $\Lts/\SX$-chamber
in $\PX$.
Thus, it suffices to show that $a_{32}$ lies in $ \Concham(\weyl_0)$.
Since $\weyl_0\in \Lts$ is a primitive vector with square-norm $0$
and we have $\Lts=\Lts\dual$, 
there exists a vector $\weyl_0\sprime\in \Lts$ such that
$\intf{\weyl_0\sprime, \weyl_0\sprime}=0$ and $\intf{\weyl_0, \weyl_0\sprime}=1$.
Then $\weyl_0$ and $\weyl_0\sprime$ span a hyperbolic plane $U_{\weyl}$ 
in $\Lts$,
and its orthogonal complement $\Lambda:=U_{\weyl}\sperp$ is isomorphic to the negative-definite Leech lattice $[\ZZ\weyl_0]\sperp/\ZZ\weyl_0$.
Thus we can write $\Lts=U_{\weyl}\oplus \Lambda$.
The Leech roots with respect to $\weyl_0$
are given by
\[
r_{\lambda}:=\left(\frac{-2-\intf{\lambda, \lambda}}{2}\right)\weyl_0 +\weyl_0\sprime+\lambda, \quad\textrm{where}\quad \lambda \in \Lambda.
\]
We put
\[
\ampleL:=2 \weyl_0 +\weyl_0\sprime.
\]
Since $\intf{\ampleL,\ampleL}>0$ and $\intf{\ampleL, r_{\lambda}}>0$ for any $\lambda\in \Lambda$,
it follows that $\ampleL$ is an interior point of $ \Concham(\weyl_0)$.
Then we confirm that the set 
\[
\Sep_{\Lts}(\ampleL, a_{32}):=\set{r\in \Rts}{\;\intf{\ampleL, r}>0, \;\; \intf{a_{32}, r}<0\; }
\]
of $(-2)$-vectors in $\Lts$ separating $\ampleL$ and $a_{32}$ is empty
by the algorithm given in~\cite{Shimada2014}.
Therefore $a_{32}$ belongs to the Conway chamber $ \Concham(\weyl_0)$.
\end{proof}
\begin{remark}
The order of the Weyl group $W(R)$ of the root lattice $R$ of type $D_5+A_2$ is $11,520$.
Consequently, there exist exactly $11,520$ Conway chambers $\Concham\sprime$ 
such that $D_0=\PX\cap \Concham\sprime$.
\end{remark}
Starting from $D_0$,
we execute the algorithm described in Section 5 of~\cite{ Shimada2024},
and obtain 
the orbit decomposition of $\VX$ under the action of $\Aut(X)$,
where $\VX$ is the set of $\Lts/\SX$-chambers contained in $\NX$ (see~\eqref{eq:VX}).
As a result, we obtain the following facts.
\newcounter{bigfact}
\setcounter{bigfact}{1}
\newcommand{\newbigfact}{(\hbox to .4cm{\hfill{\bf \arabic{bigfact}}\hfill})\; \addtocounter{bigfact}{1}}
\par
\medskip
\newbigfact The $\Lts/\SX$-chamber $D_0$ has $80$ walls.
Let $w_1, \dots, w_{80}$ be the walls of $D_0$, 
and let $v_i\in \SX\dual$ be the primitive defining vector of $w_i$.
(See Section~\ref{subsec:chambers} for the definition of the primitive defining vector.)
The values
\[
n(w_i):=\intf{v_i,v_i}\quad\textrm{and}\quad 
a(w_i):=\intf{a_{32},v_i}
\]
for each wall $w_i$ 
are given in Table~\ref{table:D0wallsdata}.
\par
\medskip
\newbigfact The group
\[
\OG(\SX, D_0):=\set{g\in \OG(\SX, \PX)}{D_0^g=D_0}
\]
is of order $32$, and is equal to
\[
\OG(\SX, a_{32}):=\set{g\in \OG(\SX, \PX)}{a_{32}^g=a_{32}}.
\]
Its subgroup
\[
\Aut(X, D_0):=\Aut(X)\cap \OG(\SX, D_0)=\set{g\in \OG(\SX, D_0)}{\eta(g)\in \{\pm 1\}}
\]
is isomorphic to the dihedral group of order $16$.
We see that $\Aut(X, D_0)$ is equal to the group 
\[
\Aut(X, a_{32}):=\set{g\in \Aut(X)}{a_{32}^g=a_{32}}
\]
of the projective model of $X$ defined by $a_{32}$.
Table~\ref{table:D0wallsdata} shows 
the orbit decomposition of the set of walls of $D_0$ by the action of 
$\Aut(X, D_0)$.
In Table~\ref{table:D0wallsvects},
the primitive defining vector of a representative wall of each orbit $o_i$ is given.
The orbits $o_5$ and $o_6$ merge into a single orbit under the action of $\OG(\SX, D_0)$,
as do the orbits $o_9$ and $o_{10}$.
Meanwhile, each of the other six orbits remains to be an orbit under $\OG(\SX, D_0)$.
\begin{table}
\[
\begin{array}{c|ccccc}
\textrm{orbit} & \textrm{size} & n(w_i) & a(w_i) & \textrm{in $\NX$} & \intf{a_{32}, a_{32}^g} \\
\hline
%
o_{1} & 8 & -2 & 1 & \textrm{out} & 33\\ 
o_{2} & 16 & -2 & 1 & \textrm{out} & 33\\ 
o_{3} & 4 & -4/3 & 2 & \textrm{inn} & 38\\ 
o_{4} & 8 & -1 & 5 & \textrm{inn} & 82\\ 
o_{5} & 8 & -3/4 & 6 & \textrm{inn} & 128\\ 
o_{6} & 8 & -3/4 & 6 & \textrm{inn} & 128\\ 
o_{7} & 8 & -3/4 & 6 & \textrm{inn} & 128\\ 
o_{8} & 4 & -1/3 & 6 & \textrm{inn} & 248\\ 
o_{9} & 8 & -1/12 & 7 & \textrm{inn} & 1208\\ 
o_{10} & 8 & -1/12 & 7 & \textrm{inn} & 1208
\end{array}
\]
\caption{Walls of $D_0$}\label{table:D0wallsdata}
\end{table}
\begin{table}
{\scriptsize
\[
\begin{array}{lcl}
%
o_{1} &\colon & [ 0, 0, 0, 0, 0, 0, 0, 0, 0, 0, 0, 0, 0, 0, 0, 0, 0, 1, 0 ]\\ 
o_{2} &\colon & [ 0, 0, 0, 0, 0, 0, 0, 0, 0, 0, 0, 0, 0, 0, 0, 0, 1, 0, 0 ]\\ 
o_{3} &\colon & [ -6, -6, 12, 8, 16, 24, 20, 15, 10, 5, 12, 8, 16, 24, 20, 15, 10, 5, 5/3 ]\\ 
o_{4} &\colon & [ 6, 5, -12, -8, -16, -24, -20, -15, -10, -5, -10, -7, -14, -20, -16, -12, -8, -4, -3/2 ]\\ 
o_{5} &\colon & [ 5, 5, -12, -8, -16, -24, -20, -15, -10, -5, -9, -6, -12, -18, -15, -12, -8, -4, -5/4 ]\\ 
o_{6} &\colon & [ 7, 6, -14, -10, -19, -28, -23, -18, -12, -6, -12, -8, -16, -24, -20, -15, -10, -5, -7/4 ]\\ 
o_{7} &\colon & [ 14, 12, -29, -19, -38, -57, -47, -36, -24, -12, -27, -18, -36, -54, -44, -33, -22, -11, -13/4 ]\\ 
o_{8} &\colon & [ 13, 12, -27, -18, -36, -54, -44, -33, -22, -11, -26, -18, -35, -52, -43, -33, -22, -11, -19/6 ]\\ 
o_{9} &\colon & [ 12, 11, -25, -17, -33, -49, -40, -30, -20, -10, -24, -16, -32, -48, -40, -30, -20, -10, -35/12 ]\\ 
o_{10} &\colon & [ 14, 12, -27, -18, -36, -54, -44, -33, -22, -11, -27, -18, -36, -54, -44, -33, -22, -11, -41/12 ]\\ 
%
%
\end{array}
\]
}
\caption{Primitive defining vectors of walls of $D_0$}\label{table:D0wallsvects}
\end{table}
\par
\medskip
\newbigfact
Let $w$ be a wall of $D_0$.
If $w\in o_1\cup o_2$,
then $w$ is an outer wall.
Suppose instead that $w\in o_3\cup \dots\cup o_{10}$.
Then 
 the $\Lts/\SX$-chamber 
adjacent to $D_0$
across the wall $w$
is congruent to $D_0$ by the action of $\Aut(X)$.
In other words, the set
\begin{equation}\label{eq:Adj}
\Adj(w):=\bigset{g\in \Aut(X)}{\parbox{6cm}{$D_0^g$ is the $\Lts/\SX$-chamber adjacent to $D_0$ across the wall $w$}}
\end{equation}
is nonempty.
Thus, by~Proposition 4.1 of~\cite{BrandhorstShimada2022}
(see also Proposition 5.1 of~\cite{ Shimada2024}),
we obtain the following result.
\begin{proposition}\label{prop:fromfact3}
{\rm (1)} The group $\Aut(X)$ acts transitively on the set $\VX$ of $\Lts/\SX$-chambers contained in $\NX$.
\par
{\rm (2)} 
From each orbit $o_{\nu}$ for $\nu=3, \dots, 10$,
we choose a wall $w\spar{\nu}\in o_{\nu}$,
and an element $g(w\spar{\nu})$ of $\Adj(w\spar{\nu})$.
Then 
 $\Aut(X)$ is generated by the finite subgroup $\Aut(X, D_0)$ together with 
eight extra automorphisms 
$g(w\spar{\nu})$ for $\nu=3, \dots, 10$.
\qed
\end{proposition}
%
\par
\medskip
\newbigfact
The outer walls in the orbit $o_1$ are given as $D_0\cap (C_1)\sperp$,
where $C_1\in \RatsX$ are the following $8$ smooth rational curves:
\begin{equation}\label{eq:o1}
L_{0\ordplus\ordplus},\;\; L_{0\ordplus\ordminus}, \;\; L_{0\ordminus\ordplus}, \;\; L_{0\ordminus\ordminus}, \;\; 
M_{2\ordplus\ordminus},\;\; M_{2\ordminus\ordplus}, \;\; M_{3\ordplus\ordplus}, \;\; M_{3\ordminus\ordminus}.
\end{equation}
The outer walls in $o_2$ are given as $D_0\cap (C_2)\sperp$,
where $C_2$ ranges through the set
\begin{equation}\label{eq:o2}
\set{L_{\gamma_1\gamma_2\gamma_3}}{\gamma_1\ne 0}.
\end{equation}
\par
\medskip
\newbigfact
In the rightmost column of Table~\ref{table:D0wallsdata},
we present $\intf{a_{32}, a_{32}^g}$,
where $g$ is an isometry in $ \OG(\SX, \PX)$
such that $D_0^{g}$ is adjacent to $D_0$
across a wall  $w\in o_{\nu}$.
(For $\nu=3, \dots, 10$, we have $g\in \Adj(w)$.)
For a fixed wall $w$,
the element $g$ is unique up to the multiplication from the left by elements 
of $\OG(\SX, D_0)=\OG(\SX, a_{32})$.
Hence  $\intf{a_{32}, a_{32}^g}$ does not depend on the choice of $g$.
\subsection{Proof of Corollary~\ref{cor:ratstransitive}}\label{subsec:ProofOfCor}
Now we can prove Corollary~\ref{cor:ratstransitive}, even though 
Theorem~\ref{thm:rats} has not been proved yet.
Let $C$ be an arbitrary element of $\RatsX$,
and set $r:=C$.
Then $\NX\cap(r)\sperp$ contains a nonempty open subset of $(r)\sperp$,
and hence there exists an $\Lts/\SX$-chamber $D\in \VX$ such that
$D\cap (r)\sperp$ is a wall of $D$.
By Proposition~\ref{prop:fromfact3},
there exists an automorphism $g\in\Aut(X)$ such that $D^{g}=D_0$.
Then 
 $D_0\cap (r^g)\sperp$ is an outer wall of $D_0$.
From~\eqref{eq:o1} and~\eqref{eq:o2},
there exists an automorphism $g\sprime\in\Aut(X, D_0)$ such that
\[
r^{gg\sprime}=L_{0\ordplus\ordplus} \quad\textrm{or}\quad r^{gg\sprime}=L_{\ordplus 0\ordplus}.
\]
By Fact (b) in Section~\ref{subsec:AutXLMrats}, there exists an automorphism $g\spprime\in\Aut(X, \LMrats)$ such that
$r^{gg\sprime g\spprime}=L_{0\ordplus\ordplus}$.
Consequently, $\RatsX$ forms a single orbit under the action of $\Aut(X)$.
\qed
\section{Geometric description of generators}\label{sec:geometric}
In Proposition~\ref{prop:fromfact3},
we provided a finite generating set of $\Aut(X)$
in lattice-theoretic terms;
that is,
we presented a finite set of isometries of $\SX$ 
that generate the subgroup $\Aut(X)$ of $\OG(\SX, \PX)$.
In this section,
we describe these generators in terms of the geometry of $X$.
 We state  our goal  precisely in Section~\ref{subsec:goal},
 and present a strategy in Section~\ref{subsec:strategy}.
 Then, in Sections~\ref{subsec:AutXD0}--\ref {subsec:extraaut9and10},
 we describe the generators geometrically.
\subsection{Goal}\label{subsec:goal}
%
%
Recall the definition~\eqref{eq:Adj} of $\Adj(w)$.
\begin{definition}\label{def:generatorassociatedwithawall}
Let $o_{\nu}$ be an orbit  of inner walls of $D_0$ under the action of $\Aut(X, D_0)$.
We say that $g\in \Aut(X)$ is a \emph{generator associated with $o_{\nu}$}
if $g\in \Adj(w)$ for a wall $w\in o_{\nu}$.
\end{definition}
The geometric origin of the finite subgroup 
$\Aut(X, \LMrats)$ of order $48$ is well understood.
(See Section~\ref{subsec:AutXLMrats}.)
The intersection $\Aut(X, D_0)\cap \Aut(X, \LMrats)$
is of order $8$, and 
an element $g$ of $\Aut(X, \LMrats)$ belongs to $\Aut(X, D_0)$ 
if and only if its action $\rho_L(g)$ on the $L$-cube (see~\eqref{eq:rhoL} and~\eqref{eq:detrhoL})
preserves the set $\{L_{0\pm\pm}\}$ of four vertices and satisfies $\det (\rho_L(g))=1$.
\par
Our goal is to provide a geometric description of 
\begin{enumerate}[(a)]
\item 
an automorphism $g\spbr{0}$ in $\Aut(X, D_0)$ 
not belonging to $\Aut(X, \LMrats)$, 
and 
\item 
generators $g\spbr{3}, \dots, g\spbr{10}$ associated with the orbits $o_3, \dots, o_{10}$ of inner walls. 
\end{enumerate}
Then $\Aut(X)$ is generated by $\Aut(X, \LMrats)$
along with $g\spbr{0}, g\spbr{3}, \dots, g\spbr{10}$.
\subsection{Strategy}\label{subsec:strategy}
For each inner wall $w$ of $D_0$,
we calculate 
\[
a(w):=a_{32}^g, \quad\textrm{where}\quad g\in \Adj(w).
\]
Note that, since $\Aut(X, D_0)=\Aut(X, a_{32})$,
 the vector $a(w)$ does not depend on the choice of $g\in \Adj(w)$.
 Then we search for as many  automorphisms $g$ with clear geometric meaning  as possible, 
 and calculate their actions on $\SX$.
 If $a_{32}^g$ is equal to $a(w)$ for some inner wall $w\in o_{\nu}$,
 then we adopt this automorphism $g$ as $g\spbr{\nu}$.
\par
To obtain many  geometric automorphisms, 
we use the Jacobian fibrations and their sections.
Let $\phi\colon X\to \PP^1$ be a Jacobian fibration with 
the zero section $z\in \RatsX$.
Let $E_{\phi}$ be the generic fiber 
of $\phi$,
which is an elliptic curve defined over the function field of $\PP^1$
with the zero element $z$.
We regard the Mordell-Weil group $\MW(\phi)$ 
of $\phi$ as a subgroup of $\Aut(X)$ by identifying 
a rational point $s\in \MW(\phi)$ of $E_{\phi}$ 
with the translation $x\mapsto x \mathbin{\ordplus_E} s$ of $E_{\phi}$ by $s$,
where $\mathbin{\ordplus_E} $ is the addition on the elliptic curve $E_{\phi}$.
We also have an involution $\iota(\phi)\in \Aut(X)$ coming from 
the inversion $x\mapsto z\mathbin{\ordminus_E} x$ of $E_{\phi}$.
\par
Suppose that we have a configuration 
\[
\Theta=\{C_0, \dots, C_n\}
\]
of smooth rational curves $C_i\in \RatsX$
whose dual graph is a \emph{connected affine}
Dynkin diagram of type $\tilA_{\ell}$, $ \tilD_m$, or $\tilE_n$.
Then $\Theta$ yields an elliptic fibration
\[
\phi_{\Theta}\colon X\to \PP^1
\]
such that $\Theta$ is the set of irreducible components 
of a reducible fiber 
\[
\phi_{\Theta}\inv(p)=\sum a_i C_i,
\]
where the coefficients $a_i\in \ZZ_{>0}$ are determined by the $\ADE$-type of $\Theta$.
(See Table 4.1 of~\cite{Shimada2024}.)
A smooth rational curve $C$ is a section of $\phi_{\Theta}$ if and only if 
\[
\sum a_i \intf{C_i, C}=1.
\]
Hence, if we find an appropriate configuration of elements of $\RatsX$
whose dual graph contains a connected affine
Dynkin diagram, 
we obtain a Jacobian fibration of $X$ 
and some elements of its Mordell--Weil group.
The procedure for computing the Mordell--Weil group and its action on $\SX$ 
is explained in~\cite{Shimada2024}.
The action of the inversion $\iota(\phi)$ 
on $\SX$ is also easily computed. 
\par
We search for  connected affine
Dynkin diagrams in the dual graph of $\LMrats$,
and when we find one, 
we search for sections of the corresponding elliptic fibration  also in $\LMrats$.
In the following, we
present 
 dual graphs of such configurations.
In each graph, 
we depict sections by black nodes,
elements of $\Theta$ 
disjoint from the zero section $z$ by white nodes,
and 
the element of $\Theta$ 
intersecting $z$ by a gray node.
Hence the white nodes form 
a connected \emph{ordinary}
Dynkin diagram, and 
the black node connected to the gray node is the zero section $z$.
\subsection{The group \texorpdfstring{$\Aut(X, D_0)$}{AutXD0}}\label{subsec:AutXD0}
The configuration $\LMrats$ contains a sub-configuration 
depicted in Figure~\ref{fig:extraaut0ellfib}.
The white node form a Dynkin diagram of type $A_1$,
and, together with the gray node, they form an affine Dynkin diagram of type $\tilA_1$.
Therefore we obtain a Jacobian fibration 
\[
\phi\spbr{0}\colon X\to \PP^1
\]
with the zero section $z:=L\sb{0 \ordminus \ordminus }$.
This Jacobian fibration  
has four reducible fibers of type $A_7+A_7+A_1+A_1$.
The Mordell--Weil group $\MW(\phi\spbr{0})$ of $\phi\spbr{0}$ is isomorphic to $\ZZ\oplus \ZZ/4\ZZ$.
The section $s:=M\sb{3\ordplus \ordplus }$ generates the free part of $\MW(\phi\spbr{0})$.
We calculate the action  on $\SX$ of  the product
\[
g\spbr{0}:=\iota(\phi\spbr{0})\cdot s\;\; \in \;\; \Aut(X).
\]
It turns out that 
$g\spbr{0}$ is an involution,
and that it belongs  to $\Aut(X, D_0)$ but not to $ \Aut(X, D_0)\cap \Aut(X, \LMrats)$.
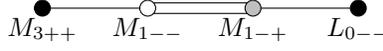
\begin{figure}
\begin{tikzpicture}[x=.7cm, y=.7cm]
%
\coordinate (e1) at (0,0);
\coordinate (th) at (2,0);
\coordinate (z) at (4,0);
\coordinate (sec) at (-2,0);
\draw [name path=rectangle] (-0, -.1) rectangle +(2,.2);
\draw (th) -- (z) ;
\draw (sec) -- (e1) ;
\filldraw [fill=white] (e1) circle [radius=3pt];
\filldraw [fill=lightgray] (th) circle [radius=3pt];
\filldraw [fill=black] (z) circle [radius=3pt];
\filldraw [fill=black] (sec) circle [radius=3pt];
\node at (e1) [below] {$M\sb{1\ordminus \ordminus }$};
\node at (th) [below] {$M\sb{1\ordminus \ordplus}$};
\node at (z) [below] {$L\sb{0\ordminus \ordminus }$};
\node at (sec) [below] {$M\sb{3\ordplus\ordplus}$};
\end{tikzpicture}
\caption{Configuration for $\phi\spbr{0}$}\label{fig:extraaut0ellfib}
\end{figure}
\begin{remark} 
The Jacobian fibration $\phi\spbr{0}$ has a beautiful property
with respect to $\LMrats$.
Since the reducible fibers of $\phi\spbr{0}$ is of type $A_7+A_7+A_1+A_1$,
there exist $20$ smooth rational curves contained in fibers of $\phi\spbr{0}$.
All these $20$ curves belong to $\LMrats$.
The other $12$ smooth rational curves in $\LMrats$ are sections of $\phi\spbr{0}$.
The section $t=L_{0\ordplus\ordminus}$ is a torsion section of order $4$,
and the following sections belong to $\LMrats$:
\[
as \mathbin{\ordplus_E} bt \qquad(a=\{-1,0,1\}, \;\; b\in \{0,1,2,3\}), \;\; \textrm{where}\;\; s:=M\sb{3\ordplus \ordplus }.
\]
\end{remark}
\subsection{A generator associated with \texorpdfstring{$o_3$}{o3}}\label{subsec:extraaut3}
Let $g\spbr{3}$ be the involution in $\Aut(X, \LMrats)$
whose action on the $L$-cube makes the exchanges 
\[
L_{\ordplus\ordplus\ordplus}\;\;\longleftrightarrow\;\; L_{\ordminus\ordminus\ordplus},
\quad 
L_{\ordplus\ordplus\ordminus}\;\;\longleftrightarrow\;\; L_{\ordminus\ordminus\ordminus},
\]
and fixes the other four vertexes $L_{\ordplus\ordminus\pm}$ and $L_{\ordminus\ordplus\pm}$.
Then $g\spbr{3}$ is a generator associated with the orbit $o_3$.
\subsection{A generator associated with \texorpdfstring{$o_4$}{o4}}\label{subsec:extraaut4}
The configuration $\LMrats$ contains a sub-configuration 
depicted in Figure~\ref{fig:extraaut4ellfib}.
The seven white nodes form a Dynkin diagram of type $E_7$,
and, together with the gray node, they form an affine Dynkin diagram of type $\tilE_7$.
Therefore we obtain a Jacobian fibration 
\[
\phi\spbr{4}\colon X\to \PP^1
\]
with the zero section $L\sb{\ordplus \ordminus \ordplus }$.
This Jacobian fibration  
has three reducible fibers of type $E_7+D_5+A_5$.
The Mordell--Weil group  of $\phi\spbr{4}$ is isomorphic to $\ZZ/2\ZZ$
and its non-trivial element is the section $ L\sb{\ordplus \ordplus \ordminus }$.
Let 
\[
g\spbr{4}\colon X\to X
\]
be the translation by the non-trivial torsion section $ L\sb{\ordplus \ordplus \ordminus }$.
Then $g\spbr{4}$ is a generator associated with the orbit $o_4$.
\begin{figure}
\begin{tikzpicture}[x=.7cm, y=.7cm]
%
\coordinate (e1) at (0, 2);
\coordinate (e2) at (-4, 0);
\coordinate (e3) at (-2, 0);
\coordinate (e4) at (0, 0);
\coordinate (e5) at (2, 0);
\coordinate (e6) at (4, 0);
\coordinate (e7) at (6, 0);
\coordinate (th) at (-6, 0);
\coordinate (z) at (-6, -2);
\coordinate (tor) at (6, -2);
\draw (e1) -- (e4) ;
\draw (th) -- (z) ;
\draw (th) -- (e2) ;
\draw (e2) -- (e3) ;
\draw (e3) -- (e4) ;
\draw (e4) -- (e5) ;
\draw (e5) -- (e6) ;
\draw (e6) -- (e7) ;
\draw (e7) -- (tor) ;
\filldraw [fill=white] (e1) circle [radius=3pt];
\filldraw [fill=white] (e2) circle [radius=3pt];
\filldraw [fill=white] (e3) circle [radius=3pt];
\filldraw [fill=white] (e4) circle [radius=3pt];
\filldraw [fill=white] (e5) circle [radius=3pt];
\filldraw [fill=white] (e6) circle [radius=3pt];
\filldraw [fill=white] (e7) circle [radius=3pt];
\filldraw [fill=lightgray] (th) circle [radius=3pt];
\filldraw [fill=black] (z) circle [radius=3pt];
\filldraw [fill=black] (tor) circle [radius=3pt];
\node at (e1) [right] {\;\;$L\sb{0 \ordminus \ordminus }$};
\node at (e2) [below] {$L\sb{\ordminus \ordminus \ordplus }$};
\node at (e3) [below] {$L\sb{\ordminus \ordminus 0 }$};
\node at (e4) [below] {$L\sb{\ordminus \ordminus \ordminus }$};
\node at (e5) [below] {$L\sb{\ordminus 0 \ordminus }$};
\node at (e6) [below] {$L\sb{\ordminus \ordplus \ordminus }$};
\node at (e7) [right] {\;\;$L\sb{0 \ordplus \ordminus }$};
\node at (th) [left] {$L\sb{0 \ordminus \ordplus }$\;\;};
\node at (z) [left] {$L\sb{\ordplus \ordminus \ordplus }$\;\;};
\node at (tor) [right] {\;\;$L\sb{\ordplus \ordplus \ordminus }$};
\end{tikzpicture}
\caption{Configuration for $\phi\spbr{4}$}\label{fig:extraaut4ellfib}
\end{figure}
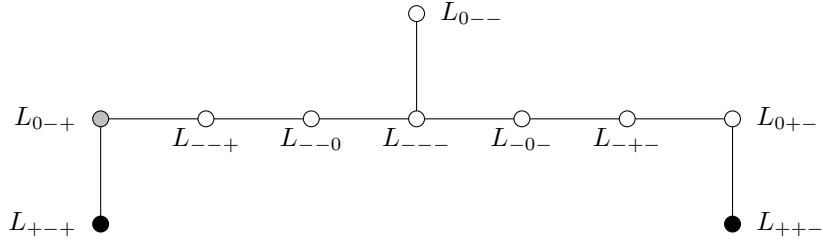
\subsection{Generators associated with \texorpdfstring{$o_5$}{o5} and \texorpdfstring{$o_6$}{o6}}\label{subsec:extraaut5and6}
The configuration $\LMrats$ contains a sub-configuration 
depicted in Figure~\ref{fig:extraaut5ellfib}.
The seven white and gray nodes form an affine Dynkin diagram of type $\tilE_6$,
and hence 
we obtain a Jacobian fibration 
\[
\phi\spbr{5}\colon X\to \PP^1
\]
with the zero section $L\sb{\ordminus 0 \ordplus }$.
This Jacobian fibration $\phi\spbr{5}$ 
has four reducible fibers of type $E_6+E_6+A_2+A_2$.
The Mordell--Weil group of $\phi\spbr{5}$ is isomorphic to $\ZZ\oplus\ZZ/3\ZZ$,
and 
the section $ L\sb{\ordplus 0 \ordminus }$ 
is of order $\infty$.
Let 
\[
g\spbr{5}\colon X\to X
\]
be the automorphism obtained by the translation by this section $L\sb{\ordplus 0 \ordminus }$.
Then $g\spbr{5}$ is a generator associated with the orbit $o_5$.
We put
\[
g\spbr{6}:=(g\spbr{5})\inv.
\]
Then $g\spbr{6}$ is a generator associated with the orbit $o_6$.
\begin{figure}
\begin{tikzpicture}[x=.7cm, y=.7cm]
%
\coordinate (e1) at (0, 2);
\coordinate (e2) at (-4, 0);
\coordinate (e3) at (-2, 0);
\coordinate (e4) at (0, 0);
\coordinate (e5) at (2, 0);
\coordinate (e6) at (4, 0);
\coordinate (th) at (2, 2);
\coordinate (z) at (4, 2);
\coordinate (sec) at (-4, 2);
\draw (e1) -- (e4) ;
\draw (e1) -- (th) ;
\draw (th) -- (z) ;
\draw (e2) -- (e3) ;
\draw (e3) -- (e4) ;
\draw (e4) -- (e5) ;
\draw (e5) -- (e6) ;
\draw (sec) -- (e2) ;
\filldraw [fill=white] (e1) circle [radius=3pt];
\filldraw [fill=white] (e2) circle [radius=3pt];
\filldraw [fill=white] (e3) circle [radius=3pt];
\filldraw [fill=white] (e4) circle [radius=3pt];
\filldraw [fill=white] (e5) circle [radius=3pt];
\filldraw [fill=white] (e6) circle [radius=3pt];
\filldraw [fill=lightgray] (th) circle [radius=3pt];
\filldraw [fill=black] (z) circle [radius=3pt];
\filldraw [fill=black] (sec) circle [radius=3pt];
\node at (e1) [above] {$L\sb{\ordminus \ordminus 0 }$\mystrutd{5pt}};
\node at (e2) [below] {$L\sb{\ordplus \ordminus \ordminus }$};
\node at (e3) [below] {$L\sb{0 \ordminus \ordminus }$};
\node at (e4) [below] {$L\sb{\ordminus \ordminus \ordminus }$};
\node at (e5) [below] {$L\sb{\ordminus 0 \ordminus }$};
\node at (e6) [below] {$L\sb{\ordminus \ordplus \ordminus }$};
\node at (th) [above] {$L\sb{\ordminus \ordminus \ordplus }$\mystrutd{5pt}};
\node at (z) [above] {$L\sb{\ordminus 0 \ordplus }$\mystrutd{5pt}};
\node at (sec) [above] {$L\sb{\ordplus 0 \ordminus }$\mystrutd{5pt}};
\end{tikzpicture}
\caption{Configuration for $\phi\spbr{5}$}\label{fig:extraaut5ellfib}
\end{figure}
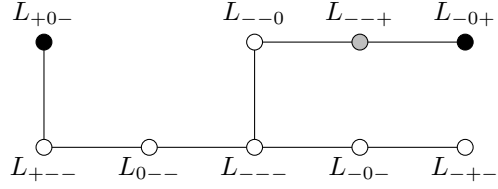
\subsection{A generator associated with \texorpdfstring{$o_7$}{o7}}\label{subsec:extraaut7}
We consider the Jacobian fibration 
\[
\phi\spbr{7}\colon X\to \PP^1
\]
associated with the configuration whose dual graph is in Figure~\ref{fig:extraaut7ellfib}.
The zero section of $\phi\spbr{7}$
is $L\sb{\ordminus \ordminus \ordplus }$.
This Jacobian fibration $\phi\spbr{7}$ 
has four reducible fibers of type $E_6+E_6+A_2+A_2$.
The Mordell--Weil group of $\phi\spbr{7}$ is isomorphic to $\ZZ\oplus\ZZ/3\ZZ$,
and the section $L\sb{\ordminus \ordminus \ordminus }$ 
is of order $\infty$.
The automorphism $g\spbr{7}\colon X\to X$
 obtained by the translation by this section $L\sb{\ordminus \ordminus \ordminus }$
 is a generator associated with the orbit $o_7$.
\begin{figure}
\begin{tikzpicture}[x=.7cm, y=.7cm]
%
\coordinate (e1) at (0, 2);
\coordinate (e2) at (-4, 0);
\coordinate (e3) at (-2, 0);
\coordinate (e4) at (0, 0);
\coordinate (e5) at (2, 0);
\coordinate (e6) at (4, 0);
\coordinate (th) at (2, 2);
\coordinate (z) at (4, 2);
\coordinate (sec) at (-4, 2);
\draw (e1) -- (e4) ;
\draw (e1) -- (th) ;
\draw (th) -- (z) ;
\draw (e2) -- (e3) ;
\draw (e3) -- (e4) ;
\draw (e4) -- (e5) ;
\draw (e5) -- (e6) ;
\draw (sec) -- (e2) ;
\filldraw [fill=white] (e1) circle [radius=3pt];
\filldraw [fill=white] (e2) circle [radius=3pt];
\filldraw [fill=white] (e3) circle [radius=3pt];
\filldraw [fill=white] (e4) circle [radius=3pt];
\filldraw [fill=white] (e5) circle [radius=3pt];
\filldraw [fill=white] (e6) circle [radius=3pt];
\filldraw [fill=lightgray] (th) circle [radius=3pt];
\filldraw [fill=black] (z) circle [radius=3pt];
\filldraw [fill=black] (sec) circle [radius=3pt];
\node at (e1) [above] {$L\sb{\ordplus \ordminus \ordplus }$\mystrutd{5pt}};
\node at (e2) [below] {$L\sb{\ordminus 0 \ordminus }$};
\node at (e3) [below] {$M\sb{2\ordminus \ordplus }$};
\node at (e4) [below] {$L\sb{\ordplus 0 \ordplus }$};
\node at (e5) [below] {$L\sb{\ordplus \ordplus \ordplus }$};
\node at (e6) [below] {$L\sb{0 \ordplus \ordplus }$};
\node at (th) [above] {$L\sb{0 \ordminus \ordplus }$\mystrutd{5pt}};
\node at (z) [above] {$L\sb{\ordminus \ordminus \ordplus }$\mystrutd{5pt}};
\node at (sec) [above] {$L\sb{\ordminus \ordminus \ordminus }$\mystrutd{5pt}};
\end{tikzpicture}
\caption{Configuration for $\phi\spbr{7}$}\label{fig:extraaut7ellfib}
\end{figure}
\subsection{A generator associated with \texorpdfstring{$o_8$}{o8}}\label{subsec:extraaut8}
We consider the Jacobian fibration 
\[
\phi\spbr{8}\colon X\to \PP^1
\]
associated with the configuration whose dual graph is in Figure~\ref{fig:extraaut8ellfib}.
The zero section of $\phi\spbr{8}$
is $L\sb{0 \ordplus \ordplus }$.
This Jacobian fibration 
has two reducible fibers of type $D_8+D_8$.
The Mordell--Weil group of $\phi\spbr{8}$ is isomorphic to $\ZZ\oplus\ZZ/2\ZZ$.
The automorphism
$g\spbr{8}\colon X\to X$
 obtained by the inversion $\iota(\phi\spbr{8})$ of the generic fiber
 is a generator associated with the orbit $o_8$.
\begin{figure}
\begin{tikzpicture}[x=.7cm, y=.7cm]
%
\coordinate (e1) at (0, 2);
\coordinate (e2) at (-2, 0);
\coordinate (e3) at (0, 0);
\coordinate (e4) at (2, 0);
\coordinate (e5) at (4, 0);
\coordinate (e6) at (6, 0);
\coordinate (e7) at (8, 0);
\coordinate (e8) at (10, 0);
\coordinate (th) at (8, 2);
\coordinate (z) at (10, 2);
\draw (e1) -- (e3) ;
\draw (e2) -- (e3) ;
\draw (e3) -- (e4) ;
\draw (e4) -- (e5) ;
\draw (e5) -- (e6) ;
\draw (e6) -- (e7) ;
\draw (e7) -- (e8) ;
\draw (e7) -- (th) ;
\draw (th) -- (z) ;
\filldraw [fill=white] (e1) circle [radius=3pt];
\filldraw [fill=white] (e2) circle [radius=3pt];
\filldraw [fill=white] (e3) circle [radius=3pt];
\filldraw [fill=white] (e4) circle [radius=3pt];
\filldraw [fill=white] (e5) circle [radius=3pt];
\filldraw [fill=white] (e6) circle [radius=3pt];
\filldraw [fill=white] (e7) circle [radius=3pt];
\filldraw [fill=white] (e8) circle [radius=3pt];
\filldraw [fill=lightgray] (th) circle [radius=3pt];
\filldraw [fill=black] (z) circle [radius=3pt];
\node at (e1) [above] {$L\sb{\ordminus \ordminus \ordminus }$\mystrutd{5pt}};
\node at (e2) [below] {$M\sb{2\ordminus \ordplus }$};
\node at (e3) [below] {$L\sb{\ordminus 0 \ordminus }$};
\node at (e4) [below] {$L\sb{\ordminus \ordplus \ordminus }$};
\node at (e5) [below] {$L\sb{0 \ordplus \ordminus }$};
\node at (e6) [below] {$L\sb{\ordplus \ordplus \ordminus }$};
\node at (e7) [below] {$L\sb{\ordplus \ordplus 0 }$};
\node at (e8) [below] {$M\sb{3\ordplus \ordplus }$};
\node at (th) [above] {$L\sb{\ordplus \ordplus \ordplus }$\mystrutd{5pt}};
\node at (z) [above] {$L\sb{0 \ordplus \ordplus }$\mystrutd{5pt}};
\end{tikzpicture}
\caption{Configuration for $\phi\spbr{8}$}\label{fig:extraaut8ellfib}
\end{figure}
\subsection{Generators associated with \texorpdfstring{$o_9$}{o9} and \texorpdfstring{$o_{10}$}{o10}}\label{subsec:extraaut9and10}
%
We consider the Jacobian fibrations 
\[
\phi\spbr{9}\colon X\to \PP^1, \quad
\phi\spbr{10}\colon X\to \PP^1
\]
associated with the two configurations whose dual graphs are in Figures~\ref{fig:extraaut10ellfib}.
Each of these Jacobian fibrations 
has two reducible fibers of type $E_8+E_8$,
and 
their Mordell--Weil groups are isomorphic to $\ZZ$.
The automorphisms $g\spbr{9}\colon X\to X$ and $g\spbr{10}\colon X\to X$ 
 obtained by the inversions $\iota(\phi\spbr{9})$
 and $\iota(\phi\spbr{10})$ of the generic fiber of $\phi\spbr{9}$
 and of $\phi\spbr{10}$
are generators associated with the orbits $o_9$ and $o_{10}$, respectively.
\begin{figure}
\begin{tikzpicture}[x=.7cm, y=.7cm]
\coordinate (e1) at (0, 2);
\coordinate (e2) at (-4, 0);
\coordinate (e3) at (-2, 0);
\coordinate (e4) at (0, 0);
\coordinate (e5) at (2, 0);
\coordinate (e6) at (4, 0);
\coordinate (e7) at (6, 0);
\coordinate (e8) at (8, 0);
\coordinate (th) at (10, 0);
\coordinate (z) at (12, 0);
\draw (e1) -- (e4) ;
\draw (e2) -- (e3) ;
\draw (e3) -- (e4) ;
\draw (e4) -- (e5) ;
\draw (e5) -- (e6) ;
\draw (e6) -- (e7) ;
\draw (e7) -- (e8) ;
\draw (e8) -- (th) ;
\draw (th) -- (z) ;
\filldraw [fill=white] (e1) circle [radius=3pt];
\filldraw [fill=white] (e2) circle [radius=3pt];
\filldraw [fill=white] (e3) circle [radius=3pt];
\filldraw [fill=white] (e4) circle [radius=3pt];
\filldraw [fill=white] (e5) circle [radius=3pt];
\filldraw [fill=white] (e6) circle [radius=3pt];
\filldraw [fill=white] (e7) circle [radius=3pt];
\filldraw [fill=white] (e8) circle [radius=3pt];
\filldraw [fill=lightgray] (th) circle [radius=3pt];
\filldraw [fill=black] (z) circle [radius=3pt];
\node at (e1) [above] {$L\sb{0 \ordminus \ordminus }$\mystrutd{5pt}};
\node at (e2) [below] {$M\sb{2\ordplus \ordminus }$};
\node at (e3) [below] {$L\sb{\ordplus 0 \ordminus }$};
\node at (e4) [below] {$L\sb{\ordplus \ordminus \ordminus }$};
\node at (e5) [below] {$L\sb{\ordplus \ordminus 0 }$};
\node at (e6) [below] {$L\sb{\ordplus \ordminus \ordplus }$};
\node at (e7) [below] {$L\sb{0 \ordminus \ordplus }$};
\node at (e8) [below] {$L\sb{\ordminus \ordminus \ordplus }$};
\node at (th) [below] {$L\sb{\ordminus \ordminus 0 }$};
\node at (z) [below] {$M\sb{3\ordplus \ordplus }$};
\end{tikzpicture}
\vskip .5cm
\begin{tikzpicture}[x=.7cm, y=.7cm]
\coordinate (e1) at (0, 2);
\coordinate (e2) at (-4, 0);
\coordinate (e3) at (-2, 0);
\coordinate (e4) at (0, 0);
\coordinate (e5) at (2, 0);
\coordinate (e6) at (4, 0);
\coordinate (e7) at (6, 0);
\coordinate (e8) at (8, 0);
\coordinate (th) at (10, 0);
\coordinate (z) at (12, 0);
\draw (e1) -- (e4) ;
\draw (e2) -- (e3) ;
\draw (e3) -- (e4) ;
\draw (e4) -- (e5) ;
\draw (e5) -- (e6) ;
\draw (e6) -- (e7) ;
\draw (e7) -- (e8) ;
\draw (e8) -- (th) ;
\draw (th) -- (z) ;
\filldraw [fill=white] (e1) circle [radius=3pt];
\filldraw [fill=white] (e2) circle [radius=3pt];
\filldraw [fill=white] (e3) circle [radius=3pt];
\filldraw [fill=white] (e4) circle [radius=3pt];
\filldraw [fill=white] (e5) circle [radius=3pt];
\filldraw [fill=white] (e6) circle [radius=3pt];
\filldraw [fill=white] (e7) circle [radius=3pt];
\filldraw [fill=white] (e8) circle [radius=3pt];
\filldraw [fill=lightgray] (th) circle [radius=3pt];
\filldraw [fill=black] (z) circle [radius=3pt];
\node at (e1) [above] {$L\sb{0 \ordminus \ordminus }$\mystrutd{5pt}};
\node at (e2) [below] {$M\sb{2\ordminus \ordplus }$};
\node at (e3) [below] {$L\sb{\ordminus 0 \ordminus }$};
\node at (e4) [below] {$L\sb{\ordminus \ordminus \ordminus }$};
\node at (e5) [below] {$L\sb{\ordminus \ordminus 0 }$};
\node at (e6) [below] {$L\sb{\ordminus \ordminus \ordplus }$};
\node at (e7) [below] {$L\sb{0 \ordminus \ordplus }$};
\node at (e8) [below] {$L\sb{\ordplus \ordminus \ordplus }$};
\node at (th) [below] {$L\sb{\ordplus \ordminus 0 }$};
\node at (z) [below] {$M\sb{3\ordminus \ordminus }$};
\end{tikzpicture}
\caption{Configurations for $\phi\spbr{9}$ and $\phi\spbr{10}$}\label{fig:extraaut10ellfib}
\end{figure}
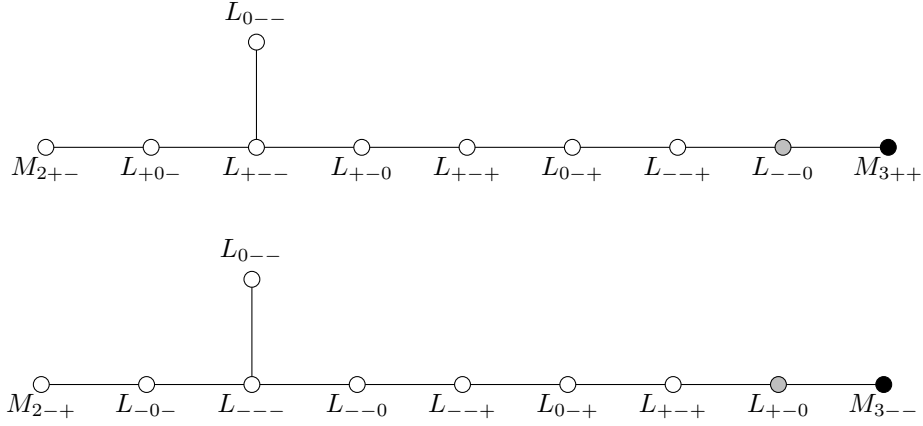
\section{Faces of \texorpdfstring{$D_0$}{D0}}\label{sec:faces}
%
%
We explain methods to enumerate the faces of $D_0$ 
and investigate them in Section~\ref{subsec:facesmu}.
Then we study the  faces of $D_0$ with codimension $2$ in Section~\ref{subsec:faces2}.
This will  lead to a presentation of $\Aut(X)$ in terms of generators and relations in Theorem~\ref{thm:relations}.
In Section~\ref{subsec:orbitCCCC},
we present
an algorithm to enumerate the faces of $\NX$ modulo $\Aut(X)$,
and prove Theorem~\ref{thm:rats}.
\subsection{Enumeration of faces of \texorpdfstring{$D_0$}{D0}}\label{subsec:facesmu}
Let $\Faces{\mu}(D)$ denote the set of faces of codimension $\mu$ of and $\Lts/\SX$-chamber $D$.
\par
The  set $\Faces{\mu}(D_0)$ of faces of $D_0$ with codimension $\mu$
can be calculated 
by induction on $\mu$ as follows.
Suppose that we have $f\in \Faces{\mu}(D_0)$.
Let $\angs{f}_{\RR} \subset \SX\tensor \RR$ denote the minimal linear subspace containing $f$,
so that the supporting linear subspace $\PPP_f$ of $f$ is equal to $\angs{f}_{\RR}\cap \PX$.
Suppose also that we have linear forms $\rho_1, \dots, \rho_k$
of $\angs{f}_{\RR}$ such that
$f$ is defined  in $\PPP_f$  by the linear inequalities $\rho_i\ge 0$ ($i=1, \dots, k$)
and that,
for each $i=1, \dots, k$,  the closed subset $f_i:=f\cap(\rho_i)\sperp$ of $f$ contains a non-empty open subset 
of the hyperplane $(\rho_i)\sperp=\set{x\in \PPP_f}{\rho_i(x)=0}$ of $\PPP_f$.
Then, using the  algorithm of linear programming
(see Algorithm 3.17 of~\cite{Shimada2015} or Section 3.4 of~\cite{DolgachevShimada2020}),
for each $i=1, \dots, k$, 
we can make a list $\rho_{i1}, \dots, \rho_{im_i}$
of linear forms 
of $\angs{f_i}_{\RR} $ such that
$f_i$ is defined in $(\rho_i)\sperp$   by the linear inequalities $\rho_{ij}\ge 0$ ($j=1, \dots, m_i$)
and that,
for each $j=1, \dots, m_i$,  the closed subset $f_{ij}:=f_i\cap(\rho_{ij})\sperp$ of $f_i$ contains a non-empty open subset 
of the hyperplane $(\rho_{ij})\sperp$ of $(\rho_i)\sperp$.
\par
Since the size of the set $\Faces{\mu}(D_0)$ grows rapidly with $\mu$ as is indicated in the table below, 
we stopped
the computation at $\mu=5$.
In the table below, the number of $\Aut(X, D_0)$-orbits in $\Faces{\mu}(D_0)$ is also given.
\[
\begin{array}{c|ccccc}
\mu & 1 & 2 & 3 & 4 & 5 \\
\hline
|\Faces{\mu}(D_0)| & 80 & 1746& 20228& 150750& 793280\mystruth{14pt}\\
\textrm{orbits} & 10 & 128& 1322& 9578 & 49880\rlap{.} \\
\end{array}
\]
\par
For each wall $w$ of $D_0$,
we choose an isometry $g_w\in \OG(\SX, \PX)$ such that
$D_0^{g_w}$ is the $\Lts/\SX$-chamber
adjacent to $D_0$ across the wall $w$
and that
\begin{equation}\label{eq:etagw}
\eta(g_w)\in \{\pm 1\}.
\end{equation}
When $w$ is an inner wall of $D_0$,
any element of $\Adj(w)$ defined by~\eqref{eq:Adj} can be taken as $g_w$.
When $w$ is an outer wall of $D_0$,
we can choose $g_w$ to be the reflection $s_r\colon x\mapsto x+\intf{x, r}r$ 
with respect to the $(-2)$-vector $r$ such that $w=D_0\cap(r)\sperp$.
(Note that $\eta(s_r)=1$.)
\par
Recall that, for a face $f$ of an $\Lts/\SX$-chamber,
we denote by $\DDD(f)$ the set of $\Lts/\SX$-chambers containing $f$.
Suppose that $f\in \Faces{\mu}(D_0)$.
All the elements of $\DDD(f)$ are enumerated by the procedure in Table~{\ref{alg:DDDf}}.
During this procedure,
$\GGG$ is a list of elements of $\OG(\SX, \PX)$
such that  $D_0^g\in \DDD(f)$ for all $g\in \GGG$,
and that, if  $g, g\sprime\in \GGG$ are distinct, then $D_0^g\ne D_0^{g\sprime}$.
Note that the condition $D_0^g\ne D_0^{g\sprime}$ is equivalent to 
 $a_{32}^g\ne a_{32}^{g\sprime}$.
Therefore the second condition for $\GGG$ is equivalent to the condition that
the list $\AAA=[a_{32}^g\mid g\in \GGG]$ is duplicate free.
\begin{table}
{\small 
\begin{algorithmic}
\State Set $\GGG:=[\id], \;\; \AAA:=[a_{32}],\;\; i:=1$.
\While{$i\le |\GGG|$}
      \State $g_i:=\textrm{the $i$th element of $\GGG$}$ (hence $D_0^{g_i}\in \DDD(f)$)
      \State $f\sprime:=f^{({g_i}\inv)}$ (hence $f\sprime\in\Faces{\mu}(D_0)$)
      \For{each wall $w$ of $D_0$}
    		\If{$w\supset f\sprime$}
			\State (Note that $D_0^{g_w {g_i}}\in \DDD(f)$ is adjacent to $D_0^{g_i}$ across the wall $w^{g_i}$ of $D_0^{g_i}$.)
			\If {$a_{32}^{g_w {g_i}}\notin \AAA$}
       				\State {Append $g_w {g_i} $ to $\GGG$, and $a_{32}^{g_w {g_i}}$ to $ \AAA$}
			\EndIf
   		\EndIf
    \EndFor
    \State  Increment  $i$ by $1$
\EndWhile
\State Set $\DDD(f)$ to be $\set{D_0^g}{g\in \GGG}$.
\State
\end{algorithmic}
}
\caption{An algorithm to find all elements of $\DDD(f)$}\label{alg:DDDf}
\end{table}
%
%
 \par
 Then we can compute the set 
\begin{equation*}\label{eq:DDDfVX}
 \DDD(f)\cap \VX=\set{D\in \DDD(f)}{D\subset \NX}.
 \end{equation*}
We can also compute the set 
\begin{equation}\label{eq:CCCf}
 \CCC(f):=\set{C\in \RatsX}{(C)\sperp\supset f}.
 \end{equation}
\begin{remark}\label{rem:gDinAutX}
For $D\in \DDD(f)$, let $g(D)$ denote the element of $\GGG$ such that $D=D_0^{g(D)}$.
Since the choice of $g_w$ satisfies~\eqref{eq:etagw},
 we have $\eta(g(D))\in \{\pm 1\}$.
 In particular,
 we have
 \[
 D\in \DDD(f) \cap \VX \;\; \Longleftrightarrow \;\; g(D) \in \Aut(X)
 \]
 by Proposition~\ref{prop:Torelli}.
\end{remark}
Computing these data for all $f\in \Faces{\mu}(D_0)$
and examining the dual graph of $\CCC(f)$ for each $f$,
 we calculate the subset 
 \[
 \Faces{\mu}(D_0, \tau):=\set{f\in \Faces{\mu}(D_0)}{ \CCC(f)\in \CCCC(\tau)}
 \]
 of $\Faces{\mu}(D_0)$ for each $\ADE$-type $\tau$.
 The group $\Aut(X, D_0)$ acts on $ \Faces{\mu}(D_0, \tau)$.
 The sizes of the set $ \Faces{\mu}(D_0, \tau)$ and 
 the numbers of $\Aut(X, D_0)$-orbits in $ \Faces{\mu}(D_0, \tau)$ are given in Table~\ref{table:FmuD0taus}.
 %
%
%
%
\begin{table} 
\[
\begin{array}{clcc}
\mu & \tau & | \Faces{\mu}(D_0, \tau)| & \textrm{orbits}\mystrutd{5.5pt}\\
\hline 
1 & A_{1} & 24 & 2 \mystruth{11pt}\\ 
2 & 2A_{1} & 276 & 23 \\ 
2 & A_{2} & 32 & 3 \\ 
3 & 3A_{1} & 1936 &126 \\ 
3 & A_{1}+A_{2} & 592 &37 \\ 
3 & A_{3} & 712 &45 
\end{array}
\qquad \qquad 
\begin{array}{clcc}
\mu & \tau & | \Faces{\mu}(D_0, \tau)| & \textrm{orbits}\mystrutd{5.5pt}\\
\hline 
4 & 4A_{1} & 8802 & 572 \mystruth{11pt} \\ 
4 & 2A_{1}+A_{2} &5056 &322 \\ 
4 & A_{1}+A_{3} & 10760 & 673 \\ 
4 & 2A_{2} & 384 & 32 \\ 
4 & A_{4} & 96 &8 \\ 
4 & D_{4} & 160 &10 
\end{array}
%
\]
\caption{Sizes of $ \Faces{\mu}(D_0, \tau)$ and the number of $\Aut(X, D_0)$-orbits}\label{table:FmuD0taus}
\end{table}
\subsection{Faces of codimension \texorpdfstring{$2$}{2}}\label{subsec:faces2}
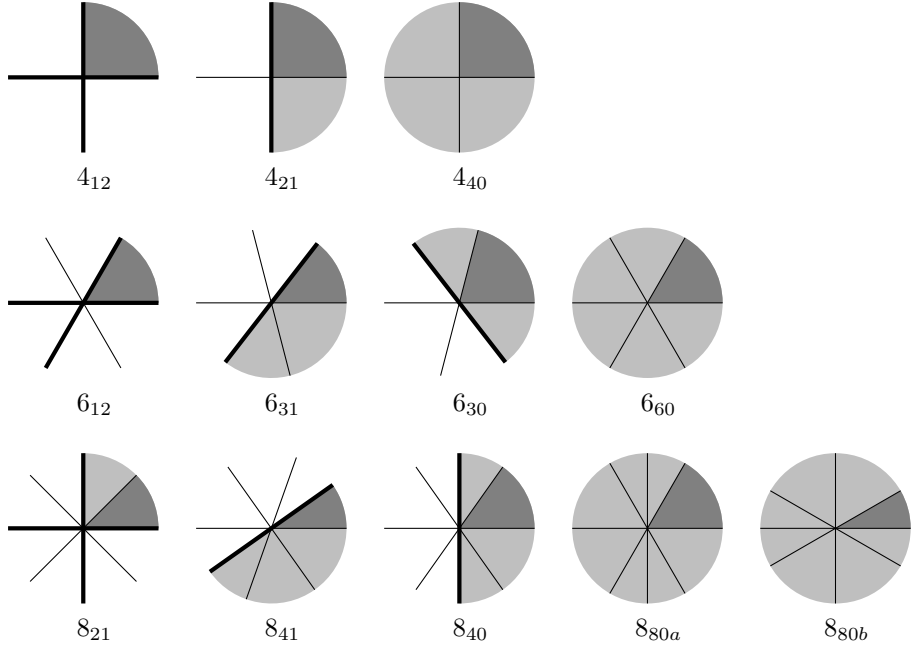
\begin{figure}
\begin{tikzpicture}[x=.5cm,y=.5cm]
\begin{scope}[xshift=0cm, yshift=0cm]
\coordinate (P0) at (2,0);
\coordinate (P1) at (1.4142,1.4142);
\coordinate (P2) at (0,2);
\coordinate (P3) at (-1.4142,1.4142);
\coordinate (P4) at (-2,0);
\coordinate (P5) at (-1.4142,-1.4142);
\coordinate (P6) at (0,-2);
\coordinate (P7) at (1.4142,-1.4142);
\fill [fill=gray, opacity=1] (P0) arc (0:45:2) --(0,0)--cycle;
\fill [fill=gray, opacity=.5] (P0) arc (0:90:2) --(0,0)--cycle;
\draw [ultra thick] (P0)--(P4);
\draw [thin] (P1)--(P5);
\draw [ultra thick] (P2)--(P6);
\draw [thin] (P3)--(P7);
\node at (.3, -2.7) {$8_{21}$};
\end{scope}
\begin{scope}[xshift=2.5cm, yshift=0cm]
\coordinate (P0) at (2,0);
\coordinate (P1) at (1.632993162, 1.154700538);
\coordinate (P2) at (0.6666666674, 1.885618083);
\coordinate (P3) at (-1.154700538, 1.632993162);
\coordinate (P4) at (-2,0);
\coordinate (P5) at (-1.632993162, -1.154700538);
\coordinate (P6) at (-0.6666666674, -1.885618083);
\coordinate (P7) at (1.154700538, -1.632993162);
\fill [fill=gray, opacity=1] (P0) arc (0:35.26438968:2) --(0,0)--cycle;
\fill [fill=gray, opacity=.5] (P5) arc (-144.7356103:35.26438968:2) --(0,0)--cycle;
\draw [thin] (P0)--(P4);
\draw [ultra thick] (P1)--(P5);
\draw [thin] (P2)--(P6);
\draw [thin] (P3)--(P7);
\node at (.3, -2.7) {$8_{41}$};
\end{scope}
\begin{scope}[xshift=5cm, yshift=0cm]
\coordinate (P0) at (2,0);
\coordinate (P1) at (1.154700539, 1.632993162);
\coordinate (P2) at (0,2);
\coordinate (P3) at (-1.154700539, 1.632993162);
\coordinate (P4) at (-2,0);
\coordinate (P5) at (-1.154700539, -1.632993162);
\coordinate (P6) at (0,-2);
\coordinate (P7) at (1.154700539, -1.632993162);
\fill [fill=gray, opacity=1] (P0) arc (0:54.73561030:2) --(0,0)--cycle;
\fill [fill=gray, opacity=.5] (P6) arc (-90:90:2) --(0,0)--cycle;
\draw [thin] (P0)--(P4);
\draw [thin] (P1)--(P5);
\draw [ultra thick] (P2)--(P6);
\draw [thin] (P3)--(P7);
\node at (.3, -2.7) {$8_{40}$};
\end{scope}
\begin{scope}[xshift=7.5cm, yshift=0cm]
\coordinate (P0) at (2,0);
\coordinate (P1) at (1.000000000, 1.732050807);
\coordinate (P2) at (0,2);
\coordinate (P3) at (-1.000000000, 1.732050807);
\coordinate (P4) at (-2,0);
\coordinate (P5) at (-1.000000000, -1.732050807);
\coordinate (P6) at (0,-2);
\coordinate (P7) at (1.000000000, -1.732050807);
\fill [fill=gray, opacity=1] (P0) arc (0:60:2) --(0,0)--cycle;
\fill [fill=gray, opacity=.5] (P0) arc (0:360:2) --cycle;
\draw [thin] (P0)--(P4);
\draw [thin] (P1)--(P5);
\draw [thin] (P2)--(P6);
\draw [thin] (P3)--(P7);
\node at (.3, -2.7) {$8_{80a}$};
\end{scope}
\begin{scope}[xshift=10cm, yshift=0cm]
\coordinate (P0) at (2,0);
\coordinate (P1) at (1.732050807, 1.000000000);
\coordinate (P2) at (0,2);
\coordinate (P3) at (-1.732050807, 1.000000000);
\coordinate (P4) at (-2,0);
\coordinate (P5) at (-1.732050807, -1.000000000);
\coordinate (P6) at (0,-2);
\coordinate (P7) at (1.732050807, -1.000000000);
\fill [fill=gray, opacity=1] (P0) arc (0:30:2) --(0,0)--cycle;
\fill [fill=gray, opacity=.5] (P0) arc (0:360:2) --cycle;
\draw [thin] (P0)--(P4);
\draw [thin] (P1)--(P5);
\draw [thin] (P2)--(P6);
\draw [thin] (P3)--(P7);
\node at (.3, -2.7) {$8_{80b}$};
\end{scope}
\begin{scope}[xshift=0.0cm, yshift=3cm]
\coordinate (P0) at (2,0);
\coordinate (P1) at (1,1.7321);
\coordinate (P2) at (-1,1.7321);
\coordinate (P3) at (-2,0);
\coordinate (P4) at (-1,-1.7321);
\coordinate (P5) at (1,-1.7321);
\fill [fill=gray, opacity=1] (P0) arc (0:60:2) --(0,0)--cycle;
\fill [fill=gray, opacity=.5] (P0) arc (0:60:2) --(0,0)--cycle;
\draw [ultra thick] (P0)--(P3);
\draw [ultra thick] (P1)--(P4);
\draw [thin] (P2)--(P5);
\node at (.3, -2.7) {$6_{12}$};
\end{scope}
\begin{scope}[xshift=2.5cm, yshift=3cm]
\coordinate (P0) at (2,0);
\coordinate (P1) at (1.224744872, 1.581138830);
\coordinate (P2) at (-0.4999999982, 1.936491674);
\coordinate (P3) at (-2,0);
\coordinate (P4) at (-1.224744872, -1.581138830);
\coordinate (P5) at (0.4999999982, -1.936491674);
\fill [fill=gray, opacity=1] (P0) arc (0:52.23875607:2) --(0,0)--cycle;
\fill [fill=gray, opacity=.5] (P4) arc (-127.7612439:52.23875607:2) --(0,0)--cycle;
\draw [thin] (P0)--(P3);
\draw [ultra thick] (P1)--(P4);
\draw [thin] (P2)--(P5);
\node at (.3, -2.7) {$6_{31}$};
\end{scope}
\begin{scope}[xshift=5.0cm, yshift=3cm]
\coordinate (P0) at (2,0);
\coordinate (P1) at (0.4999999994, 1.936491673);
\coordinate (P2) at (-1.224744872, 1.581138830);
\coordinate (P3) at (-2,0);
\coordinate (P4) at (-0.4999999994, -1.936491673);
\coordinate (P5) at (1.224744872, -1.581138830);
\fill [fill=gray, opacity=1] (P0) arc (0:75.52248782:2) --(0,0)--cycle;
\fill [fill=gray, opacity=.5] (P5) arc (-52.23875607:127.7612439:2) --(0,0)--cycle;
\draw [thin] (P0)--(P3);
\draw [thin] (P1)--(P4);
\draw [ultra thick] (P2)--(P5);
\node at (.3, -2.7) {$6_{30}$};
\end{scope}
\begin{scope}[xshift=7.5cm, yshift=3cm]
\coordinate (P0) at (2,0);
\coordinate (P1) at (1,1.7321);
\coordinate (P2) at (-1,1.7321);
\coordinate (P3) at (-2,0);
\coordinate (P4) at (-1,-1.7321);
\coordinate (P5) at (1,-1.7321);
\fill [fill=gray, opacity=1] (P0) arc (0:60:2) --(0,0)--cycle;
\fill [fill=gray, opacity=.5] (P0) arc (0:360:2) -- cycle;
\draw [thin] (P0)--(P3);
\draw [thin] (P1)--(P4);
\draw [thin] (P2)--(P5);
\node at (.3, -2.7) {$6_{60}$};
\end{scope}
\begin{scope}[xshift=0.0cm, yshift=6cm]
\coordinate (P0) at (2,0);
\coordinate (P1) at (0,2);
\coordinate (P2) at (-2,0);
\coordinate (P3) at (0,-2);
\fill [fill=gray, opacity=1] (P0) arc (0:90:2) --(0,0)--cycle;
\fill [fill=gray, opacity=.5] (P0) arc (0:90:2) --(0,0)--cycle;
\draw [ultra thick] (P0)--(P2);
\draw [ultra thick] (P1)--(P3);
\node at (.3, -2.7) {$4_{12}$};
\end{scope}
\begin{scope}[xshift=2.5cm, yshift=6cm]
\coordinate (P0) at (2,0);
\coordinate (P1) at (0,2);
\coordinate (P2) at (-2,0);
\coordinate (P3) at (0,-2);
\fill [fill=gray, opacity=1] (P0) arc (0:90:2) --(0,0)--cycle;
\fill [fill=gray, opacity=.5] (P3) arc (-90:90:2) --(0,0)--cycle;
\draw [thin] (P0)--(P2);
\draw [ultra thick] (P1)--(P3);
\node at (.3, -2.7) {$4_{21}$};
\end{scope}
\begin{scope}[xshift=5cm, yshift=6cm]
\coordinate (P0) at (2,0);
\coordinate (P1) at (0,2);
\coordinate (P2) at (-2,0);
\coordinate (P3) at (0,-2);
\fill [fill=gray, opacity=1] (P0) arc (0:90:2) --(0,0)--cycle;
\fill [fill=gray, opacity=.5] (P0) arc (0:360:2) --cycle;
\draw [thin] (P0)--(P2);
\draw [thin] (P1)--(P3);
\node at (.3, -2.7) {$4_{40}$};
\end{scope}
\end{tikzpicture}
\caption{Types of codimension-$2$ faces of $D_0$}\label{fig:Faces2}
\end{figure}
We examine the set $\Faces{2}(D_0)$.
The faces in $\Faces{2}(D_0)$ are classified into $12$ types,
which are illustrated in Figure~\ref{fig:Faces2}.
We choose a general point $p$ of $f$,
and consider a small disk $\varDelta$ centered at $p$ 
within a $2$-dimensional linear subspace in $\PX$ intersecting $f$ at $p$ orthogonally.
In Figure~\ref{fig:Faces2}, 
we depict 
the intersections of $\varDelta$ with the $\Lts/\SX$-chambers $D\in \DDD(f)$ containing $f$.
The dark gray sector is $\varDelta\cap D_0$,
and the dark and light gray sectors are $\varDelta\cap D$ with $D\subset \NX$.
Thick lines indicate $\varDelta\cap (C)\sperp$, where $C\in \RatsX$ is a smooth rational curves such that $(C)\sperp\supset f$.
\par
The type is denoted as $\sigma:=n_{lr}$,
where $n$ is the size of the set $\DDD(f)$, 
$l$ is the size of $\DDD(f)\cap \VX$,
and $r$ is the number of $C\in \CCC(f)$ such that $(C)\sperp\cap D_0$ is a wall of $D_0$.
The size of the set $\Faces{2}(D_0)_{\sigma}$ of faces of type $\sigma$ is 
listed in the second column of Table~\ref{table:angles}.
The group $\Aut(X, D_0)$ acts on $\Faces{2}(D_0)_{\sigma}$.
The numbers of $\Aut(X, D_0)$-orbits in 
$\Faces{2}(D_0)_{\sigma}$ are also presented.
\par
We index the $\Lts/\SX$-chambers $D\in \DDD(f)$
as $D_0, \dots, D_{n-1}$,
starting $D_0$ and proceeding around $f$. 
Then the dihedral angle $\theta_i$ of $D_i$ at $p$
for $i=0, 1, \dots, n/2-1$ are given in the third column of Table~\ref{table:angles}
by means of the rational number
\[
(\cos\theta_i)^2=\frac{\intf{v,v\sprime}^2}{\intf{v,v} \intf{v\sprime,v\sprime}},
\]
where $(v)\sperp\cap D_i$ and $(v\sprime)\sperp\cap D_i$ are the two walls of $D_i$ containing $f$.
The fourth column of Table~\ref{table:angles}
provides all possible pairs $kk\sprime=\{k, k\sprime\}$
of the indexes of orbits $o_{k}$ and $o_{k\sprime}$ to which the walls $(v)\sperp\cap D_0$ and $(v\sprime)\sperp\cap D_0$ 
of $D_0$ containing $f$ belong.
Here the index $10$ of $o_{10}$ 
is denoted by $t$ so that $1t$ and $2t$ mean $\{1, 10\}$ and $\{2, 10\}$, respectively.
\begin{table}
\[
\begin{array}{clllllll} 
\textrm{type\;} \sigma &| \Faces{2}(D_0)_{\sigma}| & \textrm{orbits} &\rlap{$(\cos\theta_i)^2$} && && \textrm{pairs of walls}
\mystrutd{5.5pt}\\
\hline
4_{12} &244&21 & 0 & 0 &&       & 11,12,22 \mystruth{11pt} \\
4_{21} &1096& 73& 0 & 0 &&       & 13,\dots,19, 1t, 23, \dots, 29, 2t\\
4_{40} &88& 8& 0 & \rlap{$0$}\phantom{1/16} &&        & 34, 35, 36, 37\\
6_{12} &32& 3& 1/4 & 1/4 & 1/4 & \phantom{1/16}    & 12, 22 \\
6_{31} &8& 1& 3/8 & 3/8 & 1/16 &    & 13\\
6_{30} &4& 1& 1/16 & 3/8 & 3/8 &    & 33\\
6_{60} &2& 1& 1/4 & 1/4 & 1/4 &     & 33\\
8_{21} &32& 2& 1/2 & 1/2 & 1/2 & 1/2   & 14, 24 \\
8_{41} &112& 8& 2/3 & 2/3 & 1/3 & 1/3  & 15,16, 25, 26, 27\\
8_{40} &112& 8& 1/3 & 2/3 & 2/3 & 1/3  & 45, 46,4 7\\
8_{80a} &8& 1& 1/4 & 3/4 & 3/4 & 1/4   & 38\\
8_{80b} &8& 1& 3/4 & 1/4 & 1/4 & 3/4   & 34
\end{array}
\]
\vskip 2mm
\caption{Data of codimension-$2$ faces of $D_0$}\label{table:angles}
\end{table}
%
\par
Let $w\in \Faces{1}(D_0)$ be a wall of $D_0$ that belongs to the orbit $o_i$.
The numbers and types of codimension-$2$ faces $f\in \Faces{2}(D_0)$ such that $f\subset w$
are given in Table~\ref{table:wandf}.
%
%
\begin{table}
\[
\begin{array}{ccl}
\textrm{orbit} & \textrm{total number} &\textrm{types and numbers} \mystrutd{4pt} \\ 
\hline 
o_{1} & 77 &(4_{12})^{21}(4_{21})^{45}(6_{12})^{2}(6_{31})^{1}(8_{21})^{2}(8_{41})^{6} \mystruth{11pt} \\ 
o_{2} & 74 &(4_{12})^{20}(4_{21})^{46}(6_{12})^{3}(8_{21})^{1}(8_{41})^{4} \\ 
o_{3} & 53 &(4_{21})^{22}(4_{40})^{22}(6_{31})^{2}(6_{30})^{2}(6_{60})^{1}(8_{80a})^{2}(8_{80b})^{2} \\ 
o_{4} & 42 &(4_{21})^{20}(4_{40})^{3}(8_{21})^{4}(8_{40})^{14}(8_{80b})^{1} \\ 
o_{5} & 32 &(4_{21})^{19}(4_{40})^{3}(8_{41})^{5}(8_{40})^{5} \\ 
o_{6} & 32 &(4_{21})^{19}(4_{40})^{3}(8_{41})^{5}(8_{40})^{5} \\ 
o_{7} & 30 &(4_{21})^{20}(4_{40})^{2}(8_{41})^{4}(8_{40})^{4} \\ 
o_{8} & 22 &(4_{21})^{20}(8_{80a})^{2} \\ 
o_{9} & 19 &(4_{21})^{19} \\ 
o_{10} & 19 &(4_{21})^{19} \\
\end{array}
\]
\vskip 5pt
\caption{Faces of codimension $2$ that bound a wall}\label{table:wandf}
\end{table}
\subsection{Orbit decomposition of \texorpdfstring{$\CCCC(\tau)$}{Ctau} by \texorpdfstring{$\Aut(X)$}{Aut(X)}}\label{subsec:orbitCCCC}
We present a method to calculate the orbit decomposition of the action of $\Aut(X)$
on the set $\CCCC(\tau)$ 
of $\ADE$-configurations of smooth rational curves of type $\tau$.
This method requires the sets $\Faces{\mu}(D_0)$ and $\Faces{\mu+1}(D_0)$ of all faces of codimension $\mu$ and $\mu+1$,
where $\mu$ is the Milnor number of $\tau$.
From the set $\Faces{1}(D_0)$, \dots, $\Faces{5}(D_0)$,
we obtain the orbit decomposition 
of $\CCCC(\tau)$ for all $\ADE$-types $\tau$ with $\mu\le 4$, proving Theorem~\ref{thm:rats}.
\par
Let $\CCC=\{C_1, \dots, C_{\mu}\}$ be an element of $\CCCC(\tau)$.
We define
\[
\PPPC=(C_1)\sperp\cap \dots \cap (C_{\mu})\sperp,
\]
which is a linear subspace of codimension $\mu$ in $\PX$.
\begin{proposition}\label{prop:PPPCgivesface}
The intersection $\PPPC\cap\NX$ is a face of codimension $\mu$ of the $\SX/\SX$-chamber $\NX$.
\end{proposition}
\begin{proof}
Since $\intf{a, C_i}>0$ for any ample class $a$, 
it follows that $\PPPC$ is disjoint from the interior of $\NX$.
It suffices to show that there exists a point $p$ on $\PPPC$ such that
$\intf{p,C\sprime}>0$ holds for any smooth rational curve $C\sprime$ on $X$ that is not a member of $\CCC$.
We define $m_{ij}:=\intf{C_i, C_j}$, and consider the $\mu\times \mu$ matrix $M:=(m_{ij})$,
which is the Gram matrix of the negative-definite root lattice of type $\tau$
with respect to the standard basis.
It is well known that every entry of the inverse matrix $M\inv$ is $\le 0$.
Fixing an ample class $a$, 
we define $t_1, \dots, t_{\mu}\in \QQ$ by 
\[
\left[\begin{array}{c} t_1 \\ \vdots \\ t_{\mu} \end{array}\right]=
M\inv \left[\begin{array}{c} \intf{a, C_1} \\ \vdots \\\intf{a, C_{\mu}} \end{array}\right].
\]
Since $\intf{a, C_i}>0$ for $i=1, \dots, \mu$,
we have $t_{i}\le 0$ for $i=1, \dots, \mu$.
We put
\[
p:=a-(t_1 C_1+ \dots +t_{\mu} C_{\mu}).
\]
Then we have $\intf{p,C_i}=0$ for $i=1, \dots, \mu$,
and 
\[
\intf{p, p}=\intf{p, a}=\intf{a, a}-(t_1 \intf{C_1, a}+ \dots +t_{\mu} \intf{C_{\mu}, a})>0.
\]
Thus we have $p\in \PPPC$.
For any $C\sprime \in \RatsX$
such that $C\sprime\notin\CCC$, 
we have $\intf{a, C\sprime}>0$ and $\intf{C_i, C\sprime}\ge 0$ for $i=1, \dots, \mu$.
Hence $\intf{p,C\sprime}>0$ holds.
Therefore a small neighborhood of $p$ in $\PPPC$ is contained in $\PPPC\cap \NX$.
\end{proof}
Let $[\CCC]\sperp$ be the orthogonal complement of 
the sublattice $[\CCC]$ of $\SX$ generated by the elements of $\CCC$.
Then $[\CCC]\sperp$ is a primitive sublattice of $\SX$ 
 with signature $(1, 18-\mu)$,
and 
\[
\PPPC:=\PX\cap ([\CCC]\sperp\tensor\RR)
\]
is a positive cone of $[\CCC]\sperp$.
The tessellation of $\PX$ by the $\SX/\SX$-chambers induces 
a tessellation of $\PPPC$ by the $\SX/[\CCC]\sperp$-chambers,
and $\PPPC\cap\NX$ is one of these $\SX/[\CCC]\sperp$-chambers.
On the other hand, since $\SX$ is embedded primitively into $\Lts$ in Section~\ref{subsec:Borcherds},
we can regard $[\CCC]\sperp$ as a primitive sublattice of $\Lts$,
and every $\SX/[\CCC]\sperp$-chamber is tessellated by $\Lts/[\CCC]\sperp$-chambers.
Note that every $\Lts/[\CCC]\sperp$-chamber is of the form $\PPPC\cap D$,
where $D$ is an $\Lts/\SX$-chamber and $\PPPC\cap D$ is a face of $D$ 
with supporting linear subspace $\PPPC$.
The algorithm below is
Borcherds' method applied to 
the tessellation of the $\SX/[\CCC]\sperp$-chamber 
$\PPPC\cap\NX$ by $\Lts/[\CCC]\sperp$-chambers $\PPPC\cap D$.
\par
We consider the map
\begin{equation}\label{eq:inimap}
\Faces{\mu}(D_0, \tau)\;\to\; \CCCC(\tau)
\end{equation}
given by $f\mapsto \CCC(f)$,
where $\CCC(f)$ is defined by~\eqref{eq:CCCf}.
Let $\CCC$ be an arbitrary element of $\CCCC(\tau)$.
By Proposition~\ref{prop:PPPCgivesface},
there exists an element $D$ of $ \VX$ 
such that $f_D:=\PPPC\cap D$ is a face of $D$
with supporting linear subspace $\PPPC$. 
Since $\Aut(X)$ acts on $\VX$ transitively,
there exists an automorphism $g\in \Aut(X)$ such that
$D^g=D_0$.
Then we have $f_D^g\in \Faces{\mu}(D_0, \tau)$,
and the mapping~\eqref{eq:inimap} maps $f_D^g$ to $\CCC^g$.
Therefore
 the mapping~\eqref{eq:inimap} induces a surjective map
\begin{equation}\label{eq:surjmap}
\Faces{\mu}(D_0, \tau)\;\surj\; \CCCC(\tau)/\Aut(X).
\end{equation}
\par
Fix an element $\CCC$ of $\CCCC(\tau)$.
We define
\begin{eqnarray*}
\tilVC&:=&\set{D\in \VX}{\textrm{$\PPPC\cap D$ is a face of $D$ with supporting linear subspace $\PPPC$}}\\
&=&\set{D\in \VX}{\textrm{$\PPPC\cap D$ contains a nonempty open subset of $\PPPC$}},\\
\VC&:=&\set{\PPPC\cap D}{D\in \tilVC}.
\end{eqnarray*}
Then the stabilizer subgroup 
\[
\Aut(X, \CCC):=\set{g\in \Aut(X)}{\textrm{$g$ preserves $\CCC$}}
\]
of $\CCC$
acts on $\tilVC$.
\begin{remark}
The mapping $D\mapsto \PPPC\cap D$ from $\tilVC$ to $\VC$
may not be a bijection.
For example, when $\tau=2A_1$, 
if $\PPPC\cap D_0$ is a codimension-$2$ face of $D_0$ with type $8_{21}$ (see Figure~\ref{fig:Faces2}),
then there exists an $\Lts/\SX$-chamber $D\sprime$ such that  $D\sprime\ne D_0$  and that 
$\PPPC\cap D\sprime=\PPPC\cap D_0$.
\end{remark}
For $D\in \tilVC$, there exists an automorphism $g(D)\in \Aut(X)$ 
such that $D=D_0^{g(D)}$.
Then 
\[
(\PPPC\cap D)^{(g(D)\inv)}=\PPP_{\CCC^{(g(D)\inv)}}\cap D_0
\]
is a face of $D_0$ of codimension $\mu$, and 
this face is a member of $\Faces{\mu}(D_0, \tau)$.
Recall that $\Aut(X, D_0)$ acts on $\Faces{\mu}(D_0, \tau)$.
The choice of $g(D)\in \Aut(X)$ such that $D=D_0^{g(D)}$ is unique up to the left multiplication of elements of $\Aut(X, D_0)$.
More generally, if $D\sprime \in \tilVC$
is equal to $D^{\gamma}$ for an element $\gamma \in \Aut(X, \CCC)$,
then there exists an element $h\in \Aut(X, D_0)$ such that
$h g(D) \gamma =g(D\sprime)$.
Since 
\[
(\PPPC\cap D\sprime)^{(g(D\sprime)\inv)}=(\PPPC\cap D\sprime)^{(\gamma\inv g(D)\inv h\inv)}=(\PPPC\cap D)^{(g(D)\inv h\inv)},
\] 
the mapping $D\mapsto (\PPPC\cap D)^{(g(D)\inv)}$ induces a mapping
\[
\Phi_{\CCC}\;\;\colon\;\; \tilVC \to \Faces{\mu}(D_0, \tau)/\Aut(X, D_0)
\]
that factors through the natural projection
\[
\tilVC\surj \tilVC/\Aut(X, \CCC).
\]
\begin{proposition}
For $\CCC\in \CCCC(\tau)$ and $\CCC\sprime\in \CCCC(\tau)$, the following are equivalent:
\begin{itemize}
\item[{\rm (i)}] $\CCC$ and $\CCC\sprime$ are in the same $\Aut(X)$-orbit. 
\item[{\rm (ii)}] The images of $\Phi_{\CCC}$ and of $\Phi_{\CCC\sprime}$ are the same. 
\item[{\rm (iii)}] The images of $\Phi_{\CCC}$ and of $\Phi_{\CCC\sprime}$ have nonempty intersection. 
\end{itemize}
\end{proposition}
\begin{proof}
Suppose that $\CCC\sprime=\CCC^{\alpha}$ for some $\alpha\in \Aut(X)$.
We have ${\PPP_{\CCC\sprime}}\spar{\alpha\inv}=\PPP_{\CCC}$.
For $D\in \tilV_{\CCC}$, we have $D^{\alpha} \in\tilV_{\CCC\sprime}$
and $\Phi_{\CCC}(D)= \Phi_{\CCC\sprime}(D^{\alpha})$,
because $g(D^{\alpha})=hg(D)\alpha$ for some $h\in \Aut(X, D_0)$.
Thus, the image of $\Phi_{\CCC}$ is contained in the image of $\Phi_{\CCC\sprime}$.
Therefore (i) implies (ii).
\par
The implication (ii) $\Longrightarrow$ (iii) is obvious.
The implication (iii) $\Longrightarrow$ (i) follows from the fact that,
if $f$ is an element of the $\Aut(X, D_0)$-orbit $\Phi_{\CCC}(D)$,
then $\CCC$ and $\CCC(f)$ are in the same $\Aut(X)$-orbit,
because the supporting linear subspace $\PPP_{\CCC}$ of the face $\PPP_{\CCC}\cap D$
of $D$ is mapped to the supporting linear subspace $\PPP_{\CCC(f)}$ of the face $f$
of $D_0$ by an element of $\Aut(X)$.
\end{proof}
As is seen from the surjectivity of the map \eqref{eq:surjmap}, 
every $\Aut(X)$-orbit in $\CCCC(\tau)$ contains
a configuration $\CCC(f)$ for some $f\in \Faces{\mu}(D_0, \tau)$.
Hence, calculating the images of $\Phi_{\CCC(f)}$ for all $f\in \Faces{\mu}(D_0, \tau)$,
we obtain the orbit decomposition of $\CCCC(\tau)$ by $\Aut(X)$.
\par
The images of $\Phi_{\CCC(f)}$ for faces $f\in \Faces{\mu}(D_0, \tau)$ are computed as follows.
The idea is to calculate the orbit decomposition of $\tilVC$ 
under the action of $\Aut(X, \CCC)$ 
by Borcherds' method.
For simplicity, we put
\[
[\FFF]:=\Faces{\mu}(D_0, \tau)/\Aut(X, D_0), 
\]
and for $f\in \Faces{\mu}(D_0, \tau)$, let $[f]\in [\FFF]$ denote the $\Aut(X, D_0)$-orbit containing $f$.
We construct a graph whose set of nodes is $[\FFF]$
and whose set of edges is defined as follows.
Let $f$ be an element of $\Faces{\mu}(D_0, \tau)$.
We have $D_0\in \tilV_{\CCC(f)}$,
and $\Phi_{\CCC(f)}$ maps $D_0$ to $[f]$,
as $\PPP_{\CCC(f)}\cap D_0=f$.
Using the set $\Faces{\mu+1}(D_0)$, we compute the set
\[
\Faces{\mu+1}\angs{f}:=\set{\varphi\in \Faces{\mu+1}(D_0)}{f\supset \varphi},
\]
which is the set of all walls of the $\Lts/[\CCC(f)]\sperp$-chamber 
$f=\PPP_{\CCC(f)}\cap D_0$.
For each $\varphi\in \Faces{\mu+1}\angs{f}$, we calculate the set $\DDD(\varphi)$ 
and subsequently compute the subset
\[
\DDD(f, \varphi):=\tilV_{\CCC(f)}\cap \DDD(\varphi).
\]
For an $\Lts/\SX$-chamber $D$,
we have $D\in \DDD(f, \varphi)$ if and only if 
$\PPP_{\CCC(f)}\cap D$ is an $\Lts/[\CCC(f)]\sperp$-chamber 
that is contained in $\PPP_{\CCC(f)}\cap \NX$
and that is either equal to $f$
or adjacent to $f$ across the wall $\varphi$.
For each $D\in \DDD(f, \varphi)$,
we choose an automorphism $g(D)\in \Aut(X)$
such that $D=D_0^{g(D)}$,
and consider the face 
\[
f\sprime:=(\PPP_{\CCC(f)})^{(g(D)\inv)}\cap D_0.
\]
Then $f\sprime$ is an element of $\Faces{\mu}(D_0, \tau)$ 
and $[f\sprime]\in [\FFF]$ does not depend on the choice of $g(D)$.
If $[f\sprime]\ne [f]$, we connect the nodes $[f]$ and $[f\sprime]$ by an edge.
Performing this procedure 
for all $[f]\in [\FFF]$ , $\varphi\in \Faces{\mu+1}\angs{f}$ and $D\in \DDD(f, \varphi)$,
we obtain a graph structure on $[\FFF]$.
\par
Since any pair of elements of $V_{\CCC(f)}$ 
(that is, any pair of $\Lts/[\CCC(f)]\sperp$-chambers contained in $\PPP_{\CCC(f)}\cap N_X$) 
is connected by 
the adjacency relation of $\Lts/[\CCC(f)]\sperp$-chambers,
it follows that the image of $\Phi_{\CCC(f)}$ is 
precisely the connected component of the graph $[\FFF]$
containing the node $[f]$.
\par
Therefore the number of $\Aut(X)$-orbits in $\CCCC(\tau)$ is equal to 
the number of connected components of the graph $[\FFF]$.
Using this method, we obtain a proof of Theorem~\ref{thm:rats}.
%
%
\begin{example}
We consider 
the case where $\tau=2A_1$.
The number of nodes of 
the graph $[\FFF]=\Faces{2}(D_0, 2A_1)/\Aut(X, D_0)$ is $23$,
and this graph has two connected components $[\FFF]_{21}$ and $[\FFF]_{2}$
of size $21$ and $2$, respectively.
Every face in the connected component $[\FFF]_{21}$ is of type $4_{12}$,
whereas every face in the connected component $[\FFF]_{2}$ is of type $8_{21}$.
Hence $\Aut(X)$ partitions 
the set $\CCCC(2A_1)$
of pairs of disjoint smooth rational curves 
 into two orbits
$\CCCC(2A_1)_{21}$ and $\CCCC(2A_1)_{2}$, which 
correspond to 
the connected components $[\FFF]_{21}$ and $[\FFF]_{2}$, respectively.
\par
The linear subspace 
\[
(L_{\ordminus \ordminus \ordminus})\sperp \cap (L_{\ordminus 0 \ordplus})\sperp 
\]
of $\PX$ is a supporting linear subspace of a face of $D_0$ with type $4_{12}$.
Hence the pair $\{L_{\ordminus \ordminus \ordminus}, \, L_{\ordminus 0 \ordplus}\}$
is a member of the orbit $\CCCC(2A_1)_{21}$.
\par
Let $L\sprime$ be the image of the smooth rational curve $L_{\ordplus\ordminus\ordminus}$ by 
 the automorphism $g\spbr{4}\in \Aut(X)$ of order $2$ given in Section~\ref{subsec:extraaut4}.
Then the linear subspace 
\[
(L_{\ordplus\ordminus\ordminus})\sperp \cap (L\sprime)\sperp 
\]
of $\PX$ 
is a supporting linear subspace of a face of $D_0$ with type $8_{21}$.
Hence the pair $\{ L_{\ordplus\ordminus\ordminus}, \, L\sprime \}$
is a member of the orbit $\CCCC(2A_1)_{2}$.
(Note that, for every face $f$ of type $8_{21}$, there exists a wall in the orbit $o_4$ passing through $f$. 
See Table~\ref{table:angles}.)
\end{example}
\section{Relations}\label{sec:relations}
It is well known that
a  set of defining relations of a group acting on a space of constant curvature can be derived 
 from the shape of a fundamental domain.
See, for example, the survey~\cite{VinbergShvartsman1993}.
In our current setting involving $\Aut(X)$ and $D_0$, however, 
we cannot apply this theory directly because of the following reasons.
First, the cone $D_0$ is not a fundamental domain;
it has a non-trivial automorphism group $\Aut(X, D_0)$.
Second, not all codimension-$2$ faces contribute to relations
as $D_0$ has outer faces.
Hence we provide a detailed explanation how to obtain a set of defining relations for $\Aut(X)$ 
from $D_0$. 
The main result of this section is Theorem~\ref{thm:relations}.
\begin{remark}
In~\cite{DolgachevShimada2020},
we have treated the case where $\Aut(X, D_0)$ is trivial.
\end{remark}
For simplicity,
we put
\[
\Genszero:=\Aut(X, D_0).
\]
Recall from~\eqref{eq:Adj}
that we have defined 
$\Adj(w)$ for each inner wall $w$ of $D_0$.
For $h\in \Genszero$ and $g\in \Adj(w)$, we have 
$hg\in \Adj(w)$, and this action 
of $\Genszero$ on $\Adj(w)$ by the multiplication from the left is free and simply transitive.
Note that $\Genszero$ and these $\Adj(w)$ are pairwise disjoint.
We put
\[
\GensA:=\bigsqcup_{w\,:\,\textrm{inner}} \Adj(w)
\quand 
\Gens:=\Genszero\sqcup \GensA.
\]
Since $D_0$ has exactly $56$ inner walls,
we have $|\Gens|=|\Genszero|+56\times |\Genszero|=912$.
\begin{lemma}\label{lem:closed}
The subset $\GensA$ of $\Aut(X)$ is closed under the operation $g\mapsto g\inv$.
Hence so is $\Gens=\Genszero\sqcup \GensA$.
\end{lemma}
\begin{proof}
Suppose that $g\in \Adj(w)$, 
where $w$ is an inner wall of $D_0$.
Then $D_0$ and $D_0^g$ are adjacent across $w$.
Hence $D_0^{(g\inv)}$ and $D_0$ are adjacent across $w^{(g\inv)}$.
Since $D_0^{(g\inv)}\subset \NX$, 
the wall $w^{(g\inv)}$ of $D_0$ is inner, and
we have $g\inv\in \Adj(w^{(g\inv)})$.
\end{proof}
We consider $\Gens$ as an alphabet, 
and denote by $\Gensst$ the set of finite sequences of elements of $\Gens$.
An element of $\Gensst$ is called a \emph{word}.
Note that 
the empty sequence $\empseq:=[\;]$ is also a word.
The conjunction of two words $\word{u}$ and $\word{v}$ is denoted by $\word{u}\cdot \word{v}$
or by $\word{u}\word{v}$.
We have seen in Proposition~\ref{prop:fromfact3} that the multiplication map
\[
\mult\;\;\colon\;\; \Gensst\to \Aut(X)
\]
given by $[\gamma_1, \dots, \gamma_n]\mapsto \gamma_1\cdots\gamma_n$ is surjective.
%
\begin{definition}\label{def:RRR}
A pair $\{\word{w}, \word{w}\sprime\}$ of words is called a \emph{relation}
if $\mult(\word{w})=\mult( \word{w}\sprime)$ holds.
Let $\RRR$ be a set of relations.
The \emph{$\RRR$-equivalence relation} is 
the minimal equivalence relation on $\Gensst$
that satisfies the following:
if two words $\word{u}$ and $\word{v}$ 
have decompositions
$\word{u}=\word{a}\cdot \word{w}\cdot \word{b}$ 
and 
$\word{v}=\word{a}\cdot \word{w}\sprime\cdot \word{b}$ 
with $\{\word{w}, \word{w}\sprime\}\in \RRR$,
then $\word{u}$ and $\word{v}$ are $\RRR$-equivalent.
\end{definition}
Note that, for any set of relations $\RRR$, 
if two words $\word{u}$ and $\word{v}$ are $\RRR$-equivalent,
then we have $\mult(\word{u})=\mult(\word{v})$.
\begin{definition}\label{def:RRR2}
We say that a set of relations $\RRR$ is a \emph{set of defining relations}
if every word in the fiber 
\[
\KKK:=\mult\inv(1)
\]
of the map $\mult$ over $1\in \Aut(X)$ 
is $\RRR$-equivalent to the empty word $\empseq$.
\end{definition}
Our goal is to exhibit a finite  set of defining relations.
\par
Let $\Reltriv$ be the set of relations consisting of the following pairs of words:
\begin{eqnarray*}
&&
\{[1], \empseq\},\\
&&
\hbox to 2cm {$\{[\gamma, \gamma\inv], [1]\}$}\quad (\gamma\in \Gens),\\
&&
\hbox to 2cm {$\{[h, h\sprime], [hh\sprime]\}$}\quad (h, h\sprime \in \Genszero),\\
&&
\hbox to 2cm {$\{[h, g], [hg]\}$}\quad(h \in \Genszero, \,g\in \GensA).
\end{eqnarray*}
Here, in the pair $\{[h, h\sprime], [hh\sprime]\} $, 
the word $[h, h\sprime]$ is of length $2$, whereas $[hh\sprime]$ is the word consisting of a single letter $hh\sprime\in \Genszero$.
The same remark is applied to the pair $\{[h, g], [hg]\}$.
\par
A word $\word{u}$ is said to be \emph{of $\gh$-form} if it is of the form
\[
[g_N, \dots, g_1, h] \qquad(g_N, \dots, g_1\in \GensA, \;\; h\in \Genszero).
\]
We allow $N$ to be $0$, so that $[h]$ is (and hence $[1]$ is) of $\gh$-form for any $h\in \Genszero$.
It is easy to see that every word is $\Reltriv$-equivalent to a word of $\gh$-form.
For example,
for $g_1, g_2\in \GensA$ and $h\in \Genszero$,
the word $[g_1, h, g_2]$ is $\Reltriv$-equivalent to the word $[g_1, hg_2, 1]$,
which is of $\gh$-form.
\par
Let $N$ be a non-negative integer.
A \emph{chamber path} of length $N$ is a sequence 
\[
\blambda:=(D\spar{N}, \dots, D\spar{0})
\]
of $\Lts/\SX$-chambers $D\spar{k}$ such that 
\begin{enumerate}[(i)]
\item each $D\spar{k}$ is contained in $\NX$,  and
\item  $D\spar{k}$ and $D\spar{k-1}$ are distinct and adjacent
for $k=1, \dots, N$.
\end{enumerate}
A chamber path is read from right to left, so that 
the chamber path $\blambda$ above is \emph{from $D\spar{0}$ to $D\spar{N}$}.
Let 
\[
\blambda\sprime:=(D\sp{\prime (N\sprime)}, \dots, D\sp{\prime(0)})
\]
be a chamber path of length $N\sprime$ such that $D\sp{\prime(0)}=D\spar{N}$.
Then the conjunction 
\[
\blambda\sprime\cdot \blambda:=(D\sp{\prime (N\sprime)}, \dots, D\sp{\prime(0)}, D\spar{N-1}, \dots, D\spar{0})
\]
is defined and is a chamber path of length $N\sprime+N$.
A \emph{chamber loop} is a chamber path $(D\spar{N}, \dots, D\spar{0})$
such that $D\spar{N}=D\spar{0}$.
In this case, we say that $D\spar{0}$ is the \emph{base point} of the chamber loop.
\par
Let $\word{u}=[g_N, \dots, g_1, h]$ be a word 
 of $\gh$-form.
Then we have a chamber path 
\[
\lambda(\word{u})=(D\spar{N}, \dots, D\spar{0})
\]
from $D\spar{0}=D_0$ to $D\spar{N}=D_0^{\mult(\word{u})}$ 
defined by
\[
D\spar{0}:=D_0^h, \;\; D\spar{1}:=D_0^{g_1h}, 
\quad\dots \quad 
D\spar{k}:=D_0^{g_k\cdots g_1h}, 
\quad\dots \quad 
D\spar{N}:=D_0^{g_N\cdots g_1h}.
\]
We call $\lambda(\word{u})$ the \emph{chamber path associated with $\word{u}$}.
If $\word{u}\in \KKK=\mult\inv(1)$, then 
$\lambda(\word{u})$ is a chamber loop with the base point $D_0$.
\par
Conversely, let $\blambda=(D\spar{N}, \dots, D\spar{0})$ be a chamber path of length $N$
starting from $D\spar{0}=D_0$.
We define
\[
\WWW(\blambda):=\set{\word{u}\in \Gensst}{\textrm{$\word{u}$ is of $\gh$-form such that $\lambda(\word{u})=\blambda$}}.
\]
A word $\word{u}=[g_N, \dots, g_1, h]$ of $\gh$-form is in $\WWW(\blambda)$ if and only if 
\begin{equation}\label{eq:Dspark}
D\spar{k}=D_0^{g_k \cdots g_1 h}
\end{equation}
holds for $k=0, \dots, N$.
Here we set $g_0=h$.
The elements of $\WWW(\blambda)$ 
can be enumerated by the following method.
First choose $g_0=h$ arbitrarily from $\Genszero$.
Suppose that $g_m, \dots, g_1, h$ have been obtained such that~\eqref{eq:Dspark} holds
for $k=0, \dots, m$.
Let $w\spar{m}$ be the wall between $D\spar{m}$ and $D\spar{m+1}$.
Then 
\[
w_m:=(w\spar{m})^{(g_m \cdots g_1 h)\inv}
\]
 is an inner wall of $D_0$.
We choose $g_{m+1}$ from $\Adj(w_{m})$ arbitrarily,
and append it to the beginning of the sequence $g_m, \dots, g_1, h$.
By iterating this process 
until we reach $m+1=N$,
we obtain a word in $\WWW(\blambda)$.
Repeating this process for all possible choices of $h\in \Genszero$ and $g_{m+1}\in \Adj(w_{m})$,
we obtain all words in $\WWW(\blambda)$.
Therefore
the size of the set $\WWW(\blambda)$ is equal to $|\Genszero|^{N+1}$.
\par
Now suppose that $\blambda$ is a chamber loop with the base point $D_0$.
Then, for any $\word{u}\in \WWW(\blambda)$, we have $\mult(\word{u})\in \Genszero$,
and the map $\mult$ induces a surjection from $\WWW(\blambda)$ onto $ \Genszero$.
We define 
\[
\WK(\blambda):=\WWW(\blambda)\cap \KKK.
\]
The size of the set $\WK(\blambda)$ is equal to $|\Genszero|^{N}$.
In particular, if $N=0$, then $\WK(\blambda)$ is equal to $\{[1]\}$.
\begin{remark}\label{rem:getaword}
Suppose that $g\in \Aut(X)$ is given.
Then a word $\word{u}\in \Gensst$ 
satisfying $\mult(\word{u})=g$ can be obtained by means of the following method.
We choose a chamber path $\blambda=(D\spar{N}, \dots, D\spar{0})$
from $D\spar{0}=D_0$ to $D\spar{N}=D_0^g$,
and compute an element 
\[
\word{u}\sprime=[g_N, \dots, g_1, 1]
\]
 of $ \WWW(\blambda)$ using the method above.
Since $D_0^g=D\spar{N}=D_0^{\mult(\word{u}\sprime)}$, 
there exists an element $h\in \Genszero$ such that $g= h\cdot \mult(\word{u}\sprime)$.
Then the word $\word{u}:=[h, g_N, \dots, g_1]$ satisfies $\mult(\word{u})=g$.
\end{remark}
Let $D\spar{0}$ be an $\Lts/\SX$-chamber contained in $\NX$, and 
let $f$ be an inner face of $D\spar{0}$ of codimension $2$.
Recall that $\DDD(f)$ is the set of $\Lts/\SX$-chambers $D$ such that  $D\supset f$.
We have $D\spar{0}\in \DDD(f)$.
Since $f$ is inner, we have $\DDD(f)\subset\VX$.
Then we have two chamber loops $\blambda(f)^{+}$ and $\blambda(f)^{-}$ with the base point $D\spar{0}$
such that 
\begin{enumerate}[(i)]
\item every chamber in the loop belongs to $\DDD(f)$,  and 
\item each element of $\DDD(f)\setminus \{D\spar{0}\}$ appears in the loop exactly once.
\end{enumerate}
See Figure~\ref{fig:simpleloops}.
These two loops differ only in the direction to which the loop goes around $f$.
We call these loops the \emph{simple chamber loops around $f$ with the base point $D\spar{0}$}.
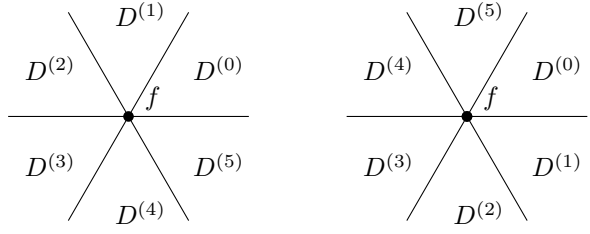
\begin{figure} 
\begin{tikzpicture}[x=.8cm,y=.8cm]
\begin{scope}[xshift=0cm, yshift=0cm]
\coordinate (P0) at (2,0);
\coordinate (P1) at (1,1.7321);
\coordinate (P2) at (-1,1.7321);
\coordinate (P3) at (-2,0);
\coordinate (P4) at (-1,-1.7321);
\coordinate (P5) at (1,-1.7321);
\draw [thin] (P0)--(P3);
\draw [thin] (P1)--(P4);
\draw [thin] (P2)--(P5);
\fill (0,0) circle (2pt);
\node at (.4, .3) {$f$};
\node at (1.5, .8) {$D\spar{0}$};
\node at (.2, 1.7) {$D\spar{1}$};
\node at (-1.3, .8) {$D\spar{2}$};
\node at (-1.3, -.8) {$D\spar{3}$};
\node at (.2, -1.6) {$D\spar{4}$};
\node at (1.5, -.8) {$D\spar{5}$};
\end{scope}
\begin{scope}[xshift=4.5cm, yshift=0cm]
\coordinate (P0) at (2,0);
\coordinate (P1) at (1,1.7321);
\coordinate (P2) at (-1,1.7321);
\coordinate (P3) at (-2,0);
\coordinate (P4) at (-1,-1.7321);
\coordinate (P5) at (1,-1.7321);
\draw [thin] (P0)--(P3);
\draw [thin] (P1)--(P4);
\draw [thin] (P2)--(P5);
\fill (0,0) circle (2pt);
\fill (0,0) circle (2pt);
\node at (.4, .3) {$f$};
\node at (1.5, .8) {$D\spar{0}$};
\node at (.2, 1.7) {$D\spar{5}$};
\node at (-1.3, .8) {$D\spar{4}$};
\node at (-1.3, -.8) {$D\spar{3}$};
\node at (.2, -1.6) {$D\spar{2}$};
\node at (1.5, -.8) {$D\spar{1}$};
\end{scope}
\end{tikzpicture}
\caption{$\blambda^+(f)$ and $\blambda^-(f)$}\label{fig:simpleloops}
\end{figure}
\par
Suppose that $f_0$ is an inner face of $D_0$ of codimension $2$.
In other words, $f_0$ is of type $4_{40}$, $6_{60}$, $8_{80a}$, or $8_{80b}$ (see~Figure~\ref{fig:Faces2}).
Let $\blambda(f_0)^{+}$ and $\blambda(f_0)^{-}$ be the 
simple chamber loops around $f_0$ with the base point $D_0$.
Then, for any word $\word{u}$ belonging to $\WWW(\blambda(f_0)^{+})$ or $ \WWW(\blambda(f_0)^{-})$, 
we have $\mult(\word{u})\in \Genszero$.
We define a set of relations $\Relface$ as 
\[
\Relface:=\bigcup_{f_0} \;\set{\; \{\word{u}, [\mult(\word{u})]\}\;}{\word{u}\in \WWW(\blambda(f_0)^{+})\cup \WWW(\blambda(f_0)^{-})},
\]
where $f_0$ ranges over the set of inner faces of $D_0$ of codimension $2$.
\begin{theorem}\label{thm:relations}
The set $\Reltriv\cup \Relface$ 
is a set of defining relations of $\Aut(X)$ 
with respect to the generating set $\Gens=\Genszero\sqcup \GensA$.
\end{theorem}
To prove this, we introduce additional definitions and propositions.
Let 
\begin{equation}\label{eq:blambda}
\blambda=(D\spar{N}, \dots, D\spar{0})\quad\textrm{with}\quad D\spar{N}=D\spar{0}=D_0
\end{equation}
be a chamber loop with the base point $D_0$.
We say that $\blambda$ is \emph{reduced to a chamber loop $\blambda\sprime$ 
by a type-$\typeI$-move} and write $\blambda\Rightarrow_{\typeI} \blambda\sprime$ if there exists 
a subsequence $(D\spar{k+1}, D\spar{k}, D\spar{k-1})$ in $\blambda$ such that $D\spar{k+1}=D\spar{k-1}$
and that $\blambda\sprime$ is obtained from $\blambda$ by removing the two chambers $D\spar{k+1}$ and $D\spar{k}$.
See Figure~\ref{fig:typeI}.
\begin{figure}
{\small
\begin{tikzpicture}[x=.7cm,y=.45cm]
\begin{scope}[xshift=0cm, yshift=0cm]
\draw [thick, <-] (-6,0)--(-5,0);
\draw [thick, <-] (-3,0)--(-2,0);
\draw [thick, <-] (2,0)--(3,0);
\draw [thick, <-] (5,0)--(6,0);
\draw [thick, <-] (-.3,1)--(-.3,2);
\draw [thick, <-] (.3,2)--(.3,1);
\node at (7, 0) {$\cdots$};
\node at (4, 0) {$D\spar{k-2}$};
\node at (-4, 0) {$D\spar{k+2}$};
\node at (-7, 0) {$\cdots$};
\node at (.2, 3) {$D\spar{k}$};
\node at (0, 0) {$D\spar{k+1}=D\spar{k-1}$};
\node at (-10, 0) {$\blambda$};
\end{scope}
\begin{scope}[xshift=0cm, yshift=-1.3cm]
\draw [thick, <-] (-5,0)--(-4,0);
\draw [thick, <-] (-2,0)--(-1,0);
\draw [thick, <-] (1,0)--(2,0);
\draw [thick, <-] (4,0)--(5,0);
\node at (-6, 0) {$\cdots$};
\node at (-3, 0) {$D\spar{k+2}$};
\node at (3, 0) {$D\spar{k-2}$};
\node at (6, 0) {$\cdots$};
\node at (0, 0) {$D\spar{k-1}$};
\node at (-10,0) {$\blambda\sprime$};
\end{scope}
\end{tikzpicture}
}
\caption{$\blambda\Rightarrow_{\typeI} \blambda\sprime$}\label{fig:typeI}
\end{figure}
We say that two chamber loops $\blambda, \blambda\sprime$ with the base point $D_0$
are \emph{connected by a type-$\typeI$-move}
if either $\blambda \Rightarrow_{\typeI}\blambda\sprime$ 
or $\blambda\sprime \Rightarrow_{\typeI}\blambda$.
\begin{proposition}\label{prop:typeImove}
Suppose that chamber loops $\blambda$ and $\blambda\sprime$ with the base point $D_0$
are connected by a type-$\typeI$-move.
Then, for each word $\word{u}\in \WK(\blambda)$,
there exists a word $\word{u}\sprime\in \WK(\blambda\sprime)$
that is $\Reltriv$-equivalent to $\word{u}$.
\end{proposition}
\begin{proof}
Let $\blambda$ be as in~\eqref{eq:blambda},
and we put $\word{u}=[g_N, \dots, g_1, h]$ with $g_0=h$.
\par
Suppose that 
$\blambda \Rightarrow_{\typeI}\blambda\sprime$ as is shown in Figure~\ref{fig:typeI}.
Then $D\spar{k+1}=D\spar{k-1}$ implies that $g_{k+1} g_k\in \Genszero$.
We put $h\sprime:=g_{k+1} g_k$,
and let $\word{u}\sprime$ be a word obtained from $\word{u}$ by removing the two letters $g_{k+1}, g_{k}$
and replacing $g_{k-1}$ with $h\sprime g_{k-1}$.
Then we see that $\word{u}\sprime$ is $\Reltriv$-equivalent to $\word{u}$,
using the relation $\{[h\sprime, g_{k}\inv ], [g_{k+1}]\}$.
It is easy to see that $\word{u}\sprime$ belongs to $ \WK(\blambda\sprime)$.
\par
Conversely, suppose that $\blambda\sprime \Rightarrow_{\typeI}\blambda$.
We assume that $\blambda\sprime$ is obtained from $\blambda$ by
putting $D\spprime, D\sprime$ on the left of a chamber $D\spar{k}$ in $\blambda$, 
where $ D\spprime=D\spar{k}$ and $D\sprime$ is adjacent to $D\spar{k}=D\spprime$.
Let $w\sprime$ be the wall between $D\spar{k}$ and $D\sprime$, 
and define 
\[
\gamma_k:=(g_k \dots g_1 h)\inv.
\]
Then $(w\sprime)^{\gamma_k}$ is an inner wall of $D_0=(D\spar{k})^{\gamma_k}$.
We choose an arbitrary element $g\sprime$ from $\Adj((w\sprime)^{\gamma_k})$.
Then we have $D\sprime=D_0\sp{g\sprime g_k \dots g_1 h}$.
We make a word $\word{u}\sprime$ by 
putting $g\sprimeinv, g\sprime$ on the left of the letter $g_k$ in $\word{u}$. 
Then $\word{u}\sprime$ is $\Reltriv$-equivalent to $\word{u}$, 
and $\word{u}\sprime$ belongs to $ \WK(\blambda\sprime)$.
\end{proof}
We say that a chamber loop $\blambda$ as in~\eqref{eq:blambda} 
is \emph{reduced to a chamber loop $\blambda\sprime$ 
by a type-$\typeII$-move} and 
write $\blambda\Rightarrow_{\typeII} \blambda\sprime$ if there exists a
subsequence 
\[
\brho=(D\spar{m}, \dots, D\spar{k})\quad \textrm{with}\quad m> k 
\]
 in $\blambda$ such that 
 \begin{enumerate}[(i)]
 \item $D\spar{m}=D\spar{k}$,
\item $\brho$ is a simple chamber loop with the base point $D\spar{k}$
 around an inner face $f$ of $D\spar{k}$ 
 of codimension $2$, and 
 \item $\blambda\sprime$ is obtained from $\blambda$ by removing the chambers $D\spar{m-1}, \dots, D\spar{k}$.
  \end{enumerate}
See Figure~\ref{fig:typeII}.
\begin{figure}
\phantom{aa}
\vskip .5cm
{\small
\begin{tikzpicture}[x=.7cm,y=.45cm]
\begin{scope}[xshift=0cm, yshift=0cm]
\draw [thick, <-] (-6,0)--(-5,0);
\draw [thick, <-] (-3,0)--(-2,0);
\draw [thick, <-] (2,0)--(3,0);
\draw [thick, <-] (5,0)--(6,0);
\draw [thick, ->] (.3,.6)--(1.3,1.6);
\draw [thick, ->] (2.0,3)--(2.0,4.3);
\draw [thick, ->] (1.5,5)--(0.5,6);
\draw [thick, ->] (-0.5,6)--(-1.5,5);
\draw [thick, ->] (-2.0,4.3)--(-2.0,3);
\draw [thick, ->] (-1.5, 1.7)--(-.5,.7);
\node at (7, 0) {$\cdots$};
\node at (4, 0) {$D\spar{k-1}$};
\node at (2.2, 2.2) {$D\spar{k+1}$};
\node at (-2.0, 2.2) {$D\spar{m-1}$};
\node at (-4, 0) {$D\spar{m+1}$};
\node at (-7, 0) {$\cdots$};
\fill (0,3.3) circle (2pt);
\node at (.5, 3.3) {$f$};
\node at (0, 0) {$D\spar{m}=D\spar{k}$};
\node at (-10, 0) {$\blambda$};
\end{scope}
\begin{scope}[xshift=0cm, yshift=-1.5cm]
\draw [thick, <-] (-5,0)--(-4,0);
\draw [thick, <-] (-2,0)--(-1,0);
\draw [thick, <-] (1,0)--(2,0);
\draw [thick, <-] (4,0)--(5,0);
\node at (-6, 0) {$\cdots$};
\node at (-3, 0) {$D\spar{m+1}$};
\node at (3, 0) {$D\spar{k-1}$};
\node at (6, 0) {$\cdots$};
\node at (0, 0) {$D\spar{m}$};
\node at (-10,0) {$\blambda\sprime$};
\end{scope}
\end{tikzpicture}
}
\caption{$\blambda\Rightarrow_{\typeII} \blambda\sprime$}\label{fig:typeII}
\end{figure}
We say that two chamber loops $\blambda, \blambda\sprime$ with the base point $D_0$
are \emph{connected by a type-$\typeII$-move}
if either $\blambda \Rightarrow_{\typeII}\blambda\sprime$ 
or $\blambda\sprime \Rightarrow_{\typeII}\blambda$.
\begin{proposition}\label{prop:typeIImove}
Suppose that chamber loops $\blambda$ and $\blambda\sprime$ with the base point $D_0$
are connected by a type-$\typeII$-move.
Then, for each word $\word{u}\in \WK(\blambda)$,
there exists a word $\word{u}\sprime\in \WK(\blambda\sprime)$
that is $(\Reltriv\cup \Relface)$-equivalent to $\word{u}$.
\end{proposition}
\begin{proof}
Let $\blambda$ be as in~\eqref{eq:blambda},
and we put $\word{u}=[g_N, \dots, g_1, h]$ with $g_0=h$.
\par
Suppose that 
$\blambda \Rightarrow_{\typeII}\blambda\sprime$ as is shown in Figure~\ref{fig:typeII}.
Then $D\spar{m}=D\spar{k}$ implies 
\[
h\sprime:=g_{m} \cdots g_{k+1}\in \Genszero.
\]
We define 
\[
\gamma_k:=(g_k \dots g_1 h)\inv.
\]
Then $f^{\gamma_k}$ is an inner face of $D_0$,
and the simple chamber loop $\brho=(D\spar{m}, \dots, D\spar{k})$ around $f$ is mapped by $\gamma_k$
to a simple chamber loop $\brho^{\gamma_k}$ 
around $f^{\gamma_k}$ with the base point $D_0$.
Moreover,  the word $[g_m, \dots, g_{k+1}, 1]$ of $\gh$-form is an element of $\WWW(\brho^{\gamma_k})$.
In particular, we have
\[
\{\; [g_m, \dots, g_{k+1}, 1], [h\sprime]\;\}\;\;\in\;\; \Relface.
\]
Let $\word{u}\sprime$ be a word obtained from $\word{u}$ by removing the letters 
$g_{m}, \dots, g_{k+1} $
and replacing $g_{k}$ by $h\sprime g_{k}$.
Then $\word{u}\sprime$ is $(\Reltriv\cup \Relface)$-equivalent to $\word{u}$, 
and we have $\word{u}\sprime\in \WK(\blambda\sprime)$.
\par
Conversely, suppose that 
$\blambda\sprime \Rightarrow_{\typeII}\blambda$.
We assume that $\blambda\sprime$ is obtained from $\blambda$ by putting 
a sequence $D\sp{\prime (n)}, \dots, D\sp{\prime (1)}$ on the left of a chamber $D\spar{k}$ in $\blambda$,
where 
\[
\brho\sprime=(D\sp{\prime (n)}, D\sp{\prime (n-1)}, \dots, D\sp{\prime (1)}, D\sp{\prime (0)})\quad \textrm{with} \quad 
D\sp{\prime (n)}=D\spar{k} \; \;\textrm{and} \; \;D\sp{\prime (0)}:=D\spar{k}
\]
is a simple chamber loop
around an inner face $f$ of $D\spar{k}$.
Again, we put
$\gamma_k:=(g_k \dots g_1 h)\inv$.
Then $f^{\gamma_k}$ is an inner face of $D_0=(D\spar{k})^{\gamma_k}$,
and $\gamma_k$ maps $\brho\sprime$ to a simple chamber loop $\brho^{\prime\gamma_k}$ 
around the inner face $f^{\gamma_k}$ of $D_0$.
Then $\WWW(\brho^{\prime\gamma_k})$ contains a word of the form 
\[
\word{v}:=[g_n\sprime, \dots, g_1\sprime, 1].
\]
We have $m(\word{v})\in \Genszero$.
Since $n>0$, by replacing $g_n\sprime$ with $\mult(\word{v})\inv g_n\sprime$ if necessary,
we can assume that 
\[
g_n\sprime \cdot \cdots \cdot g_1\sprime=1,
\]
and we have $\{\word{v}, [1]\} \in \Relface$.
We make a word $\word{u}\sprime$ from $\word{u}$ by 
putting $g_n\sprime, \dots, g_1\sprime$ on the left of the letter $g_{k}$ in $\word{u}$.
Then $\word{u}\sprime$ is $(\Reltriv\cup \Relface)$-equivalent to $\word{u}$, and $\word{u}\sprime$ belongs to $ \WK(\blambda\sprime)$.
\end{proof}
\begin{proof}[Proof of Theorem~\ref{thm:relations}]
Let $\word{u}$ be a word in $\KKK$.
We show that $\word{u}$ is $(\Reltriv\cup \Relface)$-equivalent to an empty word $\empseq=[\;]$. 
Since every word is $\Reltriv$-equivalent to a word of $\gh$-form,
we can assume that $\word{u}$ is of $\gh$-form.
Let $\blambda_0:=\lambda(\word{u})$ be the chamber loop 
with the base point $D_0$ associated with $\word{u}$.
Since the nef-and-big cone $\NX$ is simply connected,
there exists a sequence
\[
\blambda_0, \;\; \blambda_1, \;\; \dots, \;\; \blambda_n=(D_0)
\]
of chamber loops
with the base point $D_0$ such that,
for $i=1, \dots, n$,
the two loops $\blambda_{i-1}$ and $\blambda_{i}$ are connected 
by either a type-$\typeI$-move
or a type-$\typeII$-move,
and that the last chamber loop $\blambda_n$ is the loop $(D_0)$ of length $0$.
Since $\word{u}\in \WK(\blambda_0)$ and $\WK(\blambda_n)=\{[1]\}$,
Propositions~\ref{prop:typeImove} and~\ref{prop:typeIImove} imply that $\word{u}$
is $(\Reltriv\cup \Relface)$-equivalent to $[1]$,
and hence to $\empseq$.
\end{proof}
\bibliographystyle{plain}

\begin{thebibliography}{10}

\bibitem{AhlgrenOno2000}
Scott Ahlgren and Ken Ono.
\newblock Modularity of a certain {C}alabi-{Y}au threefold.
\newblock {\em Monatsh. Math.}, 129(3):177--190, 2000.

\bibitem{Apery1979}
Roger Ap\'ery.
\newblock Irrationalit\'e{} de {$\zeta (2)$} et {$\zeta (3)$}.
\newblock {\em Ast\'erisque}, (61):11--13, 1979.
\newblock Luminy Conference on Arithmetic.

\bibitem{BertinLecacheux2013}
Marie~Jos\'{e} Bertin and Odile Lecacheux.
\newblock Elliptic fibrations on the modular surface associated to
  {$\Gamma_1(8)$}.
\newblock In {\em Arithmetic and geometry of {K}3 surfaces and {C}alabi-{Y}au
  threefolds}, volume~67 of {\em Fields Inst. Commun.}, pages 153--199.
  Springer, New York, 2013.

\bibitem{BertinLecacheux2020}
Marie~Jos\'{e} Bertin and Odile Lecacheux.
\newblock Ap\'{e}ry-{F}ermi pencil of {$K3$}-surfaces and 2-isogenies.
\newblock {\em J. Math. Soc. Japan}, 72(2):599--637, 2020.

\bibitem{BertinLecacheux2022}
Marie~Jos\'{e} Bertin and Odile Lecacheux.
\newblock Elliptic fibrations of a certain {$K3$} surface of the
  {A}p\'{e}ry-{F}ermi pencil.
\newblock In {\em Publications math\'{e}matiques de {B}esan\c{c}on. {A}lg\`ebre
  et th\'{e}orie des nombres. 2022}, volume 2022 of {\em Publ. Math.
  Besan\c{c}on Alg\`ebre Th\'{e}orie Nr.}, pages 5--36. Presses Univ.
  Franche-Comt\'{e}, Besan\c{c}on, 2022.

\bibitem{BeukersPeters1984}
Frits Beukers and Chris  Peters.
\newblock A family of {$K3$} surfaces and {$\zeta (3)$}.
\newblock {\em J. Reine Angew. Math.}, 351:42--54, 1984.

\bibitem{Bor1}
Richard Borcherds.
\newblock Automorphism groups of {L}orentzian lattices.
\newblock {\em J. Algebra}, 111(1):133--153, 1987.

\bibitem{Bor2}
Richard Borcherds.
\newblock Coxeter groups, {L}orentzian lattices, and {$K3$} surfaces.
\newblock {\em Internat. Math. Res. Notices}, 1998(19):1011--1031, 1998.

\bibitem{BrandhorstShimada2022}
Simon Brandhorst and Ichiro Shimada.
\newblock Automorphism groups of certain {E}nriques surfaces.
\newblock {\em Found. Comput. Math.}, 22(5):1463--1512, 2022.

\bibitem{Conway1983}
John H. Conway.
\newblock The automorphism group of the {$26$}-dimensional even unimodular
  {L}orentzian lattice.
\newblock {\em J. Algebra}, 80(1):159--163, 1983.

\bibitem{DardanelliGeemen2007}
Elisa Dardanelli and Bert van Geemen.
\newblock Hessians and the moduli space of cubic surfaces.
\newblock In {\em Algebraic geometry}, volume 422 of {\em Contemp. Math.},
  pages 17--36. Amer. Math. Soc., Providence, RI, 2007.

\bibitem{DolgachevLatticePolarizedK3}
Igor Dolgachev.
\newblock Mirror symmetry for lattice polarized {$K3$} surfaces.
\newblock {\em J. Math. Sci.}, 81(3):2599--2630, 1996.
\newblock Algebraic geometry, 4.



\bibitem{DolgachevShimada2020}
Igor Dolgachev and Ichiro Shimada.
\newblock 15-nodal quartic surfaces. {P}art {II}: the automorphism group.
\newblock {\em Rend. Circ. Mat. Palermo (2)}, 69(3):1165--1191, 2020.

\bibitem{DinoDuco2019}
Dino Festi and Duco van Straten.
\newblock Bhabha scattering and a special pencil of {K}3 surfaces.
\newblock {\em Commun. Number Theory Phys.}, 13(2):463--485, 2019.

\bibitem{HLOY2004}
Shinobu Hosono, Bong~H. Lian, Keiji Oguiso, and Shing-Tung Yau.
\newblock Autoequivalences of derived category of a {$K3$} surface and
  monodromy transformations.
\newblock {\em J. Algebraic Geom.}, 13(3):513--545, 2004.

\bibitem{Kondo1998}
Shigeyuki Kondo.
\newblock The automorphism group of a generic {J}acobian {K}ummer surface.
\newblock {\em J. Algebraic Geom.}, 7(3):589--609, 1998.

\bibitem{MukaiOhashi2015}
Shigeru Mukai and Hisanori Ohashi.
\newblock The automorphism groups of {E}nriques surfaces covered by symmetric
  quartic surfaces.
\newblock In {\em Recent advances in algebraic geometry}, volume 417 of {\em
  London Math. Soc. Lecture Note Ser.}, pages 307--320. Cambridge Univ. Press,
  Cambridge, 2015.

\bibitem{theNikulin}
Vyacheslav V. Nikulin.
\newblock Integer symmetric bilinear forms and some of their geometric
  applications.
\newblock {\em Izv. Akad. Nauk SSSR Ser. Mat.}, 43(1):111--177, 238, 1979.
\newblock English translation: Math USSR-Izv. 14 (1979), no. 1, 103--167
  (1980).

\bibitem{OSCAR_K3Surfaces}
{OSCAR Project}.
\newblock {K3 surfaces -- borcherds method. }\hfill\\
\newblock
  \url{https://docs.oscar-system.org/dev/AlgebraicGeometry/Surfaces/K3Surfaces/#borcherds_method},
  2025.
\newblock Accessed: 2025-04-21.

\bibitem{Peters1986}
Chris Peters.
\newblock Monodromy and {P}icard-{F}uchs equations for families of
  {$K3$}-surfaces and elliptic curves.
\newblock {\em Ann. Sci. \'{E}cole Norm. Sup. (4)}, 19(4):583--607, 1986.

\bibitem{PetersStienstra1989}
Chris Peters and Jan Stienstra.
\newblock A pencil of {$K3$}-surfaces related to {A}p\'{e}ry's recurrence for
  {$\zeta(3)$} and {F}ermi surfaces for potential zero.
\newblock In {\em Arithmetic of complex manifolds ({E}rlangen, 1988)}, volume
  1399 of {\em Lecture Notes in Math.}, pages 110--127. Springer, Berlin, 1989.

\bibitem{Shimada2014}
Ichiro Shimada.
\newblock Projective models of the supersingular {$K3$} surface with {A}rtin
  invariant 1 in characteristic 5.
\newblock {\em J. Algebra}, 403:273--299, 2014.

\bibitem{Shimada2015}
Ichiro Shimada.
\newblock An algorithm to compute automorphism groups of {$K3$} surfaces and an
  application to singular {$K3$} surfaces.
\newblock {\em  Internat. Math. Res. Notices}, (22):11961--12014, 2015.

\bibitem{Shimada2016}
Ichiro Shimada.
\newblock The automorphism groups of certain singular {$K3$} surfaces and an
  {E}nriques surface.
\newblock In {\em K3 surfaces and their moduli}, volume 315 of {\em Progr.
  Math.}, pages 297--343. Birkh\"auser/Springer, [Cham], 2016.
  
\bibitem{Shimada2024}
Ichiro Shimada.
\newblock Mordell--{W}eil groups and automorphism groups of elliptic {$K3$}
  surfaces.
\newblock {\em Rev. Mat. Iberoam.}, 40(4):1469--1503, 2024.

\bibitem{AFcompdata}
Ichiro Shimada.
\newblock The automorphism group of an {A}p\'ery-{F}ermi {$K3$} surface:
  {C}omputational data.
\newblock 2025.
\newblock \url{https://doi.org/10.5281/zenodo.14842583},
  \url{https://home.hiroshima-u.ac.jp/ichiro-shimada/ComputationData.html}.

\bibitem{GAP}
{The GAP Group}.
\newblock {\em {\tt {G}{A}{P}} - {G}roups, {A}lgorithms, and {P}rogramming}.
\newblock Version 4.13.0 of 2024-03-15 (\url{http://www.gap-system.org}).

\bibitem{vanGeemenNygaard1995}
Bert van Geemen and Niels~O. Nygaard.
\newblock On the geometry and arithmetic of some {S}iegel modular threefolds.
\newblock {\em J. Number Theory}, 53(1):45--87, 1995.

\bibitem{Verrill2000}
Helena~A. Verrill.
\newblock The {$L$}-series of certain rigid {C}alabi-{Y}au threefolds.
\newblock {\em J. Number Theory}, 81(2):310--334, 2000.

\bibitem{VinbergShvartsman1993}
\`Ernest B. Vinberg and Osip~V. Shvartsman.
\newblock Discrete groups of motions of spaces of constant curvature.
\newblock In {\em Geometry, {II}}, volume~29 of {\em Encyclopaedia Math. Sci.},
  pages 139--248. Springer, Berlin, 1993.

\end{thebibliography}
\end{document}